\documentclass{article}
\usepackage{graphicx} 
\usepackage{hyperref}
\usepackage{stmaryrd}
\usepackage{framed}
\usepackage{mathtools}
\usepackage{physics}
\usepackage{pgf}
\usepackage[parfill]{parskip}
\usepackage[]{comment}
\usepackage[normalem]{ulem}

\newcommand{\ropt}{r_{\rm opt}}

\usepackage{caption}
\usepackage{tikz}
\usepackage{subcaption}
\usetikzlibrary{arrows.meta, positioning}
\usetikzlibrary{calc}
\usetikzlibrary{patterns,arrows,decorations.pathreplacing,math,decorations.markings,angles}

\usepackage{amsmath} 
\usepackage{amssymb, bm}
\usepackage{bbm}
\usepackage{authblk}
\usepackage{amsthm}
\usepackage{amstext}
\usepackage{array}
\usepackage{amsfonts}
\usepackage{algorithm}
\usepackage{algpseudocode}
\usepackage[
backend=biber,
style=numeric,
sorting=none
]{biblatex}
\newtheorem{theorem}{Theorem}[section]

\newtheorem{theorem2}{Theorem}[section]

\newtheorem{remark}[theorem2]{Remark}

\DeclareMathOperator*{\argmin}{arg\,min}

\newcommand{\xb}{\mathbf{x}}

\newcommand{\pad}[2]{\frac{\partial{#1}}{\partial{#2}}}
\newcommand{\td}[2]{\frac{d{#1}}{d{#2}}}

\usepackage{geometry}
\title{Deterministic and stochastic particle methods for the Fokker-Planck equation in S--formulation with application to point set registration}

\author[1]{Klaas Willems\thanks{Corresponding author: klaas.willems@kuleuven.be}}
\author[2]{Angelo Iollo\thanks{angelo.iollo@math.u-bordeaux.fr}}
\author[3]{Giovanni Russo\thanks{giovanni.russo1@unict.it}}
\author[4]{Tommaso Taddei\thanks{tommaso.taddei@uniroma1.it}}

\affil[1]{Department of Computer Science, KU Leuven, Leuven, Belgium}
\affil[2]{IMB, UMR 5251, Univ. Bordeaux and MONHADE Inria 33400 Talence, France}
\affil[3]{Department of Mathematics and Computer Science, University of Catania, Catania, Italy}
\affil[4]{Dipartimento di Matematica, Sapienza Universita di Roma, Rome, Italy}
\date{}

\begin{document}
	
\maketitle
\begin{abstract}
	We present two particle methods for point set registration in bounded domains based on the Fokker-Planck equation.
	The first method relies on a moving least squares discretization of the S--formulation, in which moving grid points (particles) are advected by the drift
	associated with the target distribution,
	while diffusion is resolved on a dynamically evolving particle cloud. This setting naturally leads to strong compression and expansion of the particle cloud.
	Obstacles are handled by enforcing reflective boundary conditions through a novel ghost point method.
	The second method is a Monte Carlo solver for the associated Langevin stochastic differential equation.
	It relies on a local approximation of the logarithmic gradient of the evolving particle density to extract macroscopic osmotic paths from individual stochastic trajectories.
	Owing to its inherent parallelism, this method exhibits excellent scalability and is well suited for high--dimensional registration problems.
	We illustrate the main features and performance of both approaches through extensive numerical experiments.
\end{abstract}

\textbf{Acronyms:} PSR (point set registration), MOR (model order reduction), MLS (moving least squares), FPE (Fokker-Planck equation)

\section{Introduction}
\label{section:intro}

\subsection{Point set registration in bounded domains}
The problem of finding coordinate transformations to align two or more point clouds is of broad interest in science and engineering.
In model order reduction (MOR), coordinate transformations are employed 
to align point clouds that are associated with 
coherent features of the solution field (e.g., shocks)  \cite{iollo2014advection,iollo2022mapping,razavi2025registration,taddei2020registration},
to enhance data compression.
In image processing 
\cite{balakrishnan2019voxelmorph,beg2005computing,beier2023feature},
registration techniques are employed to determine realistic interpolations of two objects,
estimate displacements between frames and combine information from different images.
In many applications, coordinate transformations must satisfy specific geometric constraints to prevent the displacement of point clouds into invalid regions; mathematically, this often requires the mapping to be a bijection from the domain of interest onto itself.

Given the two- or three-dimensional domain $\Omega \subset \mathbb{R}^d$, and the reference and the target point clouds 
$\{X_i^0\}_{i=1}^N$, 
$\{ X_j^\infty\}_{j=1}^M \subset \Omega$ and
$T_{\rm max} \in \mathbb{R}_+ \cup \{ +\infty \}$, 
we seek a flow of diffeomorphisms  $\Phi:\Omega \times [0,T_{\rm max}) \to \Omega$  such that
(i) $\Phi(\cdot, t)$ is 
equal to the identity for $t=0$ and is continuous  with respect to $t$; 
(ii)   $\Phi(\cdot, t)$ is a diffeomorphism of the domain that preserves its boundary for all $t$;
(iii)  the mapped point cloud $\{ \Phi(X_i^0,t)  \}_{i=1}^N$ is close --- in the Hausdorff distance --- to the target point cloud as $t\to T_{\rm max}$.
Note that, since the point correspondences are unknown, directly estimating both the correspondence and the transformation generally leads to a challenging combinatorial optimization problem. Several methods for point set registration (PSR) recast the deterministic problem in a probabilistic setting to avoid the need for combinatorial optimization \cite{zhu2019review}.
In more detail, 
we assume  that 
$\mathcal{X}^0 = \{X_i^0 \}_{i=1}^N$ and
$\mathcal{X}^\infty = \{X_j^\infty \}_{j=1}^M$ are independent and identically distributed (iid) samples from two probability distributions with densities $\rho_0$ and $\rho_\infty$.
Both densities are nonnegative and normalized to have unit mass.
Next, we seek a transport map $\Phi$ that interpolates the two densities, that is,
the pushforward measure $\Phi(\cdot, t)_\# \rho_0$
converges to $\rho_\infty$ as $t\to T_{\rm max}$.
Expectation-maximization procedures have been proposed to directly estimate the transformation that registers the two point clouds \cite{myronenko2010point};
see also \cite{labatut2026parametric} for an application to MOR in bounded domains.
Alternatively,  in \cite{iollo2025point}, we resorted to the Fokker-Planck equation (FPE)  to determine the transport map. The FPE  describes the evolution of a probability density $\rho(\cdot, t) = \Phi(\cdot, t)_\# \rho_0$ and 
reads as a linear parabolic equation with Neumann boundary conditions on $\partial \Omega$, constant diffusion and time-independent advection.
As discussed in Section \ref{section:formulation}, we can prove that 
$\rho(\cdot, t) $ converges to $\rho_\infty$ as $t\to T_{\rm max}=+\infty$ under mild assumptions on the reference and the target distributions.

In this work, we build upon  \cite{iollo2025point} to devise a general PSR method in bounded domains based on the FPE. The application of the FPE to PSR requires addressing three major tasks. \begin{enumerate}
	\item
	\emph{Density estimation:} given the point clouds 
	$\mathcal{X}^0 = \{X_i^0 \}_{i=1}^N$ and
	$\mathcal{X}^\infty = \{X_j^\infty \}_{j=1}^M$,  estimate the probability densities $\rho_0$ and $\rho_\infty$.
	\item
	\emph{Numerical approximation of the Fokker-Planck equation:} determine an estimate of the solution to the FPE.
	\item
	\emph{Transport of particles:}
	estimate the flow $\Phi(\cdot ,t)$ based on the estimate of the density and use it to  determine the particles' trajectories.
\end{enumerate}
In this work, we focus on the latter two tasks:
first, we propose a moving least squares (MLS) method for the FPE;
second, we present a Monte Carlo (MC) method that relies on the stochastic formulation of the FPE; finally, we discuss the sensitivity of the procedure with respect to the choice of the density.

\subsection{Contributions and relation to previous work}

The present work builds on the use of the Fokker-Planck equation as a tool for point set registration in bounded domains. As shown in~\cite{iollo2025point}, the FPE provides a natural framework for mapping probability distributions. This is particularly advantageous in bounded domains or in the presence of obstacles, where alternative approaches based on optimal transport require the solution of highly nonlinear evolution equations and the explicit treatment of constraints, which rapidly becomes prohibitive. In contrast, the FPE is a linear parabolic equation with time--independent coefficients, for which reflective boundaries are straightforward to impose. In this work, we further argue that particle methods are especially well suited for the numerical approximation of the FPE in the context of point set registration. Since the goal is to map distributions with typically local support, particle formulations concentrate the computational effort where the density is non--negligible, whereas grid--based methods must discretize the full domain. Moreover, particle methods can directly provide the trajectories of the induced transport map, which are essential for applications such as registration--based model order reduction \cite{cucchiara_model_2024}. Motivated by these observations, we develop two complementary particle methods for this purpose.

The first contribution of this paper is a deterministic particle method for the FPE based on a moving least squares (MLS) discretization \cite{lancaster_surfaces_1981, levin_approximation_1998}. Meshfree MLS techniques have been widely used for the numerical solution of partial differential equations, notably within generalized finite difference \cite{clain_stencil_2024, kuhnert_grid_2001} and moving particle semi--implicit (MPS) methods \cite{tamai_least_2014, wang_compact_2023}, due to their flexibility in handling complex geometries without mesh generation. While such methods have been successfully applied to incompressible flow problems \cite{koshizuka_numerical_1998, jandaghian_enhanced_2020}, their extension to diffusion--dominated kinetic equations such as the FPE poses additional challenges. In this work, we formulate an MLS discretization of the FPE in its S--formulation and adopt a Lagrangian viewpoint in which particles are advected by the drift associated with the target distribution, while diffusion is resolved on a dynamically evolving particle cloud. This setting gives rise to strong compression and expansion of the particle distribution and requires a careful treatment of boundary conditions, which we enforce through a novel ghost point method that consistently imposes no--flux constraints at the boundary. To the best of our knowledge, this is the first least--squares particle method specifically adapted to the FPE and employed for point set registration.

The second particle approach considered in this paper is a stochastic method that relies on the Langevin stochastic differential equation associated with the FPE
\cite{karatzas_brownian_1998}.
Monte Carlo simulation of stochastic dynamics provides a classical alternative to deterministic PDE solvers, with well--established theoretical foundations and
excellent scalability \cite{caflisch_monte_1998}.
In the present context, the stochastic formulation offers two key advantages.
First, it bypasses the explicit numerical solution of the FPE by simulating independent reflected Langevin trajectories, which naturally encode the desired boundary behavior.
Second, the method only requires explicit knowledge of the score of the target density, that is, its logarithmic gradient, while the source distribution is represented
implicitly by the initial particle ensemble.
We show how to recover deterministic osmotic paths from the stochastic trajectories, thus reconstructing the transport map required for point set registration.
To account for spatial and temporal variations in the particle density, we introduce an adaptive-radius procedure that adjusts the sampling regions used to estimate these paths.
This approach connects the present work to recent developments in score--based learning \cite{song_score-based_2020}.

The remainder of the paper is organized as follows. In Section~\ref{section:formulation}, we recall the Fokker-Planck formulation for point set registration and introduce its deterministic and stochastic representations. Section~\ref{section:MLS} presents the MLS particle method, including the spatial discretization, time integration, and boundary treatment. In Section~\ref{section:MC}, we describe the Monte Carlo method based on reflected Langevin dynamics and detail the extraction of osmotic paths from the stochastic simulations. Section~\ref{section:numerics} contains numerical experiments that assess the accuracy and robustness of both methods on a range of test cases, including registration problems with obstacles. Finally, Section~\ref{section:conclusion} summarizes the main results and outlines directions for future research.

\section{Formulation}
\label{section:formulation}
In this section, we recall the formulation of point set registration based on the
FPE and introduce the deterministic and stochastic viewpoints that
underpin the proposed methods. We first describe the deterministic Fokker-Planck
formulation and the associated transport map, and then review its stochastic counterpart
based on Langevin dynamics, which forms the basis of the Monte Carlo approach.

\subsection{Fokker-Planck equation for point set registration}
\label{sub:FPE}
As discussed in the introduction, given the probability densities $\rho_0$ and $\rho_\infty$ and the particles
$\mathcal{X}^0 = \{X_i^0 \}_{i=1}^N$ and
$\mathcal{X}^\infty = \{X_j^\infty \}_{j=1}^M$ such that
\begin{equation}
	\label{eq:iid_samples}
	X_i^0 \overset{\rm iid}{\sim} \rho_0, \quad
	X_j^\infty \overset{\rm iid}{\sim} \rho_\infty,
	\;\;
	i=1,\ldots,N; \quad
	j=1,\ldots,M,
\end{equation}
we seek a transport map $\Phi: \Omega \times \mathbb{R}_+ \to \Omega$ such that the density associated with the pushforward measure $\Phi(\cdot, t)_\# \rho_0$
\[
\rho(x,t) = \rho_0\left(\Phi^{-1}(x,t)\right)\left|\det\nabla\Phi^{-1}(x,t)\right|
\]
converges to $\rho_\infty$ as $t\to \infty$.
Below, we illustrate how the FPE can be used to achieve this goal.

We introduce the solution $\rho:\Omega \times \mathbb{R}_+ \to \mathbb{R}_+$ to the FPE,
\begin{equation}
	\label{eq:FPequation}
	\left\{
	\begin{array}{ll}
		\displaystyle{
			\partial_t \rho + \nabla \cdot \left(  \rho \nabla S_\infty - \nabla \rho 
			\right) = 0
		}     &  \displaystyle{{\rm in} \; \Omega \times \mathbb{R}_+}\\[3mm]
		\rho(\cdot,0) = \rho_0     & 
		{\rm in} \; \Omega  \\[3mm]
		\left( \rho \nabla S_\infty  -  \nabla \rho 
		\right) \cdot \mathbf{n} = 0     & 
		{\rm on} \; \partial \Omega \times \mathbb{R}_+  \\
	\end{array}
	\right.
\end{equation}
where 
$S_\infty = \log \rho_\infty$, with $\nabla S_\infty = \frac{\nabla \rho_\infty}{\rho_\infty}$ denoting the score of \(\rho_\infty\).
The score represents the spatial rate of change of the density relative to its local value.
The unit normal $\mathbf{n}$ to \(\partial \Omega\) points outward from the domain $\Omega$. Finally, we define
the velocity
\begin{equation}
	\label{eq:velocityFP}
	u(x,t) =  \nabla \left( S_\infty(x) - \log \rho (x,t) \right),
	\quad
	x \in \Omega, \;\; t\in \mathbb{R}_+;
\end{equation}
and the associated flow map
\begin{equation}
	\label{eq:FPflow}
	\Phi(\xi,t) = X(\xi,t) \quad
	{\rm where} \;\;
	\left\{
	\begin{array}{ll}
		\displaystyle{
			\partial_t  X(\xi,\tau) = u   
			\left(  
			X(\xi,\tau) ,\tau
			\right) 
		}     &  
		\tau \in (0,t],
		\\[3mm]
		X(\xi,0) = \xi. &   \\
	\end{array}
	\right.
\end{equation}
The next theorem provides a rigorous justification of the previous definitions. We refer to \cite{iollo2025point} and to the references therein for the proofs.

\begin{theorem}
	\label{th:FPE_analysis}
	Let $\Omega \subset \mathbb{R}^d$ be a Lipschitz domain; let $\rho_0,\rho_\infty \in \mathcal{P}_2(\Omega)$ satisfy
	\begin{equation}
		\label{eq:hypotheses_FPE}
		\inf_{x\in \Omega} \rho_\infty(x) > 0,
		\quad
		{\rm and} \quad
		\nabla \rho_\infty \in {\rm Lip}(\Omega, \mathbb{R}^d).
	\end{equation}
	Then, the following hold.
	\begin{enumerate}
		\item
		There exists a unique weak solution $\rho \in C(\mathbb{R}_+; \mathcal{P}_2(\Omega))$ to \eqref{eq:FPequation}.
		\item
		The solution to \eqref{eq:FPequation} converges to $\rho_\infty$ in $L^1(\Omega)$; that is,
		$\lim_{t\to \infty} \| \rho(\cdot, t) - \rho_\infty  \|_{L^1(\Omega)} = 0$.
		\item
		If   $X_1^0,\ldots,X_N^0 \overset{\rm iid}{\sim} \rho_0$, and $\Phi$ is defined as in \eqref{eq:FPflow}, we have that
		$\Phi(X_i^0,t) \overset{\rm iid}{\sim} \rho(\cdot, t)$ for all $t>0$.
	\end{enumerate}
\end{theorem}

Theorem \ref{th:FPE_analysis} provides the foundation for the use of the FPE for PSR; we can distinguish between three tasks.
\begin{enumerate}
	\item
	\emph{Density estimation:} given the point clouds 
	$\mathcal{X}^0 = \{X_i^0 \}_{i=1}^N$ and
	$\mathcal{X}^\infty = \{X_j^\infty \}_{j=1}^M$, estimate the densities $\rho_0$ and $\rho_\infty$.
	\item
	\emph{Numerical approximation of the Fokker-Planck equation:} determine an estimate of the solution to \eqref{eq:FPequation}.
	\item
	\emph{Transport of particles:} integrate the equations \eqref{eq:FPflow} to determine the particles' trajectories.
\end{enumerate}
We note that the algorithm directly exploits the density $\rho_0$ as it provides the initial condition for the FPE; 
on the other hand, we only need the score of the target density, $\nabla \log \rho_\infty $. The density estimation procedure should ensure the satisfaction of the hypotheses of Theorem \ref{th:FPE_analysis} (cf. \eqref{eq:hypotheses_FPE}). As shown in Section \ref{section:PSRAcrossCylinder}, the choice of the density has a direct impact on the stiffness of the PDE \eqref{eq:FPequation} and ultimately on the difficulty of its numerical approximation.

Following \cite{iollo2025point}, we investigate two different methods for the approximation of the particles' trajectories. We note that we rely on the hypothesis that the velocity \eqref{eq:velocityFP} is Lipschitz continuous.
As discussed in \cite{ambrosio2017lecture}, it is possible to rigorously extend the definition of the transport map $\Phi$ to abstract spaces. We also note that the particle methods
discussed in  Section \ref{section:MLSSection} couple the numerical approximation of \eqref{eq:FPequation} with the transport of particles \eqref{eq:FPflow}. 

\subsection{S-formulation of the Fokker-Planck equation}
\label{sub:S-form}
In Section \ref{sub:FPE} we introduced the FPE for the density $\rho$ and the transport velocity $u$. 
With this notation, the first equation in \eqref{eq:FPequation}
can be written as 
\[
\pad{\rho}{t} + \nabla\cdot(\rho u) = \pad{\rho}{t} + u\cdot \nabla \rho + \rho\nabla\cdot u = 0. 
\]
Dividing by $\rho$\footnote{Here we assume $\rho>0$ in $\Omega$ for all time. As a consequence of the discrete maximum principle, if $\rho_0>0$, then $\rho(x,t)>0 \; \forall t\geq 0$. } and introducing the variable $S(x,t)$ as $S = \log \rho$, we obtain the equation
\[
\pad{S}{t} + u\cdot\nabla S + \nabla\cdot u = 0.
\]
By exploiting the definition of $u$ in \eqref{eq:velocityFP}, we obtain a transport equation for the new variable $S$, which can be rewritten in the following mixed Lagrangian-Eulerian formulation:
\begin{equation}
	\label{eq:S_Eulerian}
	\pad{S}{t} + \nabla(S_\infty-S)\cdot\nabla S = \Delta S - \Delta S_\infty
\end{equation}
This equation takes the form of a viscous Hamilton-Jacobi equation, which has been widely studied in other contexts, such as optimal control. In the absence of the $S_\infty$ term, the equation has been analyzed in \cite{BidautVeron2015HJ}. In addition, it can be recast in a Lagrangian-Eulerian formulation as follows:
\begin{equation}
	\label{eq:S_Lagrangian}
	\td{S}{t}  =  \Delta S - \Delta S_\infty, \quad \td{x}{t} =  \nabla(S_\infty-S),
\end{equation}
where the total derivative is taken along the flow induced by the 
velocity $u$.  This formulation is particularly suitable in the context of particle methods: each particle $i$ carries information about its position, $x_i(t)$, and the value of the unknown, $S_i(t) \approx S(x_i(t), t)$. We expect that the stiffness in system \eqref{eq:S_Lagrangian} arises from the diffusion term $\Delta S$, which is linear in this formulation and therefore more amenable to an implicit treatment
(see Section \ref{section:MLSSection}).
The boundary conditions in \eqref{eq:FPequation} take the more standard Neumann form
\begin{equation}\label{eq:SBC}
	\pad{S}{n} = \pad{S_\infty}{n} \quad \textrm{on}\> \partial\Omega\times\mathbb{R}_+, 
\end{equation}
which can be easily implemented in a particle scheme by the use of ghost particles mirrored across the boundary (see Section 
\ref{section:MLSBCs}).

\begin{remark}
	In the absence of the driving term (that is, when \(S_\infty\) is constant in \(\Omega\)), Eq.~\eqref{eq:FPequation} reduces to the linear heat equation for the density $\rho$.
	The velocity $u$ then becomes the osmotic velocity associated with the diffusion process
	\[
	u = -\nabla S = -\frac{\nabla\rho}{\rho}.
	\]
	The flow map induced by such osmotic velocity transforms the initial density $\rho(x,0)$ into the density $\rho(x,t)$, which solves the heat equation.
	This observation is at the core of a particle method for the heat equation based on the direct reconstruction of the osmotic velocity from the particle distribution
	\cite{russo1990deterministic}.
	If, on the other hand, one includes the potential $S_\infty(x)$, then the flow map asymptotically transforms the initial density $\rho(x,0)$ into the target density $\rho_\infty$.
	This is consistent with the convergence result stated in Theorem \ref{th:FPE_analysis}.
	The connection between particle methods based on osmotic velocity and the diffusion equation has been studied in
	\cite{russo1990particle} and \cite{russo1990deterministic}.
\end{remark}

\subsection{Stochastic formulation of the Fokker-Planck equation}
We recall here the stochastic counterpart of the Fokker-Planck formulation
introduced in Section~\ref{section:formulation}. This derivation provides both
a probabilistic interpretation of \eqref{eq:FPequation} and the basis for the
Monte Carlo method described in Section~\ref{section:langevin}.

Consider the stochastic differential equation (SDE)
\begin{equation}
	\label{eq:langevin}
	dX_t = \nabla S_\infty(X_t)\,dt + \sqrt{2}\, dW_t ,
\end{equation}
where $W_t$ is a standard $d$-dimensional Wiener process and
$S_\infty = \log \rho_\infty$. Equation \eqref{eq:langevin} is the first-order Langevin (Smoluchowski)
dynamics associated with the potential $-S_\infty$. The SDE is related to the Fokker-Planck equation in the following way. Let $\varphi \in C_c^\infty(\Omega)$.

Applying Itô's formula to $\varphi(X_t)$ \cite[Chapter~4]{oksendal_sde} gives
\[
d\varphi(X_t)
=
\nabla \varphi(X_t)\cdot \nabla S_\infty(X_t)\,dt
+ \Delta \varphi(X_t)\,dt
+ \sqrt{2}\,\nabla\varphi(X_t)\cdot dW_t.
\]
Taking expectations and noting that the stochastic integral has zero mean, we obtain
\[
\frac{d}{dt}\mathbb{E}[\varphi(X_t)]
=
\mathbb{E}\big[
\nabla \varphi(X_t)\cdot \nabla S_\infty(X_t)
+ \Delta \varphi(X_t)
\big].
\]
If we denote by $\rho(\cdot,t)$ the density of $X_t$, the preceding quantity can be written as
\[
\frac{d}{dt}
\int_\Omega \varphi(x)\rho(x,t)\,dx
=
\int_\Omega
\Big(
\nabla \varphi \cdot \nabla S_\infty
+ \Delta \varphi
\Big)
\rho\,dx.
\]
Since \(\varphi\) and all its derivatives vanish at the boundary, we can apply integration by parts to obtain an evolutionary equation for \(\rho\)
\[
\partial_t \rho
=
- \nabla \cdot \big(\rho \nabla S_\infty\big)
+ \Delta \rho ,
\]
which coincides with \eqref{eq:FPequation}. We refer for instance to \cite[Sec.~3.4]{pavliotis_stochastic} for a detailed derivation of the FPE from the corresponding SDE.

Since $\Omega$ is bounded, the stochastic process $X_t$ must remain in the closure $\overline{\Omega} = \Omega \cup \partial \Omega$ almost surely for all $t \ge 0$. The probabilistic counterpart of the no-flux (Neumann) boundary condition in the FPE \eqref{eq:FPequation} is the reflected SDE
\begin{equation}
	\label{eq:reflected_langevin}
	dX_t = \nabla S_\infty(X_t)\, dt + \sqrt{2}\, dW_t + n(X_t)\, dL_t,
\end{equation}
where $n(x)$ is the inward unit normal on $\partial\Omega$ and $L = (L_t)_{t\ge 0}$ denotes the boundary local time.
This non-decreasing process remains constant while $X_t$ lies in the interior of the domain and increases only when $X_t$ reaches the boundary; equivalently,
\[
L_t = \int_0^t \mathbf{1}_{\{X_s \in \partial\Omega\}}\, dL_s.
\]
Consequently, the term $n(X_t)\,dL_t$ acts only at the boundary and provides the minimal inward displacement required to prevent the stochastic trajectory from leaving
the domain. Existence, pathwise uniqueness, and the precise link to the Neumann boundary condition for the associated FPE were established by Lions and Sznitman
\cite{lions_sznitman_reflected} in smooth or Lipschitz domains.

The stochastic formulation can replace the deterministic solution of the FPE by the simulation of independent reflected Langevin trajectories. The stochastic formulation has the following requirements. 
\begin{itemize}
	\item The target density enters the Langevin drift only through its score, $\nabla S_\infty = \nabla \log \rho_\infty$. Consequently, the target density only needs to be known up to a positive multiplicative constant. 
	\item The initial distribution $\rho_0$ does not need to be known
	analytically. It is sufficient to initialize the particles with 
	samples $X_1^0,\dots,X_N^0$ satisfying
	\eqref{eq:iid_samples}. They implicitly define the  empirical measure
	\[
	\mu_0^N = \frac1N \sum_{i=1}^N \delta_{X_i^0}
	\]
	that acts as a discrete approximation of $\rho_0$.
	
\end{itemize}

\section{Meshfree method for the Fokker-Planck equation}
\label{section:MLSSection}
In this section, we introduce the meshfree method based on the S-formulation of the FPE introduced in Section \ref{sub:S-form}:
\begin{align}
	\frac{dS}{dt} & =\Delta S - \Delta S_\infty, \label{eq:population} \\
	\frac{dx}{dt} & = -\nabla S +\nabla S_\infty, \label{eq:characteristic}
\end{align} with \(S_\infty = \log \rho_{\infty}\) and initial condition \( \rho(\mathbf{x}, 0) = \rho_0(\mathbf{x})\), and boundary conditions \eqref{eq:SBC}. In the MLS method, we consider a set of grid points \(\mathbf{x}_i\), initially distributed according to \(\rho_0(\mathbf{x})\). The grid points are moved according to the characteristic equation \eqref{eq:characteristic}. On the moving grid, we solve equation \eqref{eq:population}. By solving the equation with a set of moving grid points or particles, all computational resources are used to advance the solution where the distribution \(\rho(\mathbf{x}, t)\) is non-zero. There is no need to discretize the equations on some truncated subdomain of \(\mathbb{R}^2\) using a fine mesh.

For the application of this method to point set registration, we require several ingredients. The characteristic equation \eqref{eq:characteristic} and population equation \eqref{eq:population} require the gradient and Laplacian of \(S_\infty = \log \rho_\infty\). These quantities are either known analytically, or it can be approximated, see Appendix \ref{appendix:DensityEstimation} on density estimation. The initial distribution of particles is sampled from \(\rho_0\), or analogously \(S_0\). It is therefore sufficient to know the distribution \(\rho_0\) up to a multiplicative constant, or analogously \(S_0\) up to an additive constant. 

In Section \ref{section:MLS}, we illustrate the MLS method used to discretize the right-hand sides of Equations \eqref{eq:population} and \eqref{eq:characteristic}. Then, in Section \ref{section:MLSTimeDiscretization}, we discuss the time discretization method. Finally, in Section \ref{section:MLSBCs}, we discuss boundary conditions. 

\subsection{Moving Least Squares}
\label{section:MLS}
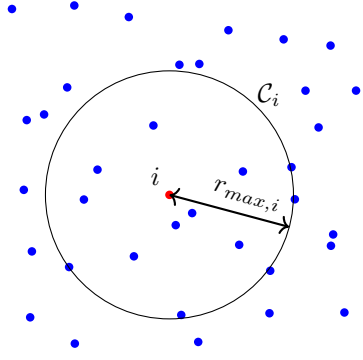
\begin{figure}	
	\centering
	\hspace*{\fill}  
	\begin{tikzpicture}[scale=0.8]
		
		\pgfmathsetseed{2023}
		
		\def\gridSize{8}
		\def\spacing{1}
		
		\foreach \i in {0,...,5} {
			\foreach \j in {0,...,5} {
				\pgfmathsetmacro{\dx}{rnd}
				\pgfmathsetmacro{\dy}{rnd}
				\fill[blue] ({(\i+\dx*0.95)*\spacing}, {(\j+\dy*0.95)*\spacing}) circle (2pt);
			}
		}
		
		\pgfmathsetmacro{\x}{2.7}
		\pgfmathsetmacro{\y}{2.7}
		\def\r{2.05}
		\def\angle{-15}
		
		\fill[red] (\x, \y) circle (2pt);
		\node[black, above left] at (\x, \y) {$i$};
		\draw (\x, \y) circle (\r);
		\node[black] at ({\x + \r*cos(45) + 0.2}, {\y + \r*sin(45) + 0.2}) {$\mathcal{C}_i$};

		\draw[<->, thick] (\x,\y) -- ++(\angle:\r);
		\node[anchor=east, rotate=\angle] at ($(\x,\y)!0.5!(\x,\y)+(\angle:\r) + (0, 0.3)$) {$r_{max, i}$};
		
	\end{tikzpicture}
	\hspace*{\fill}  
	\caption{An illustration of a grid of points \(\xb_j\), the set of neighbors \(\mathcal{C}_i\) and the distance \(r_{max, i}\).}
	\label{fig:particleManagement}
\end{figure}
In this section, we summarize a two-dimensional Moving Least Squares (MLS) method. Consider a set of \(N\) points \(\xb_i = [x_i, y_i]^\top, \, i = 1 \dots N\) initially distributed according to \(\rho_0(x)\). We associate to each point \(\xb_i\) a local central neighborhood \(\mathcal{C}_i\), as illustrated in Figure \ref{fig:particleManagement}. The set \(\mathcal{C}_i\) contains the closest \(N_b\) neighbors to point \(\xb_i\), i.e., \begin{align}
	\mathcal{C}_i = \{\, \xb_j \mid j \neq i \text{ and } d(\xb_i, \xb_j) \text{ is among the } N_b \text{ smallest} \,\}
\end{align} where \(d(\xb_i, \xb_j)\) denotes the Euclidean distance, \(d(\xb_i, \xb_j) = \norm{\xb_i - \xb_j}_2\). These neighbors can be computed efficiently using a stationary background grid, see \cite[Sec.~3.2]{seibold_m-matrices_2006}. We define \(r_{{\rm max}, i}\) as the distance to the farthest point in \(\mathcal{C}_i\), 
\begin{align}
	r_{{\rm max}, i} = \max_{j \in \mathcal{C}_i}\,  d(\xb_i, \xb_j). 
\end{align} MLS approximates \(S\) by performing a Taylor expansion around the point \(S(\mathbf{x}_i) = S_i\), temporarily neglecting the time dependence of \(S\). This yields \begin{align}
	S_j = S_i + D_i^\top X_{ij} + \mathcal{O}(\Delta x_{ij}^3) + \mathcal{O}(\Delta y_{ij}^3), \quad j \in \mathcal{C}_i, \label{eq:MLSTaylor}
\end{align} with 
\begin{align}
	D_i = \left[\pdv{S_i}{x_i}, \pdv{S_i}{y_i}, \pdv[2]{S_i}{x_i}, \pdv[2]{S_i}{y_i}, \pdv[2]{S_i}{x_i}{y_i}\right]^\top, \quad 
	X_{ij} = \left[\Delta x_{ij}, \Delta y_{ij}, \frac{1}{2}\Delta x_{ij}^2, \frac{1}{2}\Delta y_{ij}^2, \Delta x_{ij}\Delta y_{ij} \right]^\top,
\end{align} where \(\Delta x_{ij} = x_j - x_i\) and \(\Delta y_{ij} = y_j - y_i\). Equation \eqref{eq:MLSTaylor} defines a polynomial in \(\Delta x_{ij}\) and \(\Delta y_{ij}\) with coefficients \(D_i\) that approximates the data near \(\xb_i\). Spatial derivatives at \(\xb_i\) are then approximated by minimizing the \(L^2\) norm of the interpolation error in \(\mathcal{C}_i\) with respect to some weight function 
\begin{align}
	\tilde{D}_i = \argmin_{D_i} \sum_{j \in \mathcal{C}_i} w_{ij} \left(S_j - S_i - D_i^\top X_{ij} \right)^2. \label{eq:LS}
\end{align} We use a Gaussian weight function 
\begin{align}
	w_{ij} = w(\xb_i, \xb_j) = \exp \left(-6 \frac{\norm{\xb_i - \xb_j}^2}{r_{{\rm max}, i}^2} \right), 
\end{align} 
although many other choices exist \cite{seibold_m-matrices_2006}. The least-squares problem can be solved by standard techniques such as an SVD or QR decomposition.
Rescaling of the vector \(X_{ij}\) is performed to avoid problems due to ill-conditioning. Solving \eqref{eq:LS} simultaneously provides approximations of all five
derivatives contained in \(\tilde{D}_i\). Each approximation can be written as a finite-difference formula:
\begin{align}
	\widetilde{\partial_x S_i}
	&= \sum_{j \in \mathcal{C}_i} \alpha_{ij}^{x}(S_j-S_i),
	&
	\widetilde{\partial_y S_i}
	&= \sum_{j \in \mathcal{C}_i} \alpha_{ij}^{y}(S_j-S_i),
	\\
	\widetilde{\partial_{xx} S_i}
	&= \sum_{j \in \mathcal{C}_i} \alpha_{ij}^{xx}(S_j-S_i),
	&
	\widetilde{\partial_{yy} S_i}
	&= \sum_{j \in \mathcal{C}_i} \alpha_{ij}^{yy}(S_j-S_i),
	\\
	\widetilde{\partial_{xy} S_i}
	&= \sum_{j \in \mathcal{C}_i} \alpha_{ij}^{xy}(S_j-S_i).
\end{align}
The coefficients are obtained from the corresponding rows of the weighted least-squares pseudoinverse. In particular, the Laplacian is approximated by
\begin{align}
	\widetilde{\Delta S_i}
	= \sum_{j \in \mathcal{C}_i}
	\left(\alpha_{ij}^{xx}+\alpha_{ij}^{yy}\right)(S_j-S_i).
\end{align}
Because we used a quadratic approximation in \eqref{eq:MLSTaylor}, first-order derivatives are approximated to second-order, and second-order derivatives achieve
first-order accuracy. The extension to higher orders or higher dimensions is straightforward.

We introduce the following vector/matrix notation. The vector \(\mathbf{p} \in \mathbb{R}^{2N}\) is a concatenation of the $x$ and $y$ coordinates of all points.
Similarly, we define the vector \(\mathbf{S} \in \mathbb{R}^N \) containing all the values \(S_i\), \begin{align*}
	p &= [x_1, \dots x_N, \, y_1 \dots y_N]^\top, \\
	S &= [S_1, \, S_2, \, \dots \, S_N]^\top. 
\end{align*} 
Then, we define the sparse matrices \(D_L, D_x, D_y \in \mathbb{R}^{N \times N}\) such that the matrix-vector products \(D_L S, D_x S\) and \(D_y S\) yield
discrete approximations of the the Laplacian and the first derivatives in \(x\) and \(y\), respectively. For example, the matrix \(D_x\) contains
the coefficients \(\alpha_{ij}^{x}\) on the off-diagonal and \(-\sum_{j \in \mathcal{C}_i} \alpha_{ij}^{x}\) on its diagonal. Similarly, the off-diagonal entries of
\(D_L\) are \(\alpha_{ij}^{xx}+\alpha_{ij}^{yy}\). The MLS approximations of the spatial derivatives are used to discretize the spatial variable in the S-formulation
of the Fokker-Planck equation.

\begin{remark}
	The MLS method presented here suffers from ill-conditioning for higher-order approximations due to the use of a shifted monomial basis. The Local Anisotropic Basis Function Method (LABFM) provides a promising alternative \cite{king_high_2020,king_high-order_2022}. It employs better-conditioned basis functions and does not rely on the solution to a least-squares method to determine the coefficients. In our numerical experiments, LABFM performed well for basic test cases such as those in Section \ref{section:DiffusionTestCase}, but exhibited boundary instabilities in the test cases from Sections \ref{section:gaussianAcrossCylinder} and \ref{section:PSRAcrossCylinder}. We therefore use the quadratic MLS discretization described above throughout this work. Developing a robust boundary treatment for LABFM is left for future investigation.
\end{remark}

\subsection{Time discretization}
\label{section:MLSTimeDiscretization}
\subsubsection{First-order scheme}
\label{section:FirstOrderScheme}
To avoid stiffness due to the diffusion operator, we discretize \eqref{eq:population} with an implicit Euler method.
The characteristic equation \eqref{eq:characteristic} is discretized using an explicit Euler method.
Although the spatial derivatives of \(S_{\infty}\) can be obtained analytically, they are also approximated using the discrete operators \(D_L, D_x\) and \(D_y\).
This yields a well-balanced scheme that preserves the equilibrium \(S = S_\infty\) at the discrete level. The resulting fully discrete scheme is
\begin{align}
	S^{n+1} &= S^n + \Delta t D_L^{n+1} \left( S^{n+1} - S_{\infty}^{n+1} \right), \label{eq:populationUpdate} \\
	p^{n+1} &= p^n - \Delta t \left[D_x^n, \, D_y^n\right]^\top \left( S^{n} - S_{\infty}^{n} \right), \label{eq:characteristicUpdate}
\end{align} 
where the superscript \(n\) denotes the numerical solution at time \(t^n = n\,\Delta t\), with \(n = 0,\ldots,N_t\).
First, the particle positions are updated using Equation \eqref{eq:characteristicUpdate}. The discrete spatial operators at the next time step, \(D_L^{n+1}, D_x^{n+1}\) and \(D_y^{n+1}\),
are then computed by solving the least-squares problems \eqref{eq:LS} at every point.
Finally, the log-density \(S^{n+1}\) is updated by solving the linear system \eqref{eq:populationUpdate}.
In Section \ref{section:scoreEvaluation}, we discuss how to evaluate \(S_\infty\) stably. This scheme is first-order in time.
Due to the underlying MLS approximation, Equation \eqref{eq:populationUpdate} is first-order in space, while Equation \eqref{eq:characteristicUpdate} is second-order in space.
To control the explicit integration of the characteristic equation, we approximate the Jacobian of its velocity at every node by the Hessian of \(\tilde S = S_\infty-S\),
\begin{align*}
H_i = \begin{pmatrix}
	\pdv[2]{\tilde{S}_i}{x_i} & \pdv[2]{\tilde{S}_i}{x_i}{y_i} \\[2mm]
	\pdv[2]{\tilde{S}_i}{x_i}{y_i} & \pdv[2]{\tilde{S}_i}{y_i}
\end{pmatrix}.
\end{align*}
Linearizing the characteristic equation around the current particle position gives \(\dot{\delta p}=H_i\delta p\).
For a negative eigenvalue \(\lambda\) of \(H_i\), the explicit Euler method is stable provided that \(\Delta t\leq 2/|\lambda|\).
Since the second derivatives are readily available from the MLS approximation, computing these eigenvalues is inexpensive.
We therefore use the most negative eigenvalue to define a conservative local stability restriction,
\begin{align*}
\Delta t = \begin{cases}
	\min \left(\Delta t_{\max}, -\frac{1}{10 \lambda_{\min}} \right) \quad \text{if } \lambda_{\min} < 0, \\
	\Delta t_{\max} \quad \text{else},
\end{cases}
\end{align*}
where
\begin{align*}
\lambda_{\min} = \min_i \lambda_{\min}(H_i),
\end{align*}
and \(\Delta t_{\max}\) denotes the user-defined maximum allowed time step.
The factor \(1/10\) provides a conservative safety margin relative to the explicit Euler stability bound.

\subsubsection{Second-order scheme}
\label{section:SecondOrderScheme}
The first-order scheme from Section \ref{section:FirstOrderScheme} can be extended to second-order using the Implicit-Explicit (IMEX) Runge-Kutta framework \cite{ascher_implicit-explicit_1995}. In this work, we use the ARS(1, 2, 2) scheme, of which the Butcher tableaux are given by 
\begin{align}
	\begin{array}{c|cc}
		0 & 0 & 0 \\
		\frac{1}{2} & \frac{1}{2} & 0  \\[1mm]
		\hline
		& 0 & 1
	\end{array}, 
	\quad
	\begin{array}{c|cc}
		0 & 0 & 0 \\
		\frac{1}{2} & 0 & \frac{1}{2} \\[1mm]
		\hline
		& 0 & 1
	\end{array}.
	\label{eq:ButcherTableaus}
\end{align} 
This scheme corresponds to applying the explicit midpoint rule to the characteristic equation \eqref{eq:characteristic} (left tableau) and the implicit midpoint rule to the population equation \eqref{eq:population}
(right tableau). In conclusion, the second-order method reads as \begin{align*}
	p^{n+1/2} &= p^n - \frac{\Delta t}{2} \left[D_x^n, \, D_y^n\right]^\top \left( S^{n} - S_{\infty}^{n} \right), \\
	S^{n+1/2} &= S^n + \frac{\Delta t}{2} D_L^{n+1/2} \left( S^{n+1/2} - S_{\infty}^{n+1/2} \right), \\
	p^{n+1} &= p^n - \Delta t \left[D_x^{n+1/2}, \, D_y^{n+1/2}\right]^\top \left( S^{n+1/2} - S_{\infty}^{n+1/2} \right), \\
	S^{n+1} &= S^n + \Delta t D_L^{n+1} \left( S^{n+1} - S_{\infty}^{n+1} \right).
\end{align*}

\subsection{Boundary conditions}
\label{section:MLSBCs}
By discretizing the distribution with a finite set of particles, we inevitably obtain particles on the boundary of the convex hull of all particles. We refer to these particles as edge particles. Edge particles do not require any additional treatment compared to particles that lie strictly inside the convex hull of all particles. Spatial derivatives for the density equation \eqref{eq:populationUpdate} and characteristic equation \eqref{eq:characteristicUpdate} associated with an edge particle can still be computed using MLS from the set \(\mathcal{C}_i\). However, we note that this neighborhood is no longer centered around the edge particle, but `upwinded' in the direction of the nearest neighboring particles. This feature gives this method its interesting property: the spatial domain does not need to be truncated and discretized. 

For the application to point set registration, reflective boundary conditions at walls are required. These can be implemented using ghost particles. All particles within a certain distance of a reflective wall are reflected across the wall. We refer to the reflected particles as {\em ghost particles}. We associate to each ghost particle the same value of the log-density \(S_{g_i}^n = S_i^n\) and \(S_{g_i, \infty}^n = S_{i, \infty}^n\) as the associated `true' particle. Then, when particles venture close to the boundary, ghost particles are automatically included in their neighborhoods \(\mathcal{C}_i\). Due to the reflection of the log-density to the ghost particle, the ghost particles will act as a reflective force. When a particle is on the boundary, due to symmetry of the stencil, the normal component of the velocity of the particle with respect to the boundary is zero. 

\begin{remark}
	Reflective boundary conditions can alternatively be enforced by imposing zero normal velocity directly in the least-squares problem \cite{suchde_conservation_2018}. We tested this constrained formulation in preliminary numerical experiments; for the configurations considered here, however, the ghost-particle approach exhibited greater numerical robustness. We therefore employ the ghost-particle treatment throughout this work and do not investigate the constrained formulation further.
\end{remark}

\subsection{Evaluation of the score}
\label{section:scoreEvaluation}
In the numerical experiments below, the final density \(\rho_{\infty}\) is estimated using a Gaussian mixture model (GMM) or kernel density estimation (KDE); see
\cite[Chapter~6]{hastie2009elements} and Appendix \ref{appendix:DensityEstimation}. In the GMM case, direct evaluation of the score
\begin{align*}
	\nabla \log \rho_{\infty}(\mathbf{x}_i^n) = \nabla \log \sum_{k=1}^K w_k \phi_k(\mathbf{x}_i^n)
\end{align*}
can cause numerical underflow when all component densities are very small. The score depends only on the relative contributions of the mixture components. Consequently,
multiplying all weighted component densities by the same positive factor does not change the score because this factor cancels between the numerator and denominator.
We exploit this invariance using the log-sum-exp trick \cite{blanchard_accurately_2021}: before exponentiation, we subtract the largest weighted component log-density.
The largest rescaled contribution is then equal to one, while all others lie between zero and one, preventing simultaneous underflow in the numerator and denominator.

For \(k=1,\dots,K\), define the exponent and normalization terms
\begin{align*}
	e_k(\mathbf{x}_i^n) &= -\frac{1}{2}(\mathbf{x}_i^n-\boldsymbol{\mu}_k)^\top \boldsymbol{\Sigma}_k^{-1}(\mathbf{x}_i^n-\boldsymbol{\mu}_k),
	& Z_k &= (2\pi)^{-1}\det(\boldsymbol{\Sigma}_k)^{-1/2},
\end{align*}
such that \(\phi_k(\mathbf{x}_i^n)=Z_k\exp(e_k(\mathbf{x}_i^n))\) for \(d=2\). We introduce the weighted component log-densities and their maximum,
\begin{align*}
	a_{ik}^n &= \log w_k + \log Z_k + e_k(\mathbf{x}_i^n),
	& h_i^n &= \max_{k=1,\dots,K} a_{ik}^n,
\end{align*}
and define the component responsibilities
\begin{align*}
	\gamma_{ik}^n = \frac{\exp(a_{ik}^n-h_i^n)}{\sum_{\ell=1}^K \exp(a_{i\ell}^n-h_i^n)}.
\end{align*}
The density can be recovered as
\begin{align*}
	\rho_\infty(\mathbf{x}_i^n) = \exp(h_i^n)\sum_{k=1}^K \exp(a_{ik}^n-h_i^n),
\end{align*}
while the common factor \(\exp(h_i^n)\) cancels from the score. Since the score of the \(k\)-th Gaussian component is
\(-\boldsymbol{\Sigma}_k^{-1}(\mathbf{x}_i^n-\boldsymbol{\mu}_k)\), the mixture score is the responsibility-weighted average
\begin{align*}
	\nabla \log \rho_\infty(\mathbf{x}_i^n)
	= -\sum_{k=1}^K \gamma_{ik}^n \boldsymbol{\Sigma}_k^{-1}(\mathbf{x}_i^n-\boldsymbol{\mu}_k).
\end{align*}

\section{Monte Carlo simulation of the Langevin stochastic dynamics with boundaries}
\label{section:langevin}
In this section, we outline a stochastic particle method for point set registration. First, in Section \ref{section:MC}, we discuss the discretization and treatment of the boundary conditions. Then, in Section \ref{section:OsmoticPaths}, we discuss how to obtain the osmotic paths from the stochastic particles. Finally, in Section \ref{section:MCAdaptiveRadius}, we give a method to adaptively tune a hyperparameter that balances the bias and Monte Carlo error.  

\subsection{Monte Carlo particle method}
\label{section:MC}
The Langevin SDE \eqref{eq:langevin} is discretized using a standard Euler--Maruyama scheme \cite[Chapter~10]{kloeden_numerical_1992} with \(i = 1, \dots, N\) particles \begin{align}
	\mathbf{X}^{n+1}_i = \mathbf{X}^n_i + \nabla \log \rho_\infty(\mathbf{X}^n_i) \Delta t + \sqrt{2\Delta t} \,\boldsymbol{\xi}^n_i, \quad \boldsymbol{\xi}^n_i \sim \mathcal{N}(\mathbf{0},\mathbf{I}), \label{eq:EM}
\end{align} where \(\mathbf{I} \in \mathbb{R}^{d \times d}\) is the identity matrix. Initial positions \(\mathbf{X}^0_i\) are sampled from the distribution \(\rho_0(\mathbf{x})\). The evaluation of the score \(\nabla \log \rho_\infty(\mathbf{X}^n_i)\) in Equation \eqref{eq:EM} is done using the log-sum-exp trick from Section \ref{section:scoreEvaluation}. This scheme is of weak order one, meaning that expectations of sufficiently smooth functionals of the solution are approximated with first-order accuracy in the time step. Since the diffusion coefficient is constant, the scheme is also of strong order one, i.e., individual sample paths are approximated with first-order accuracy in the mean-square sense. For a more detailed discussion of numerical methods for SDEs, we refer the reader to \cite{kloeden_numerical_1992}.
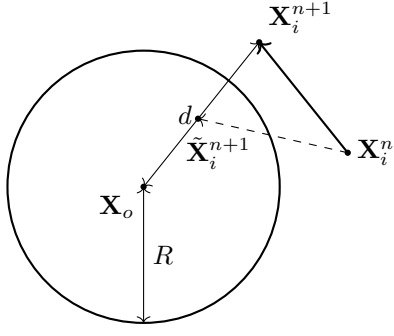
\begin{figure}
	\centering
	\begin{tikzpicture}[scale=0.9]
		
		\def\R{2}
		
		\coordinate (O) at (0,0);
		
		\draw[thick] (O) circle (\R);
		\draw[<->, thin] (0, 0) -- (0, -2) node[midway, right] {$R$};
		
		\coordinate (Pinside) at (0.8, 1);
		\fill (Pinside) circle (1.3pt);
		\node at (1.1,0.5) {$\Tilde{\mathbf{X}}_i^{n+1}$};
		
		\coordinate (Pout) at (1.69878,2.1234);
		\fill (Pout) circle (1.3pt) node[above right] {$\mathbf{X}_i^{n+1}$};
		
		\coordinate (Poutside) at (3,0.5);
		\fill (Poutside) circle (1.3pt) node[right] {$\mathbf{X}_i^{n}$};
		
		\draw[->, thick] (Poutside) -- (Pout);
		\draw[->, dashed] (Poutside) -- (Pinside);
		\draw[<->, thin] (0, 0) -- (Pout) node[midway, left] {$d$};
		
		\fill (O) circle (1.3pt) node[below left] {$\mathbf X_o$};
		
	\end{tikzpicture}
	\caption{Illustration of the symmetric reflection rule.}
	\label{fig:SymmetricReflection}
	
\end{figure}

In the numerical experiments for bounded domains, we will study the dynamics past fully reflective obstacles. In order to simulate the Skorokhod problem
\cite{lions_sznitman_reflected} in the presence of obstacles, we adopt a consistent approximation. For simplicity, we consider circular obstacles with center
\(\mathbf{X}_o\) and radius \(R\), and we denote by \(\Tilde{\mathbf{X}}_i^{n+1}\) the position of the particle generated by Equation \eqref{eq:EM} that is inside the circle,
i.e., $\|\Tilde{\mathbf{X}}^{n+1}_i - \mathbf{X}_o\| < R$. Here, the subscript \(i\) indexes the particles, while the subscript \(o\) identifies the obstacle.
Then, the particles are symmetrically reflected outward. That is, they are placed at a distance \begin{align*}
	d = R + \left(R - \| \Tilde{\mathbf{X}}^{n+1}_i - \mathbf{X}_o \|\right) = 2R - \| \Tilde{\mathbf{X}}^{n+1}_i - \mathbf{X}_o \| 
\end{align*} from the center of the obstacle \(\mathbf{X}_o\) along the same radial direction. The final update for their position reads \begin{align*}
	\mathbf{X}^{n+1}_i = \mathbf{X}_o +  d \frac{\Tilde{\mathbf{X}}^{n+1}_i - \mathbf{X}_o}{\|\Tilde{\mathbf{X}}^{n+1}_i - \mathbf{X}_o\|}. 
\end{align*} See also Figure \ref{fig:SymmetricReflection}. 
\subsection{Osmotic paths}
\label{section:OsmoticPaths}
We now require a method to extract the `osmotic paths', which are the solutions of \begin{align}
	\frac{d\mathbf{x}}{dt} & =  \nabla S_\infty(\mathbf{x}) - \nabla S(\mathbf{x}), \quad \mathbf{x}(0) = \mathbf{X}_c, \label{eq:osmoticPaths}
\end{align} from the stochastic paths \(\{\mathbf{X}_i^n\}_{i = 1 \dots N}^{n = 0 \dots N_t}\) generated by the discretized SDE \eqref{eq:EM}. We do so by considering a small circular neighborhood of radius \(r\) centered at \(\mathbf{X}_c^n\). We define the indices of the particles inside the small circle as \(C^n = \{i \in \{1,\ldots,N\} : \norm{\mathbf{X}_i^n - \mathbf{X}_c^n}_2 \leq r\}\), and \(\abs{C^n}\) denotes the number of such particles. Then, we solve \eqref{eq:osmoticPaths} with a first-order method, \begin{align}
	\mathbf{X}_c^{n+1} = \mathbf{X}_c^n + \Delta t \left[ \nabla S_\infty(\mathbf{X}_c^n) - \nabla S(\mathbf{X}_c^n, n \Delta t) \right]. \label{eq:osmoticPathsdiscretized}
\end{align} The score of \(\rho_{\infty}\), i.e. \(\nabla S_\infty\), can be evaluated at the center of the small circle \(\mathbf{X}_c^n\); however, the score of \(\rho\) must be approximated using the particles in the local neighborhood \(C^n\). Theorem \ref{theorem:velocity} yields a method to achieve this.  

First, we define the following objects. Let \(p : \mathbb{R}^2 \rightarrow \mathbb{R}\) be a twice continuously differentiable probability density function. For \(\mathbf{x} \in \mathbb{R}^2\) and \(r > 0\), define the ball \(B(\mathbf{x}, r) = \{\mathbf{y} \in \mathbb{R}^2 : \vert \mathbf{y} - \mathbf{x} \vert \leq r \}\). The ball \(B(\mathbf{x}, r)\) is the continuous neighborhood corresponding to the index set \(C^n\). Now, define 
\begin{align*}
	m(\mathbf{x}, r) := \int_{B(\mathbf{x}, r)} p(\mathbf{y}) d\mathbf{y}, \quad I(\mathbf{x}, r) := \int_{B(\mathbf{x}, r)} (\mathbf{y} - \mathbf{x}) p(\mathbf{y}) d\mathbf{y}, 
\end{align*} 
and the local mean-shift vector
\begin{align}
	V(\mathbf{x}, r) := \frac{I(\mathbf{x}, r)}{m(\mathbf{x}, r)} = \frac{\int_{B(\mathbf{x}, r)} (\mathbf{y} - \mathbf{x}) p(\mathbf{y})d\mathbf{y}}{\int_{B(\mathbf{x}, r)} p(\mathbf{y})d\mathbf{y}} \label{eq:scoreVelocity}. 
\end{align} 
The following theorem connects the score of a density \(p\) with the local mean-shift vector in Equation \eqref{eq:scoreVelocity}. \begin{theorem} \label{theorem:velocity}
	Given \(\mathbf{x} \in \Omega\), suppose that the density \(p\) is of class \(C^3\) and \(p(\mathbf{x}) > 0\). Then, the mean-shift vector \(V(\mathbf{x},r)\)
	defined in \eqref{eq:scoreVelocity} satisfies
	\begin{align*}
		V(\mathbf{x}, r) = \frac{r^2}{4} \frac{\nabla p(\mathbf{x})}{p(\mathbf{x})} + \mathcal{O}(r^4),
	\end{align*} or equivalently, \begin{align}
		\frac{\nabla p(\mathbf{x})}{p(\mathbf{x})} = \frac{4}{r^2} V(\mathbf{x}, r) + \mathcal{O}(r^2). \label{eq:scoreEq}
	\end{align}
\end{theorem}
\begin{proof}
	Making the change of variables \(\mathbf{z} = \mathbf{y} - \mathbf{x}\) gives \begin{align}
		m(\mathbf{x}, r) = \int_{\abs{\mathbf{z}} \leq r} p(\mathbf{x} + \mathbf{z})d\mathbf{z}, \quad I(\mathbf{x}, r) = \int_{\abs{\mathbf{z}} \leq r} \mathbf{z} p(\mathbf{x} + \mathbf{z}) d\mathbf{z}. \label{eq:numAndDenum} 
	\end{align} We perform a Taylor expansion of \(p\) at \(x\), \begin{align*}
		p(\mathbf{x} + \mathbf{z}) = p(\mathbf{x}) + \nabla p(\mathbf{x}) \cdot \mathbf{z} + \frac{1}{2} \mathbf{z}^\top H_p(\mathbf{x}) \mathbf{z} + \mathcal{O}(\abs{\mathbf{z}}^3),
	\end{align*} where \( H_p(\mathbf{x})\) denotes the Hessian. Then, the integrands in Equation \eqref{eq:numAndDenum} can be expanded. For the mass \(m(\mathbf{x}, r)\), we find \begin{align}
		m(\mathbf{x}, r) &= \int_{\abs{\mathbf{z}} \leq r} \left[p(\mathbf{x}) + \nabla p(\mathbf{x}) \cdot \mathbf{z} + \frac{1}{2} \mathbf{z}^\top H_p(\mathbf{x}) \mathbf{z} + \mathcal{O}(\abs{\mathbf{z}}^3)\right] d\mathbf{z} = \pi r^2 p(\mathbf{x}) + \mathcal{O}(r^4), \label{eq:scoreMass}	\end{align} where the linear term vanishes due to symmetry. Expansion of the numerator yields \begin{align*}
		I(\mathbf{x}, r) = \int_{\abs{\mathbf{z}} \leq r} \mathbf{z} \left[p(\mathbf{x}) + \nabla p(\mathbf{x}) \cdot \mathbf{z} + \frac{1}{2} \mathbf{z}^\top H_p(\mathbf{x}) \mathbf{z} + \mathcal{O}(\abs{\mathbf{z}}^3) \right] d\mathbf{z}. 	\end{align*} The odd terms in \(z\) vanish due to symmetry. The remaining quadratic term yields \begin{align*}
		\int_{\abs{\mathbf{z}} \leq r} \mathbf{z} \left( \nabla p(\mathbf{x}) \cdot \mathbf{z}\right) d\mathbf{z} = \left( \int_{\abs{\mathbf{z}} \leq r} \mathbf{z} \mathbf{z}^\top d\mathbf{z} \right) \nabla p(\mathbf{x}) = \frac{\pi r^4}{4} \nabla p(\mathbf{x}), 
	\end{align*} such that \begin{align}
		I(\mathbf{x}, r) = \frac{\pi r^4}{4} \nabla p(\mathbf{x}) + \mathcal{O}(r^6). \label{eq:scoreNumerator}
	\end{align} Combining Equations \eqref{eq:scoreMass} and \eqref{eq:scoreNumerator}, we find \begin{align*}
		V(\mathbf{x}, r) = \frac{I(\mathbf{x}, r)}{m(\mathbf{x}, r)} = \frac{\frac{\pi r^4}{4} \nabla p(\mathbf{x}) + \mathcal{O}(r^6)}{\pi r^2 p(\mathbf{x}) + \mathcal{O}(r^4)} = \frac{r^2}{4} \frac{\nabla p(\mathbf{x})}{p(\mathbf{x})}\left(1 + \mathcal{O}(r^2)\right) = \frac{r^2}{4} \frac{\nabla p(\mathbf{x})}{p(\mathbf{x})} + \mathcal{O}(r^4). 
	\end{align*} This completes the proof.
\end{proof} We are given a collection of particles \(\mathbf{X}_i^n\), generated by the discretized Langevin dynamics \eqref{eq:EM}, which thus sample the distribution \(\rho(\mathbf{x}, n\Delta t)\). These particles can now be used to approximate the score of \(\rho(\mathbf{x}, n\Delta t)\) by combining the result of the previous theorem \eqref{eq:scoreEq} and by approximating the integrals in Equation \eqref{eq:scoreVelocity} in a Monte Carlo fashion, \begin{align}
	\nabla S(\mathbf{X}_c^n, n \Delta t) = \frac{\nabla \rho(\mathbf{X}_c^n, n \Delta t)}{\rho(\mathbf{X}_c^n, n \Delta t)} \approx \frac{4}{r^2} \frac{1}{\abs{C^n}} \sum_{i \in C^n} \left(\mathbf{X}_i^n - \mathbf{X}_c^n\right). \label{eq:meanShift}
\end{align} The right-hand side of Equation \eqref{eq:meanShift} contains the mean shift vector with respect to the center of the circle. If the density gradient is zero, the mean shift vector will be zero, due to the symmetry of the particles \(\mathbf{X}_i^n\) around \(\mathbf{X}_c^n\). In contrast, when the density gradient is nonzero, on average, particles will be more concentrated in the direction of the gradient. As a result, the mean shift vector will tend to point along the gradient direction, with a magnitude proportional to the magnitude of the gradient. The size of the circular region \( r \) is a hyperparameter that controls the trade-off between statistical and deterministic errors along the osmotic path. For larger regions, the statistical error of the mean shift vector is reduced, as more samples contribute to the estimate. However, this comes at the cost of increased bias due to higher-order terms in the Taylor expansion. Conversely, smaller regions reduce this bias but lead to higher statistical error. The mean shift vector was first derived in \cite{fukunaga_estimation_1975} in the context of nonparametric gradient estimation of densities and is widely used in mode seeking, clustering, and density optimization \cite{yizong_cheng_mean_1995}. 

\subsubsection{Optimal radius for bias-variance balancing of the score estimator}
\label{section:MCAdaptiveRadius}
We now use the score estimator \eqref{eq:meanShift} to relate the optimal radius to the local particle density to balance the bias and variance of the score estimator. The deterministic truncation error is \(\mathcal{O}(r^2)\) by Theorem~\ref{theorem:velocity}. Each displacement \(\mathbf{X}_i^n-\mathbf{X}_c^n\), with \(i\in C^n\), has magnitude at most \(r\) and therefore a standard deviation of order \(r\). Consequently, the standard deviation of the sample mean in \eqref{eq:meanShift} is \(\mathcal{O}(r/\sqrt{\abs{C^n}})\). Multiplication by \(4/r^2\) gives the statistical error
\[
	\mathcal{O}\left(\frac{1}{r\sqrt{\abs{C^n}}}\right).
\]
The expected number of particles in the neighborhood is
\[
	\mathbb{E}\left[\abs{C^n}\right]
	= N\int_{B(\mathbf{X}_c^n,r)}\rho(\mathbf{y},n\Delta t)\,d\mathbf{y}
	\approx N\pi r^2\rho(\mathbf{X}_c^n,n\Delta t),
\]
where \(N\) is the total number of Monte Carlo particles. Writing \(\rho=\rho(\mathbf{X}_c^n,n\Delta t)\) for brevity and replacing \(\abs{C^n}\) by its leading-order expectation, the total error in the score estimate is
\[
	{\rm err} = \mathcal{O}(r^2)+\mathcal{O}\left(\frac{1}{r^2\sqrt{N\rho}}\right).
\]
Let $A$ and $B$ denote constants such that
\[
    {\rm err} = Ar^2 + \frac{B}{r^2\sqrt{N\rho}} + {\rm h.o.t.}
\]
Neglecting higher-order terms, the minimum is obtained by differentiating the expression with respect to $r$, giving
\[
    2A\ropt - \frac{2B}{\ropt^3\sqrt{N\rho}} = 0
\]
and hence
\[
    \ropt^4 = \frac{B}{A\sqrt{N\rho}}.
\]
Thus, the optimal radius and the expected number of particles inside the neighborhood scale with the density as
\[
	\ropt = \mathcal{O}\left((N\rho)^{-1/8}\right) \quad \Longrightarrow \quad \mathbb{E}\left[\abs{C^n}\right] = \mathcal{O}\left((N\rho)^{3/4}\right).
\] 
We introduce a proportionality constant for the optimal radius as \begin{align}
		\ropt(\rho) = \kappa (N\rho)^{-1/8} = r_{\mathrm{ref}} \left( \frac{N_{\mathrm{ref}} \rho_{\mathrm{ref}}}{N\rho} \right)^{1/8}, \label{eq:rOpt}
\end{align} with \(\kappa = r_{\mathrm{ref}} \left( N_{\mathrm{ref}} \rho_{\mathrm{ref}} \right)^{1/8}\). We note that the density in equation \eqref{eq:rOpt} is position and time-dependent. The proportionality constant \(\kappa\) is calibrated empirically
using the simple problem presented in Section \ref{section:radiusCalibration}. In principle, its value is problem-dependent and should therefore be recalibrated for each
simulation. However, since all point sets considered here are represented by Gaussian distributions of comparable scale, we use the same calibrated value in the subsequent
experiments. We use the following algorithm to update the radius of a small circular neighborhood. We initialize the radius as \(r_0 = r_{\text{init}}\).
Then, in every time step, the density in the ball is estimated using the radius and number of particles in the ball at the previous time step \begin{align}
		\hat{\rho}^n = \frac{\abs{C^{n-1}}}{N \pi (r^{n-1})^2}. \label{eq:densityEst}
\end{align} Then, the new radius is computed from the lagged density estimate and \eqref{eq:rOpt} as \(r^n = \ropt(\hat{\rho}^n)\). For robustness, we apply two additional safeguards. First, if the old neighborhood is empty, i.e., \(\abs{C^{n-1}} = 0\), the density is undefined and we set \(r^n = 2r^{n-1}\). In either case, the resulting radius is restricted to the interval \([r_{\min},r_{\max}] = [1\times 10^{-3}, 1]\), thereby preventing excessively small or large neighborhoods. The complete Monte Carlo procedure is summarized in Algorithm~\ref{alg:MC}.
\begin{remark}
	Alternatively to the procedure above, one could recompute the density using the new radius and iterate to obtain a more self-consistent relation between the local density and the optimal radius. 
\end{remark}
\begin{remark}
	The same particle ensemble is used both to estimate the local density in \eqref{eq:densityEst} and to evaluate the mean-shift estimator in \eqref{eq:meanShift}. Consequently, the adaptive radius and the score estimate are statistically dependent. This dependence can be reduced by sample splitting: one independent particle ensemble determines the radius, while a second ensemble evaluates the score. Independence is not required for consistency as the radius is guaranted to stay bounded; and sample splitting requires additional particles and thus increases computational cost. In the point-set registration experiment, the shared- and split-ensemble formulations produced nearly identical paths. Thus, for the particle counts considered here, the effect of this statistical dependence appears negligible. 
\end{remark}

\begin{algorithm}[H]
	\caption{Monte Carlo approximation of osmotic paths with adaptive radii}
	\label{alg:MC}
	\begin{algorithmic}[1]
		\Require Initial density $\rho_0$, target score $\nabla S_\infty$, time step $\Delta t$, number of steps $N_t$, number of Monte Carlo particles $N$, initial path
		centers $\{\mathbf{X}_{c,k}^0\}_{k=1}^{M}$, initial radius $r_{\mathrm{init}}$, radius bounds $r_{\min}$ and $r_{\max}$, and calibrated values $r_{\mathrm{ref}}$,
		$\rho_{\mathrm{ref}}$, and $N_{\mathrm{ref}}$
		\State Sample $\mathbf{X}_i^0\sim\rho_0$ for $i=1,\ldots,N$
		\For{$k=1,\ldots,M$}
			\State $r_k^{\mathrm{old}}\gets r_{\mathrm{init}}$
			\State $C_k^{\mathrm{old}}\gets\{i:\|\mathbf{X}_i^0-\mathbf{X}_{c,k}^0\|\leq r_k^{\mathrm{old}}\}$
		\EndFor
		\For{$n=0,\ldots,N_t-1$}
			\For{$k=1,\ldots,M$}
				\If{$|C_k^{\mathrm{old}}|>0$}
					\State $\widehat{\rho}_k^n\gets |C_k^{\mathrm{old}}|/[N\pi(r_k^{\mathrm{old}})^2]$
					\State $r_k^n\gets r_{\mathrm{ref}}[N_{\mathrm{ref}}\rho_{\mathrm{ref}}/(N\widehat{\rho}_k^n)]^{1/8}$
				\Else
					\State $r_k^n\gets 2r_k^{\mathrm{old}}$
				\EndIf
				\State $r_k^n\gets\min\{r_{\max},\max\{r_{\min},r_k^n\}\}$
				\State $C_k^n\gets\{i:\|\mathbf{X}_i^n-\mathbf{X}_{c,k}^n\|\leq r_k^n\}$
				\If{$|C_k^n|>0$}
					\State $\widehat{\nabla S}_k^n\gets \dfrac{4}{(r_k^n)^2|C_k^n|}\displaystyle\sum_{i\in C_k^n}(\mathbf{X}_i^n-\mathbf{X}_{c,k}^n)$
					\State $\widetilde{\mathbf{X}}_{c,k}^{n+1}\gets\mathbf{X}_{c,k}^n+\Delta t[\nabla S_\infty(\mathbf{X}_{c,k}^n)-\widehat{\nabla S}_k^n]$
					\State Reflect $\widetilde{\mathbf{X}}_{c,k}^{n+1}$ at the boundary to obtain $\mathbf{X}_{c,k}^{n+1}$
				\Else
					\State $\mathbf{X}_{c,k}^{n+1}\gets\mathbf{X}_{c,k}^{n}$
				\EndIf
				\State $r_k^{\mathrm{old}}\gets r_k^n$ and $C_k^{\mathrm{old}}\gets C_k^n$
			\EndFor
			\For{$i=1,\ldots,N$}
				\State Draw $\boldsymbol{\xi}_i^n\sim\mathcal{N}(\mathbf{0},\mathbf{I})$
				\State $\widetilde{\mathbf{X}}_i^{n+1}\gets\mathbf{X}_i^n+\Delta t\,\nabla S_\infty(\mathbf{X}_i^n)+\sqrt{2\Delta t}\,\boldsymbol{\xi}_i^n$
				\State Reflect $\widetilde{\mathbf{X}}_i^{n+1}$ at the boundary to obtain $\mathbf{X}_i^{n+1}$
			\EndFor
		\EndFor
	\end{algorithmic}
\end{algorithm}
\clearpage

\section{Numerical experiments}
\label{section:numerics}
In this section, we perform several tests to evaluate the performance of the proposed MLS method and the stochastic particle method. First, in Section \ref{section:radiusCalibration}, we empirically obtain an optimal scaling constant for the adaptive radius algorithm used to obtain the osmotic paths from the random trajectories in the Monte Carlo method. Then, in Section \ref{section:DiffusionTestCase}, we consider a simple test case with only diffusion to check the order of the MLS method and verify the behavior of the stochastic method. In Section \ref{section:ADTestCase}, we use the FPE to map two Gaussians onto one another. In Section \ref{section:gaussianAcrossCylinder}, we add an obstacle to verify the reflective boundary conditions. Finally, in Section \ref{section:PSRAcrossCylinder}, we apply both methods to a test case for point set registration \cite{iollo_point-set_2025}. The software used for all simulations and figure generation is publicly available at~\cite{Willems_Moving_Least_Squares}.

\subsection{Calibration of the adaptive-radius constant}
\label{section:radiusCalibration}
We consider an initial Gaussian \(\rho_0 = \mathcal{N}(\mu_0, \sigma_0^2 I_2)\) with \(\mu_0 = [0, 0]^\top\) and \(\sigma_0 = 1\). The exact score of the initial Gaussian is \begin{align}
	\nabla \log \rho_0(\mathbf{X}) = -\mathbf{X}. 
\end{align} To calibrate the adaptive radius function \eqref{eq:rOpt}, we proceed in the following manner. We generate 10 logarithmically spaced candidate radii between \(1 \times 10^{-2}\) and \(1\). The reference density is defined as the median of the exact initial density evaluated at the 51 region centers, \begin{align}
	\rho_{\mathrm{ref}} = \underset{k=1,\ldots,51}{\operatorname{median}}\, \rho_0(\mathbf{X}_{c,k}^0).
\end{align} It therefore represents a typical density at the locations where the score is estimated. For each candidate reference radius \(r_{\mathrm{ref}}\), the radius of the \(k\)-th region is set to \begin{align}
	r_k = r_{\mathrm{ref}} \left( \frac{\rho_{\mathrm{ref}}}{\rho_0(\mathbf{X}_{c,k}^0)} \right)^{1/8},
\end{align} in accordance with \eqref{eq:rOpt}, with \(N=N_{\mathrm{ref}}\). For every reference radius, we estimate the score of the initial Gaussian using the method from Section \ref{section:OsmoticPaths}. The error of the score is then averaged over \(20\) realizations of \(N_{\mathrm{ref}} = 10^7\) particles and 51 regions, i.e., \begin{align}
	E(r_{\mathrm{ref}})
	= \frac{1}{51 \times 20}
	\sum_{\ell=1}^{20}
	\sum_{k=1}^{51}
	\left\|
	\widehat{\nabla \log \rho_0}^{\,(\ell)}
	(\mathbf{X}_{c,k}^0; r_k)
	+ \mathbf{X}_{c,k}^0
	\right\|_2^2. \label{eq:rOptError}
\end{align} We select the candidate reference radius \(r_{\mathrm{ref}}\) that minimizes \eqref{eq:rOptError}. The proportionality constant is subsequently computed as \(\kappa=r_{\mathrm{ref}}(N_{\mathrm{ref}}\rho_{\mathrm{ref}})^{1/8}\). This gives \begin{align}
	r_{\mathrm{ref}} = 2.15443469 \times 10^{-1}, \; \rho_{\mathrm{ref}} = 7.43962967 \times 10^{-2}, \; N_{\mathrm{ref}} = 10^7, \; \kappa = 1.1675547312210792. 
\end{align} 

\subsection{Diffusion test case}
\label{section:DiffusionTestCase}

We consider an initial Gaussian \(\rho_0 = \mathcal{N}(\mu_0, \sigma_0^2 I_2)\) with \(\mu_0 = [0, 0]^\top\) and \(\sigma_0 = 1\). The target log-density \(S_{\infty}\) is set to zero so that only diffusion takes place. In this simplified case, the analytical solution for the density is given by \begin{align*}
	\rho(\mathbf{x}, t) = \frac{B}{t + A} \exp(-\frac{x^2 + y^2}{4 (t + A)}),
\end{align*} with \(A = 0.5 \sigma_0^2\) and \(B = \frac{1}{4 \pi}\). 

\subsubsection{MLS}
\label{section:DiffusionMLS}
The initial points for the MLS method are sampled from \(\rho_0\) in the following way. First, \(N\) quasi-random points on the domain \([0, 1]^2\) are generated using a van der Corput sequence \cite{niederreiter_random_1992}. These points are then mapped to the distribution \(\rho_0\) using the inverse transform sampling method \cite{caflisch_monte_1998}. In Figure \ref{fig:convergenceInitialDistribution}, the initial distribution of points for \(N = 254\) is plotted. The simulation is performed with an increasing number of particles up to time \(t_{\rm f} = 10\). The neighborhoods of particles consist of \(N_b = 16\) points. Since the characteristic interparticle distance in two dimensions scales as \(h\sim N^{-1/2}\), we refine the maximum time step according to \begin{align*}
	\Delta t_{\max}=\frac{4}{\sqrt{N}}=O(h).
\end{align*} The eigenvalue-based adaptive criterion does not impose a smaller time step in this test. Because the logarithm of the Gaussian solution remains quadratic, its spatial derivatives are reproduced exactly by the quadratic MLS approximation, up to numerical errors. The observed errors are therefore dominated by the temporal discretization. Accordingly, the Euler and midpoint methods exhibit errors of order \(O(\Delta t)=O(N^{-1/2})\) and \(O(\Delta t^2)=O(N^{-1})\), respectively. The error at final time is measured in the relative \(L_1\) norm \begin{align}
	L_{1, \rho} = \frac{ \sum_{i=1}^N \abs{\exp S_i - \rho(\mathbf{x}_i, t)}}{\sum_{i=1}^N \abs{\rho(\mathbf{x}_i, t)}}. \label{eq:L1DensityError}
\end{align} 
The convergence graph for the first-order Euler method and the second-order IMEX midpoint method is shown in Figure \ref{fig:convergenceDiffusion}. We observe first- and second-order convergence for the Euler and midpoint methods, respectively, consistent with their temporal discretization errors. In Figure \ref{fig:convergenceDiffusionParticlePaths}, the paths of several particles are plotted. As expected, we observe that the particles travel on straight lines outward from the origin. 
\begin{figure}
	\centering
	\begin{subfigure}[b]{0.49\textwidth}
		\centering
		\includegraphics[width=\linewidth]{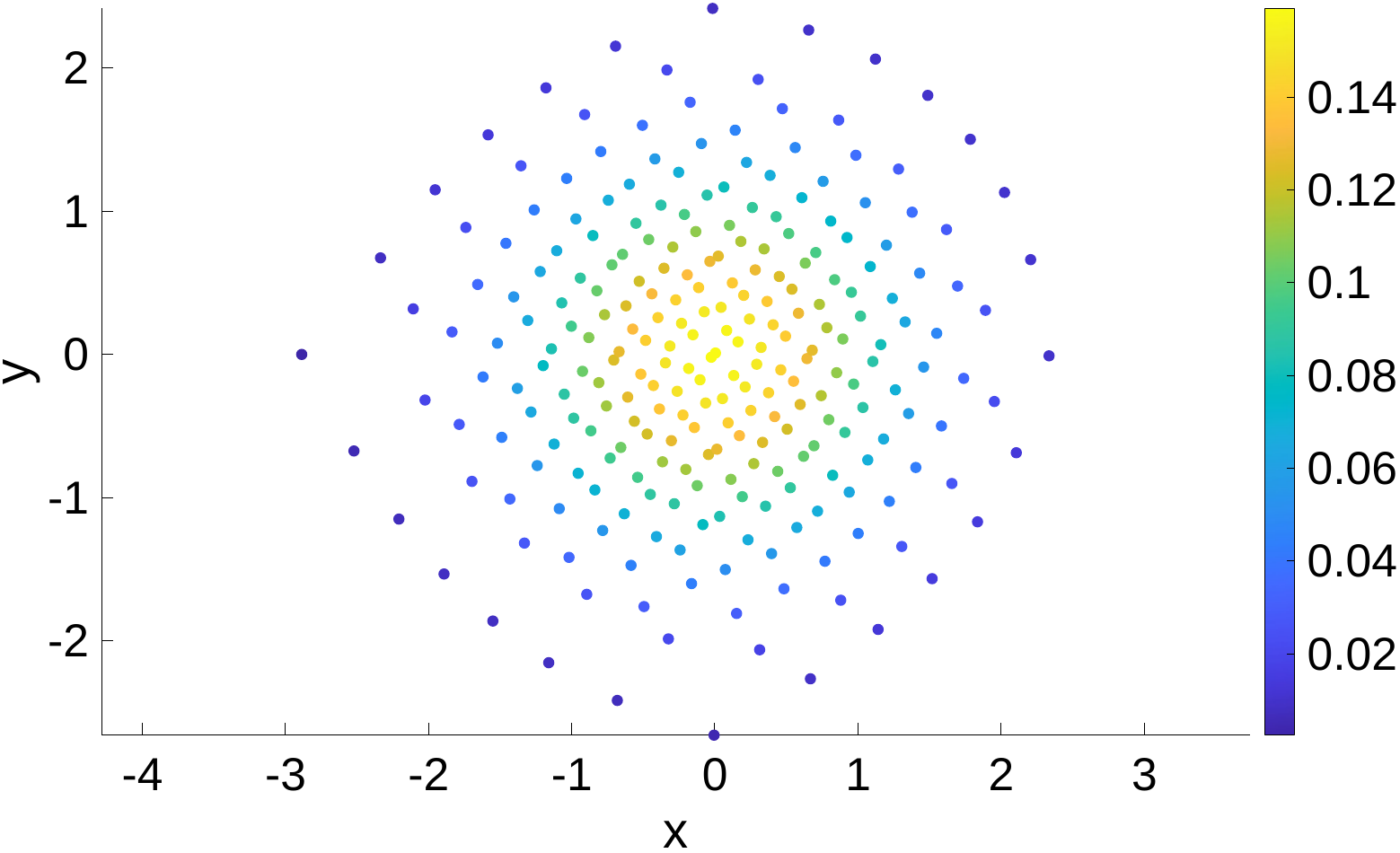}
		\caption{Initial distribution of points for \(N = 254\).}
		\label{fig:convergenceInitialDistribution}
	\end{subfigure}
	\begin{subfigure}[b]{0.49\textwidth}
		\centering
		\includegraphics[width=\linewidth]{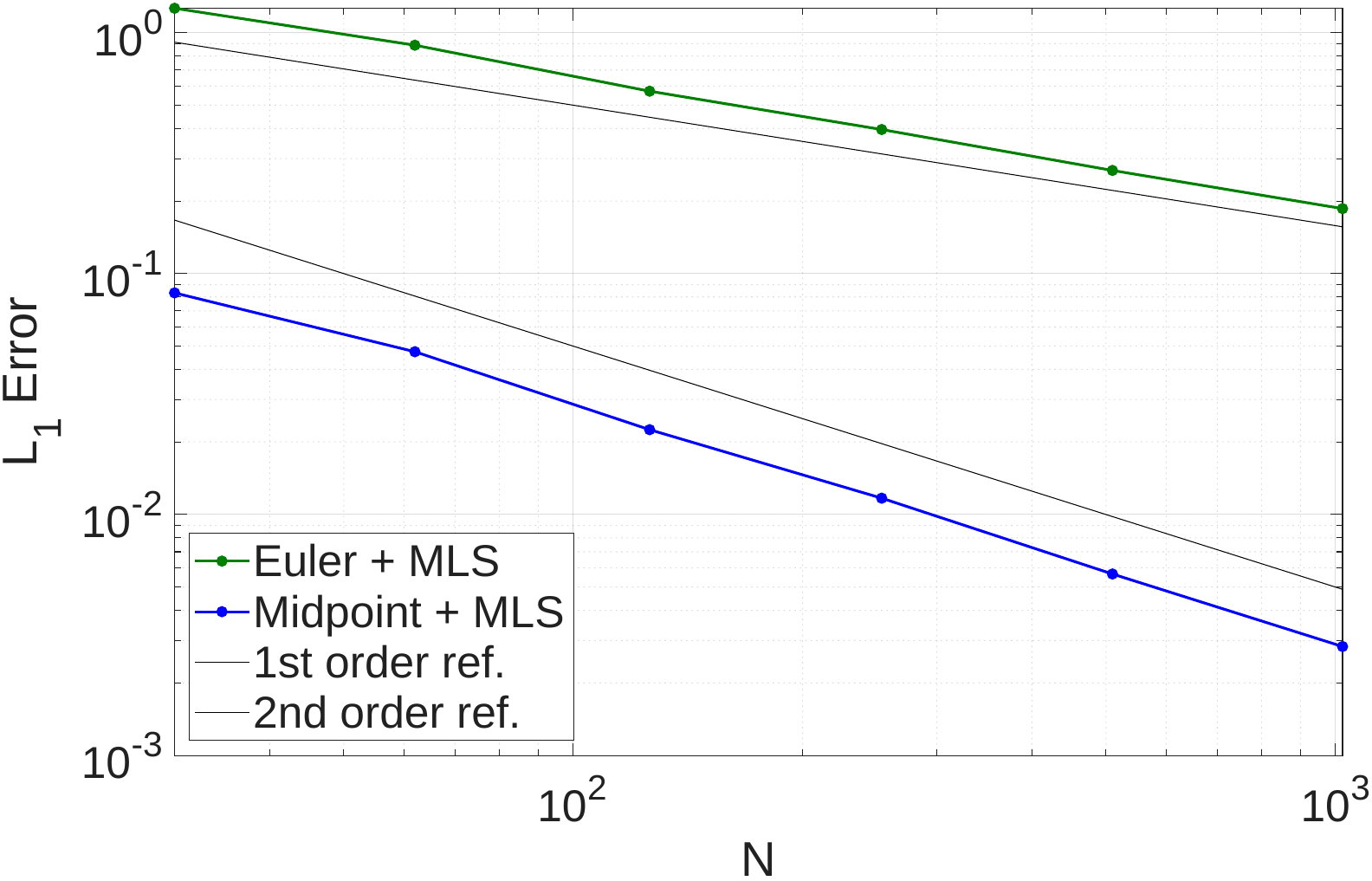}
		\caption{Convergence graph. }
		\label{fig:convergenceDiffusion}
	\end{subfigure}
	\begin{subfigure}[b]{0.49\textwidth}
		\centering
		\includegraphics[width=\linewidth]{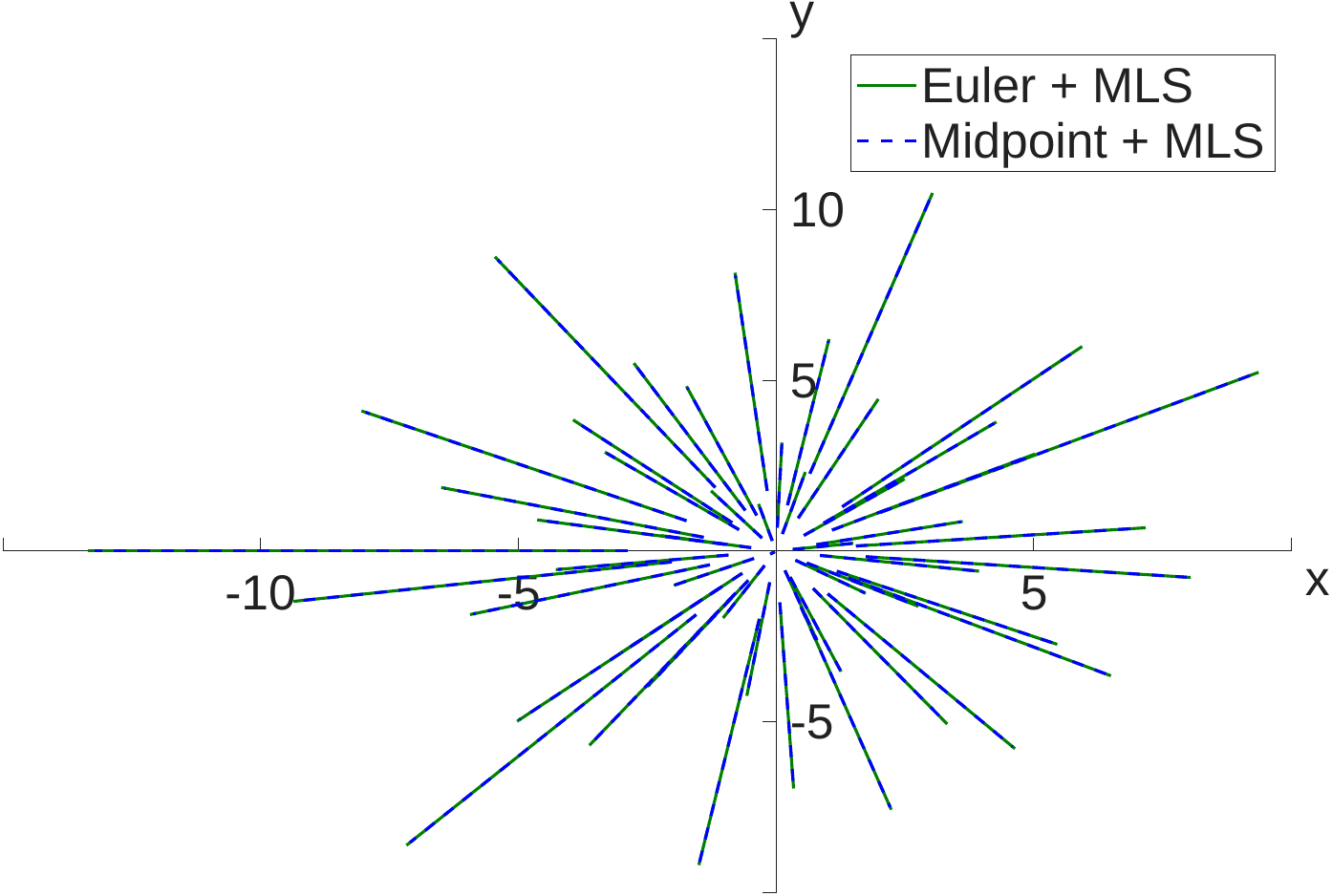}
		\caption{The trajectories of several particles for \(N = 254\).}
		\label{fig:convergenceDiffusionParticlePaths}
	\end{subfigure}
	\caption{Convergence analysis of the MLS scheme for diffusion of a Gaussian.}
	\label{fig:convergence1}
\end{figure}

\subsubsection{Stochastic method}
\label{section:DiffusionMC}
We now perform the same test case with the Monte Carlo method described in Section \ref{section:MC}. One million particles are randomly sampled from the initial distribution. These Monte Carlo particles are then advanced using the discretized SDE for 200 time steps, with \(\Delta t = 1 \times 10^{-2}\), up to time 2. To obtain the osmotic paths, we place several small regions of fixed radius \(r_{\mathrm{init}} = 0.13\) at the same locations as in the MLS simulation, see Figure \ref{fig:convergenceDiffusionParticlePaths}. Initially, the small regions contain between 494 and 31,968 particles, depending on the location of the small region compared to the center of the initial distribution. The osmotic paths are then obtained by advancing the location of the small regions using the approximation of the score, see Section \ref{section:OsmoticPaths}. The locations of the particles and the osmotic paths are plotted in Figure \ref{fig:DiffusionTestMC1}. As before, we observe the small regions move on radially outward trajectories, with the expected Monte Carlo noise. Initially, some small regions lie at the edge of the particle cloud and therefore contain few particles, resulting in noisy paths. Then, as time evolves, the average number of particles per region decreases, and the same \(10^6\) particles spread out over a larger area. Thus, as time increases, the paths also become noisier. In Figure \ref{fig:DiffusionTestMC2}, the same simulation is performed with the adaptive radius method from Section \ref{section:MCAdaptiveRadius}. Clearly, the adaptive radius selection is able to balance the bias and Monte Carlo error, significantly improving the results. 
\begin{figure}
	\centering
	\begin{subfigure}[b]{0.49\textwidth}
		\centering
		\includegraphics[width=\linewidth]{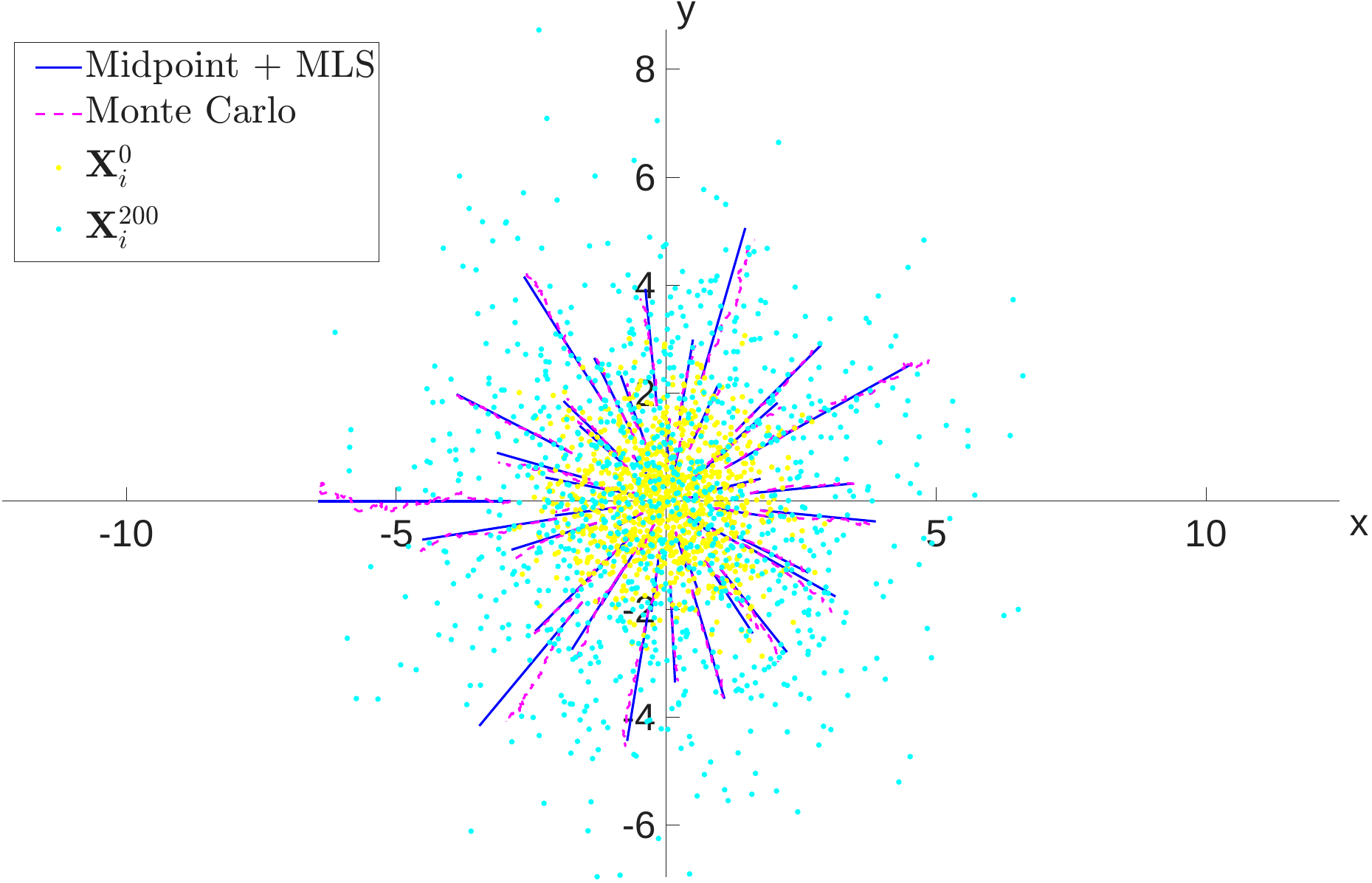}
		\caption{Fixed radius.}
		\label{fig:DiffusionTestMC1}
	\end{subfigure}
	\begin{subfigure}[b]{0.49\textwidth}
		\centering
		\includegraphics[width=\linewidth]{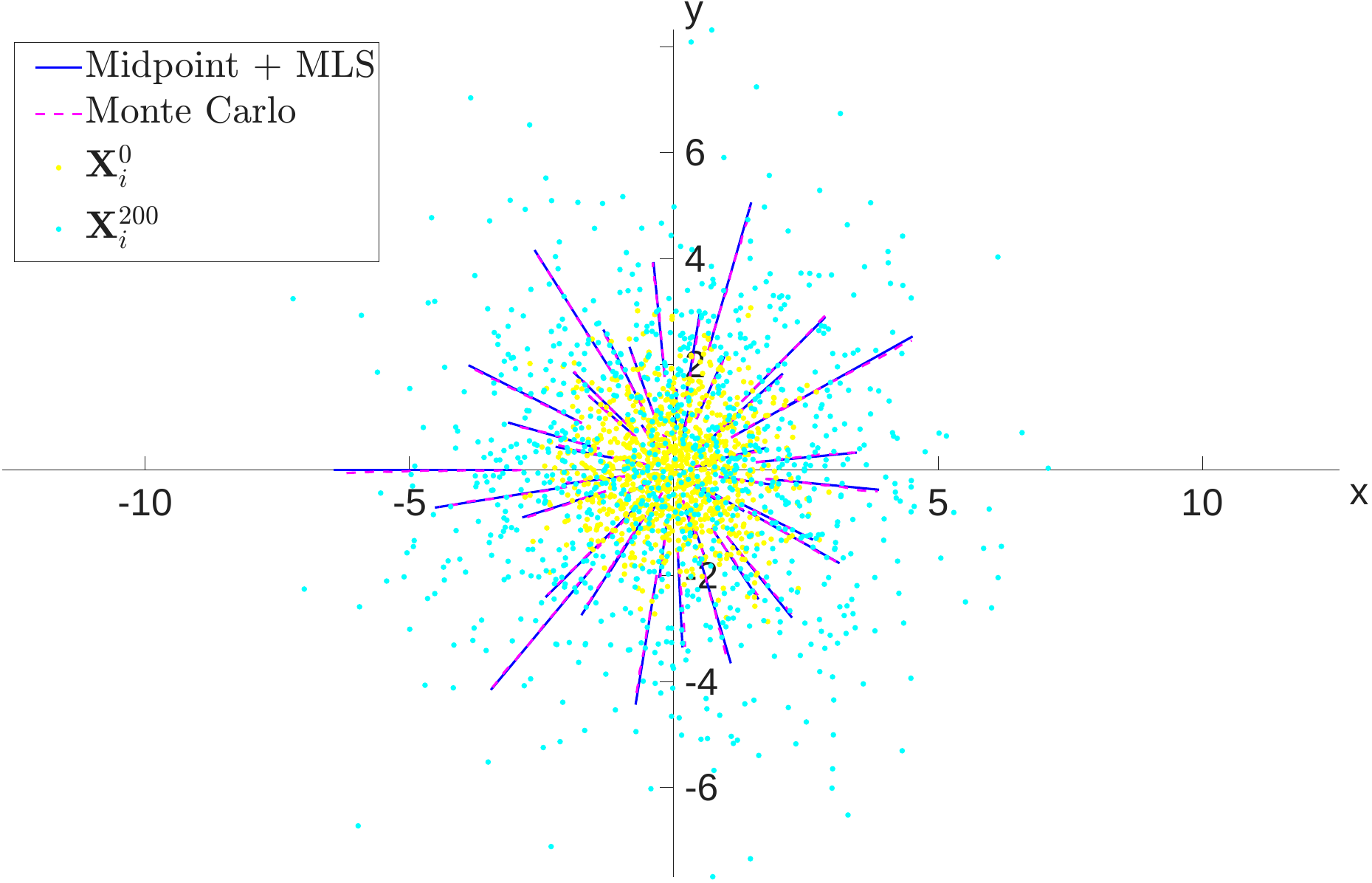}
		\caption{Adaptive radius.}
		\label{fig:DiffusionTestMC2}
	\end{subfigure}
	\caption{The initial distribution of particles \(\mathbf{X}^{0}_i\), final distribution of particles \(\mathbf{X}^{200}_i\), and the osmotic paths obtained by advancing the location of the small regions. To make the figures more readable, we only plot \(10^3\) Monte Carlo particles. For comparison, we plot the paths obtained using the MLS particle method that are also given in Figure \ref{fig:convergenceDiffusionParticlePaths}. }
	\label{fig:DiffusionTestMC}
\end{figure}

\subsection{Advection-diffusion test case}
\label{section:ADTestCase}
We now consider a second scenario in which a Gaussian with \(\mu_0 = [2, 2]^\top\) and \(\sigma_0 = 0.3\) is mapped onto a Gaussian at the origin with standard deviation \(\sigma_{\infty} = 2\). Due to the difference in means of the distributions, we now also expect advection to take place. In this case, the exact solution is given by a Gaussian with mean \(\mu(t)\) and covariance matrix \(\Sigma(t)\) \cite{kampen_stochastic_2007} \begin{align*}
	\mu(t) &= \exp \left(-\frac{t}{\sigma_{\infty}^2}\right) \mu_0, \\
	\Sigma(t) &= \left[ c(t) \sigma_0^2 + \left(1 - c(t)\right) \sigma_{\infty}^2 \right] I_2,
\end{align*} with \(c(t) = \exp \left(-\frac{2t}{\sigma_{\infty}^2}\right)\). The simulations are performed up to time \(t_f = 6.5\).

\subsubsection{MLS}
\label{section:ADMLS}
As in the previous example, we use quasi-randomly distributed points and keep the number of neighbors \(N_b = 16\). The simulation is carried out on grids of progressively finer resolution. Similarly to before, the time step is refined according to \begin{align} 
	\Delta t_{\max}=\frac{2}{\sqrt{N}}=O(h).
\end{align} In addition to reporting the error of the density at the final positions \eqref{eq:L1DensityError}, we also plot the relative \(L_1\) error in the particle positions. \begin{align}
	L_{1,x} = \frac{\sum_{i=1}^N \abs{x_i^n - x_i(t^n)}}{\sum_{i=1}^N \abs{x_i(t^n)}}. \label{eq:L1PositionError}
\end{align} The true final position \(x_i\) can be computed by integrating the characteristic equation \eqref{eq:characteristic}, using the exact solution of the density, with a higher-order Runge-Kutta method. The convergence graph is given in Figure \ref{fig:convergenceAD}. The methods again achieve first-order and second-order for the density and \(x\)-coordinate. The error in the \(y\)-coordinate is omitted because it behaves identically to the error in the \(x\)-coordinate. In Figure \ref{fig:AdvectionDiffusionPathsMLS}, the initial distribution \(\rho_0\) and final distribution \(\rho_{\infty}\) are plotted, as well as several particle paths. We plot both the paths of the particles during simulations, and the paths obtained by integrating the characteristic equation \eqref{eq:characteristic} with the exact solution. The midpoint method is also visually more accurate.
\begin{figure}
	\centering 
	\begin{subfigure}[b]{0.49\textwidth}
		\centering
		\includegraphics[width=0.99\linewidth]{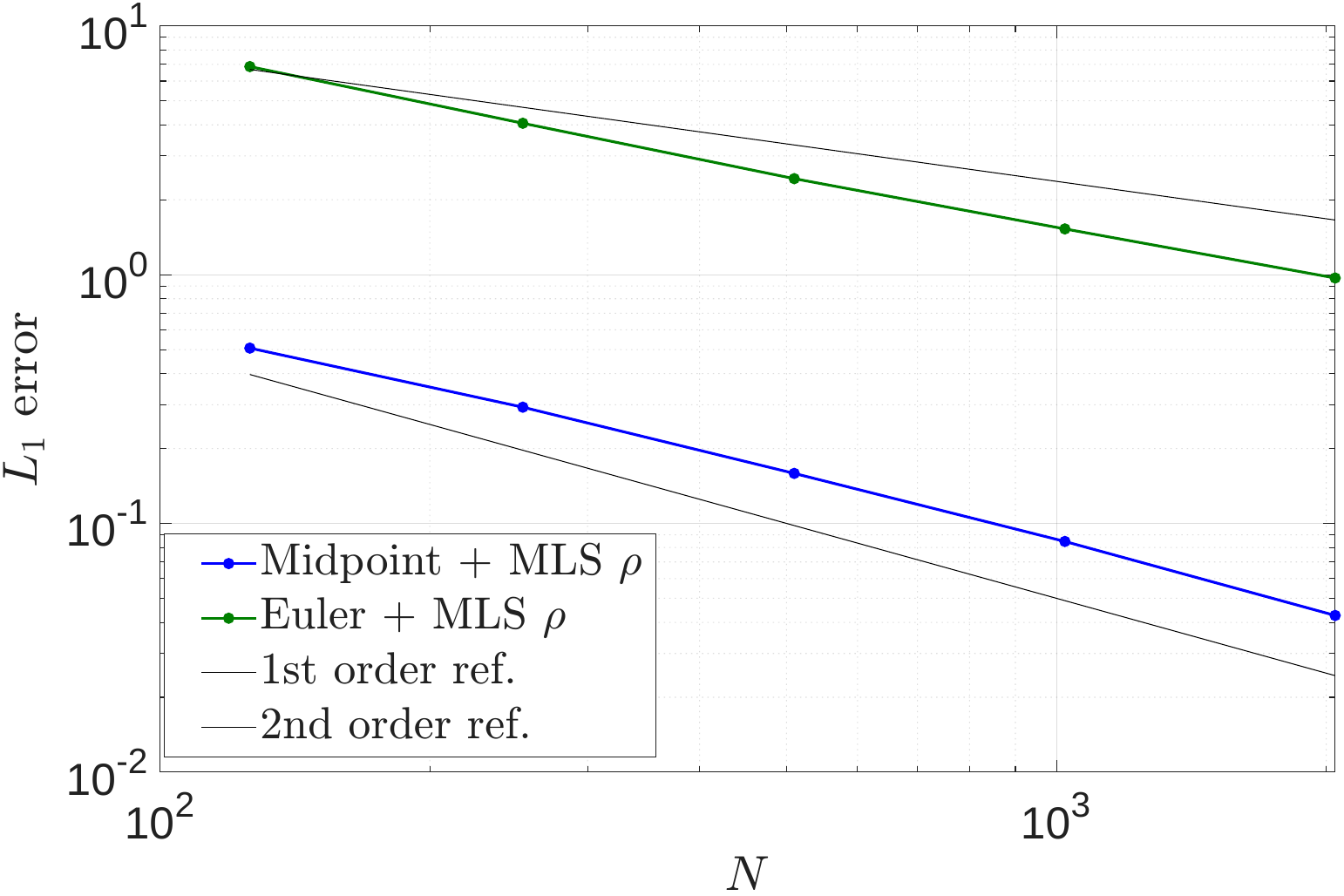}
		\caption{}
	\end{subfigure}
	\begin{subfigure}[b]{0.49\textwidth}
		\centering
		\includegraphics[width=0.99\linewidth]{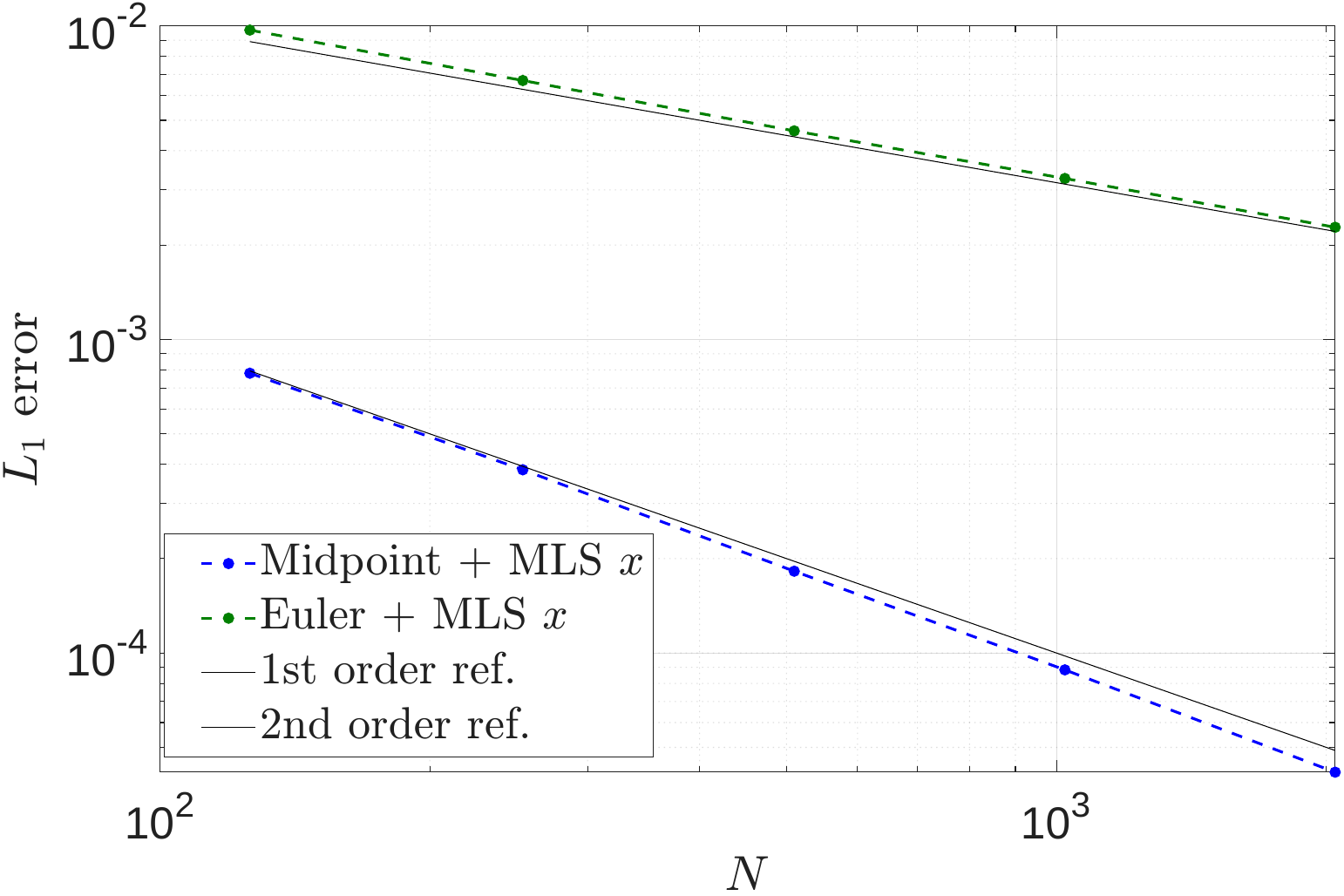}
		\caption{}
	\end{subfigure}
	\caption{Convergence for the density (left) and final position (right) of the particles for the MLS scheme.}
	\label{fig:convergenceAD}
\end{figure}

\subsubsection{Stochastic method}
\label{section:ADMC}
For comparison, we repeat the test case with the Monte Carlo particle method. Ten million Monte Carlo particles are sampled from the initial condition and advanced through time with time step \(\Delta t = 1\times10^{-2}\). We place several small regions with initial radius \(r_{\mathrm{init}} = 0.1\) at the same initial locations as in the MLS simulation, see Figure \ref{fig:AdvectionDiffusionPathsMLS}. Throughout the simulation, the radii are adapted to balance the bias and Monte Carlo error. The results are plotted in Figure \ref{fig:AdvectionDiffusionPathsMC}. We observe good agreement between the paths generated by the Monte Carlo method and those computed by integrating the osmotic velocity with an expression for the exact density distribution. 
\begin{figure}
	\centering 
	\begin{subfigure}[b]{0.49\textwidth}
		\centering
		\includegraphics[width=\linewidth]{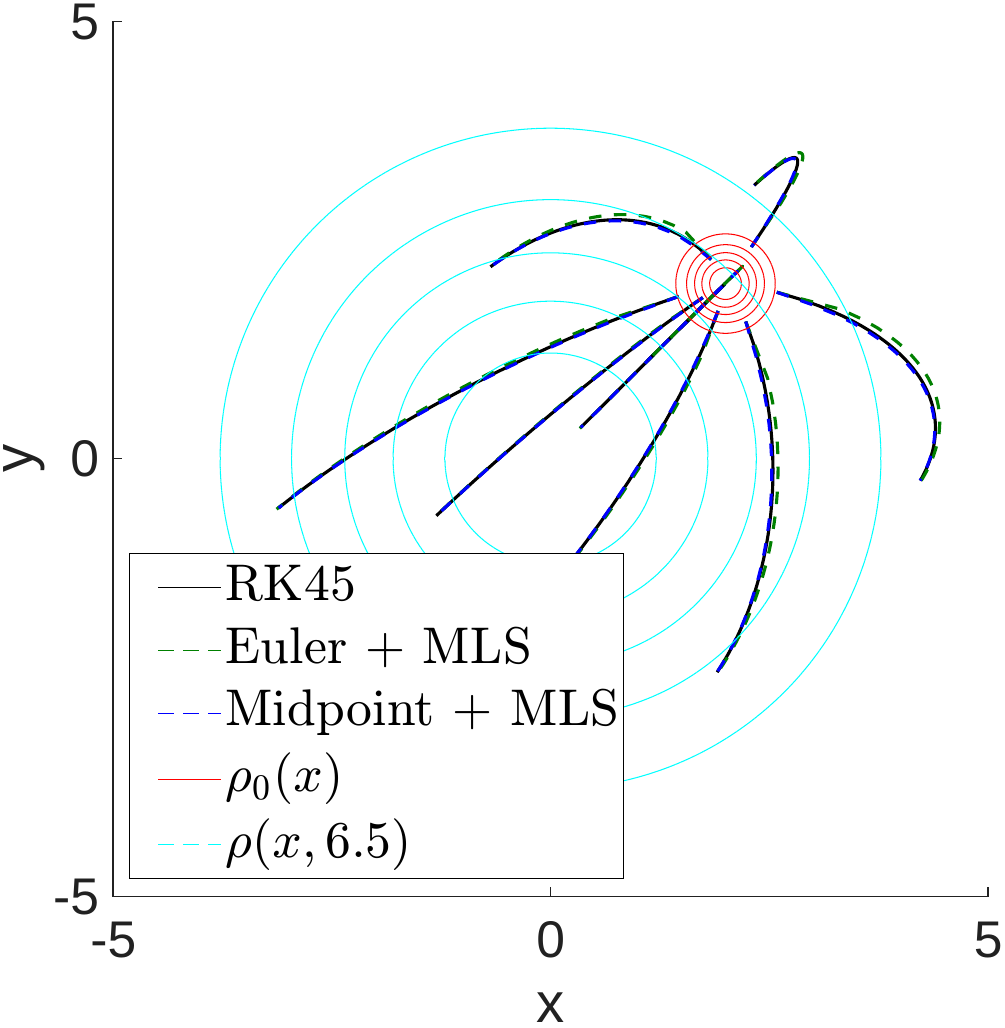}
		\caption{The trajectories of several MLS particles for \(N = 254\).}
		\label{fig:AdvectionDiffusionPathsMLS}
	\end{subfigure}	
	\begin{subfigure}[b]{0.49\textwidth}
		\centering
		\includegraphics[width=\linewidth]{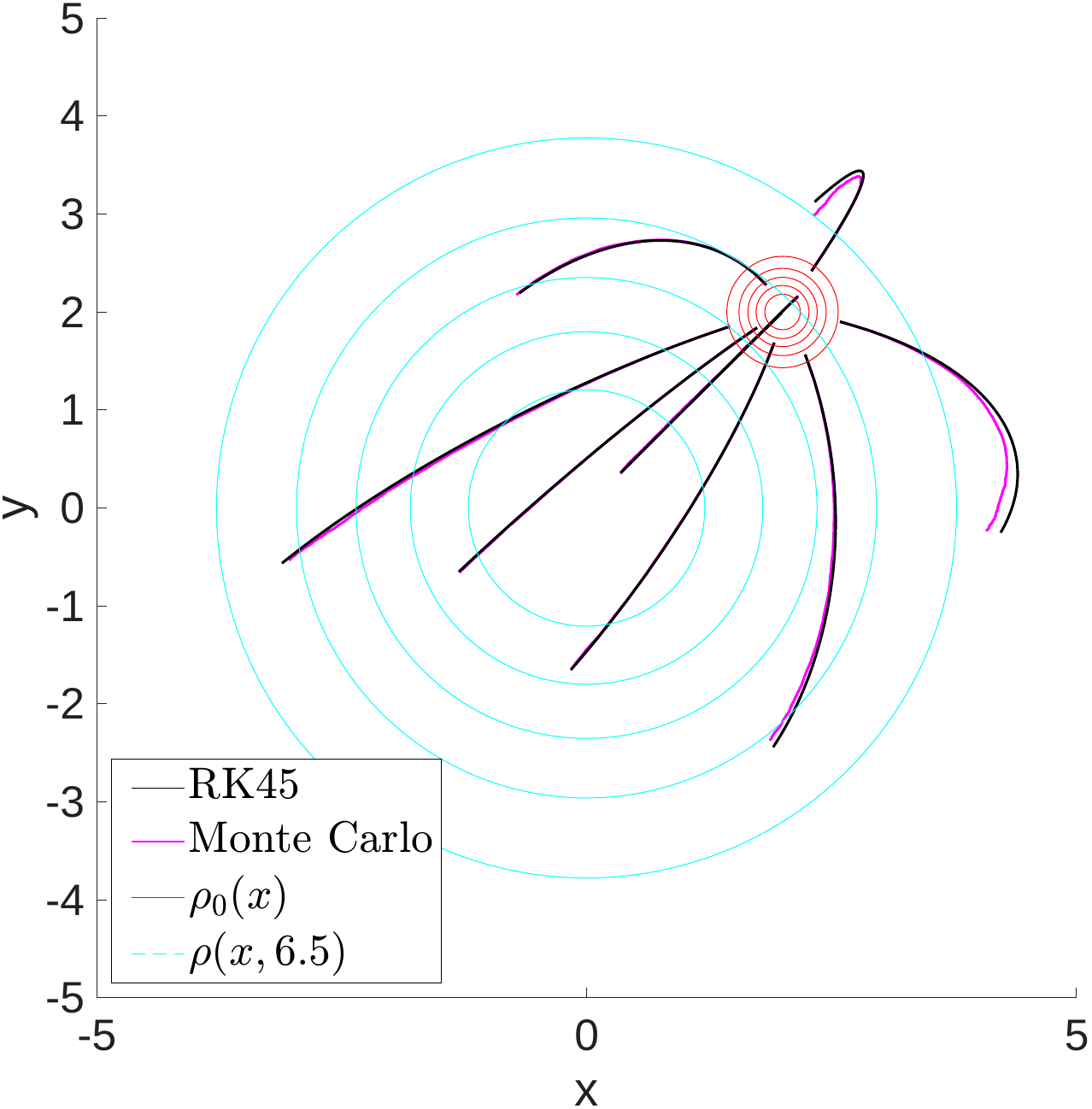}
		\caption{The trajectories generated by the Monte Carlo method.}
		\label{fig:AdvectionDiffusionPathsMC}
	\end{subfigure}
	\caption{Several trajectories generated by the MLS particle method (left) and Monte Carlo particle method (right).}
	\label{fig:pathsAD}
\end{figure}

\subsection{Transport of a Gaussian across a cylinder}
\label{section:gaussianAcrossCylinder}

\subsubsection{Problem setup}
\label{section:gaussianAcrossCylinderSetup}
We consider a test case from \cite{iollo_point-set_2025} in which a Gaussian is transported to another Gaussian with a circular obstacle between them. The initial and final distributions are given by \begin{align*}
	\rho_0 &= \mathcal{N}(\mu_0, \sigma_0^2 I_2) \quad \mu_0 = \left[-2, 0\right]^\top, \; \sigma_0^2 = 0.1, \\
	\rho_\infty &= \mathcal{N}(\mu_\infty, \sigma_0^2 I_2) \quad \mu_\infty = \left[2, 0\right]^\top. 
\end{align*} The circular obstacle has radius \(R = 0.5\) and is placed at the origin, \(\mathbf{X}_o = \left[0, 0\right]^\top\). The introduction of the obstacle allows us to test how the MLS and Monte Carlo methods handle boundary conditions. Moreover, it provides motivation for using the FPE for point set registration, as solving the FPE on a bounded domain is simpler than addressing optimal transport problems in the presence of obstacles. We perform the simulation with the MLS particle method and the stochastic particle method up to the final time \(t_{\rm f} = 0.7\). 

\subsubsection{MLS}
\label{section:gaussianAcrossCylinderMLS}
The initial distribution of particles is sampled using quasi-random numbers from the initial distribution \(\rho_0\). The simulation is performed with time step \(\Delta t_{\max} = 3 \times 10^{-4}\). We use \(N = 254\) particles. Each particle is assigned \(N_b = 20\) neighbors in every time step. The simulation is performed with the IMEX midpoint method, as it is slightly more accurate than the IMEX Euler method. During the simulation, the adaptive time-stepping criterion reduces the time step when the particle cloud becomes irregular upon interacting with the boundary. The distribution of the particles and the ghost particles at several moments in time are plotted in Figure \ref{fig:example1MLS1}. We observe that the particle distribution splits into two groups, one of which moves along the top of the obstacle, the other along the bottom. As the particle cloud hits the boundary of the obstacle, we observe a significant increase in density. The ghost point method correctly enforces the reflective boundary condition.
Behind the obstacle, the groups merge into the final distribution. While the particles were initially distributed according to $\rho_0$,  at the final time the particle distribution is slightly distorted. 
We note that the small asymmetries between the upper and lower halves of the distribution are due to the initial asymmetric particle distribution.

Some particle paths are plotted in Figure \ref{fig:pathsExample1a}. We observe a flow of particles around the boundary of the obstacle. One particle that approaches the obstacle boundary at a perpendicular angle, somewhat abruptly turns up to avoid the boundary. 

\begin{figure}
	\centering 
	\begin{subfigure}[b]{0.49\textwidth}
		\centering
		\includegraphics[width=\linewidth]{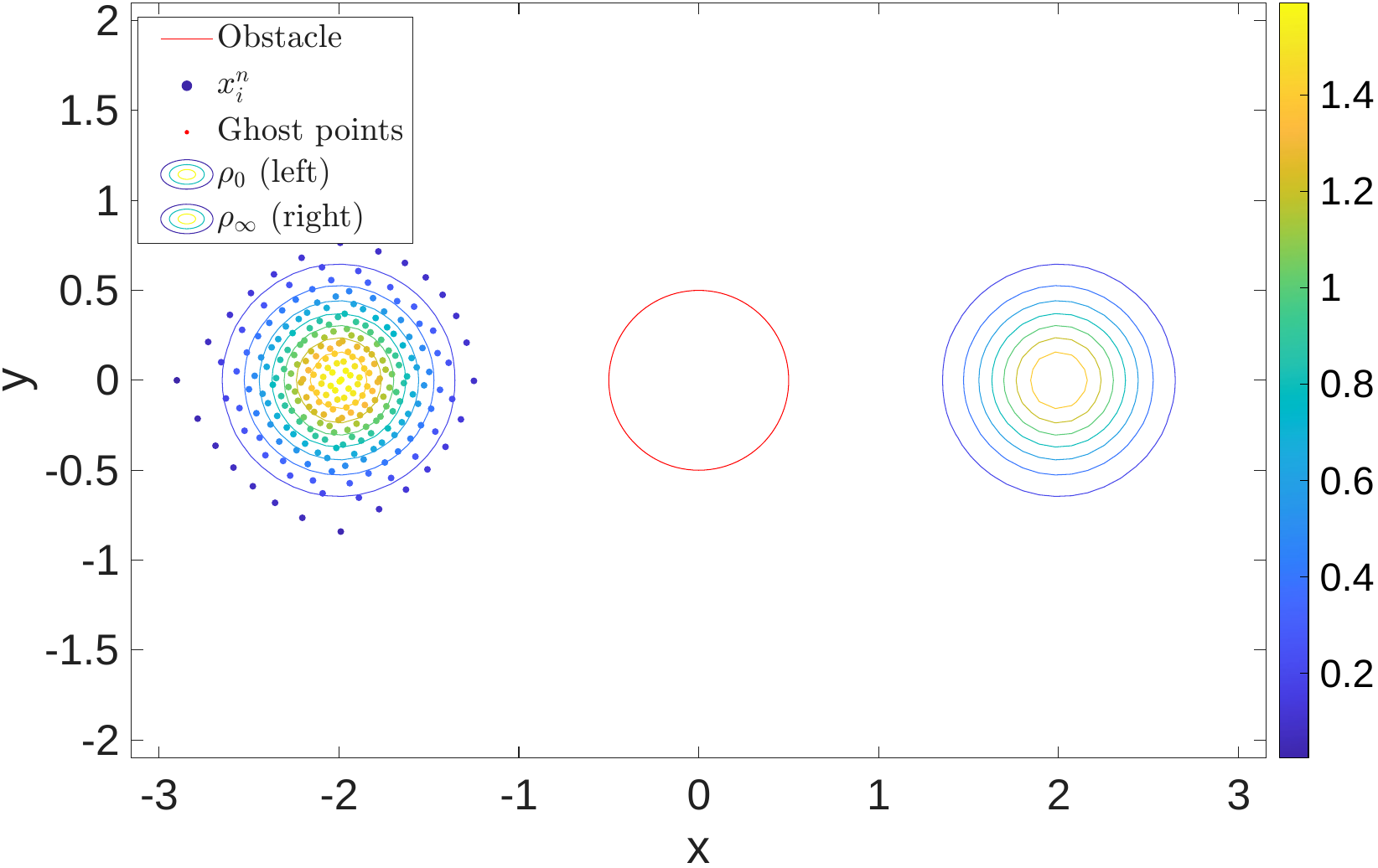}
		\caption{\(t = 0.0\).}
		\label{fig:example1a}
	\end{subfigure}
	\begin{subfigure}[b]{0.49\textwidth}
		\centering
		\includegraphics[width=\linewidth]{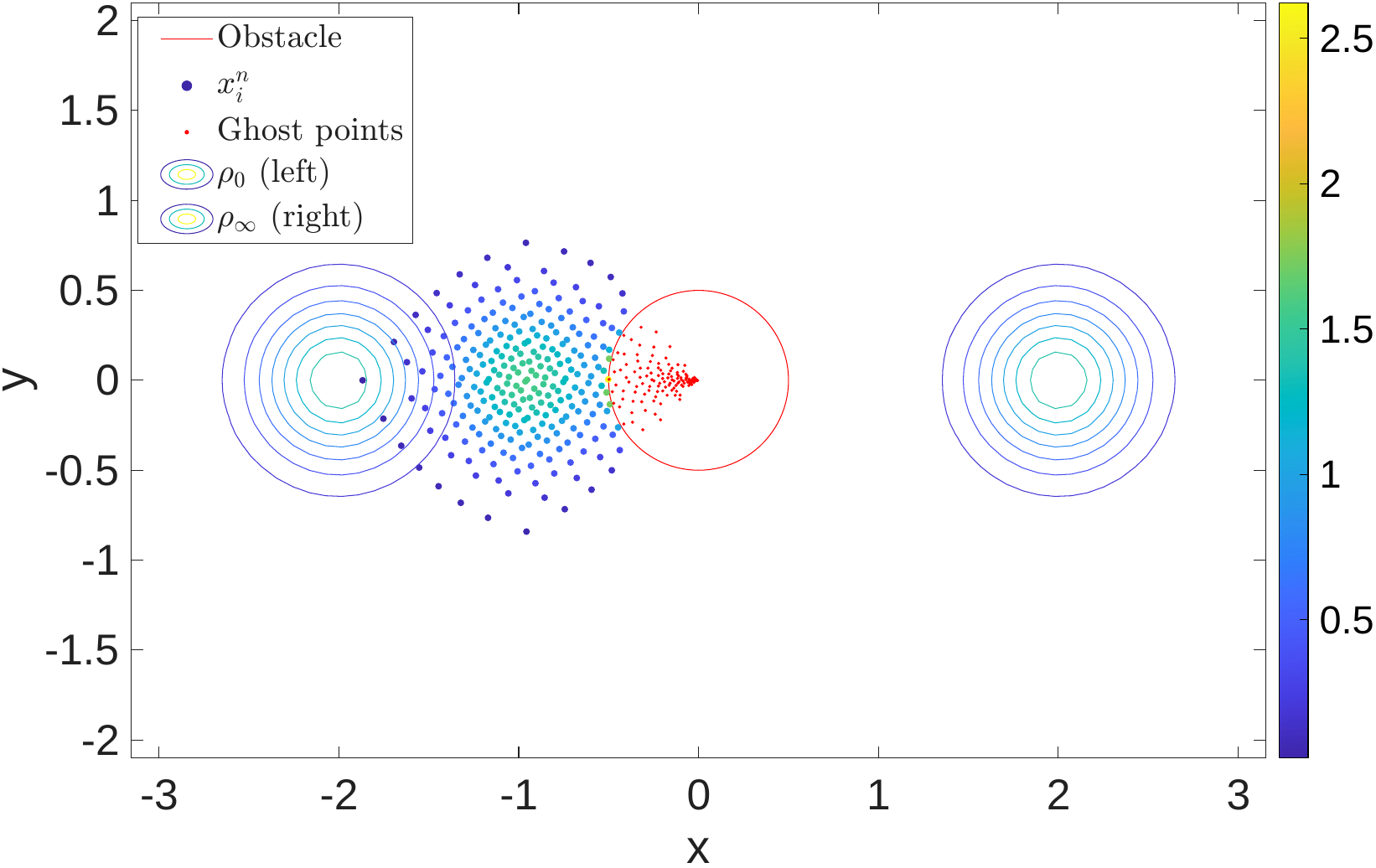}
		\caption{\(t = 3.0 \times 10^{-2}\).}
		\label{fig:example1b}
	\end{subfigure}	
	\begin{subfigure}[b]{0.49\textwidth}
		\centering
		\includegraphics[width=\linewidth]{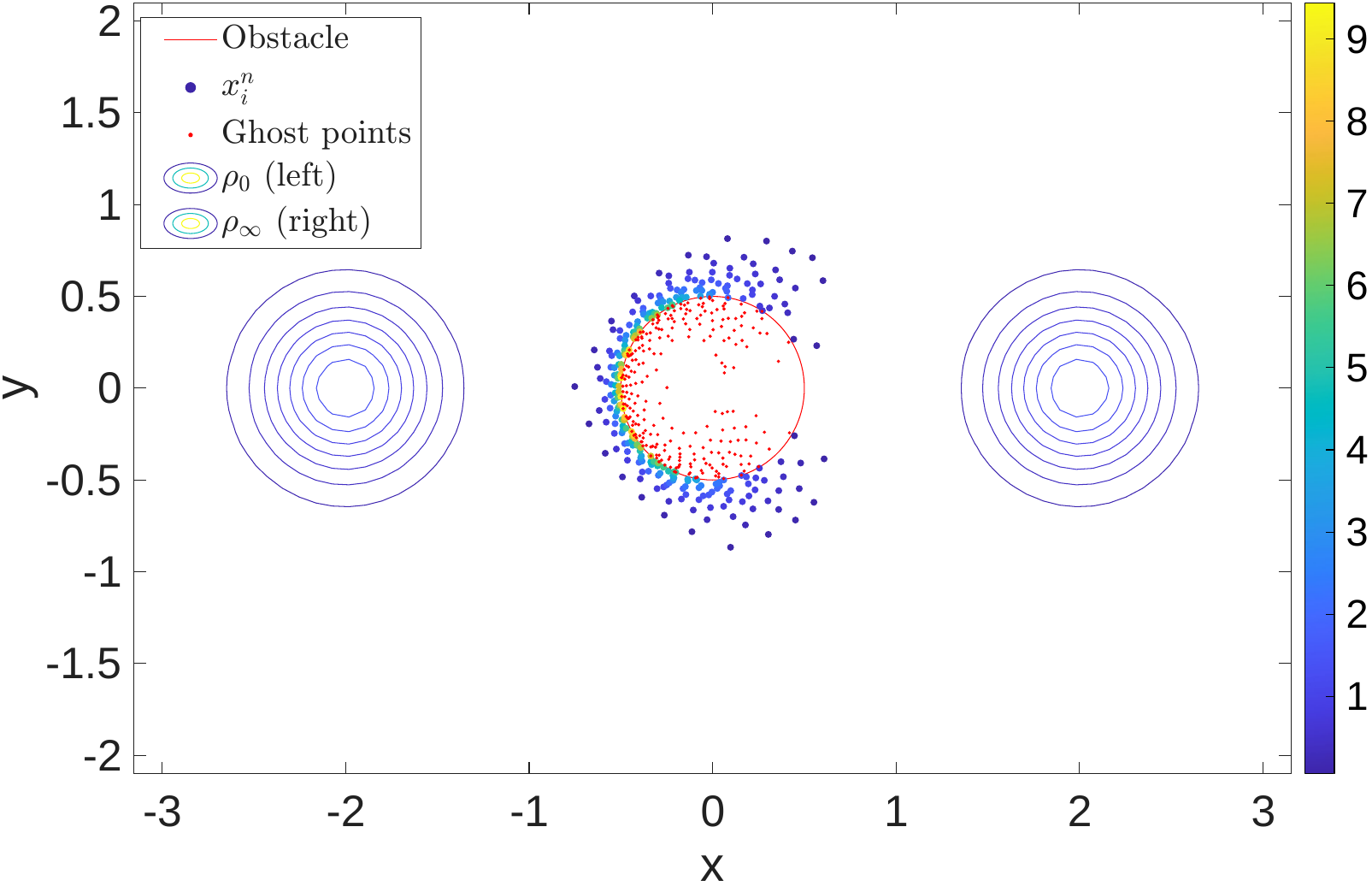}
		\caption{\(t = 7.5 \times 10^{-2}\).}
		\label{fig:example1c}
	\end{subfigure}
	\begin{subfigure}[b]{0.49\textwidth}
		\centering
		\includegraphics[width=\linewidth]{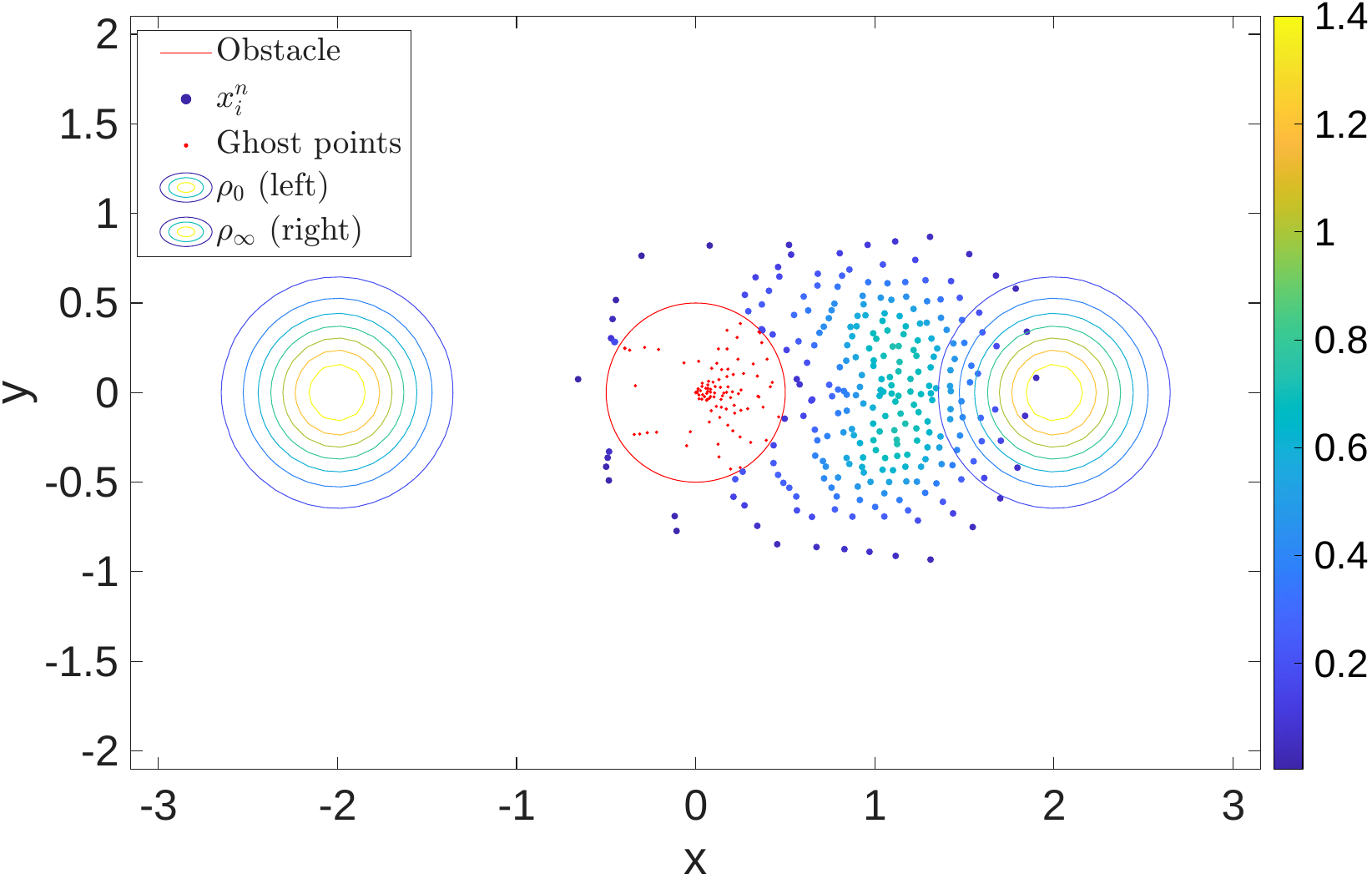}
		\caption{\(t = 1.65 \times 10^{-1}\).}
		\label{fig:example1d}
	\end{subfigure}
	\begin{subfigure}[b]{0.49\textwidth}
		\centering
		\includegraphics[width=\linewidth]{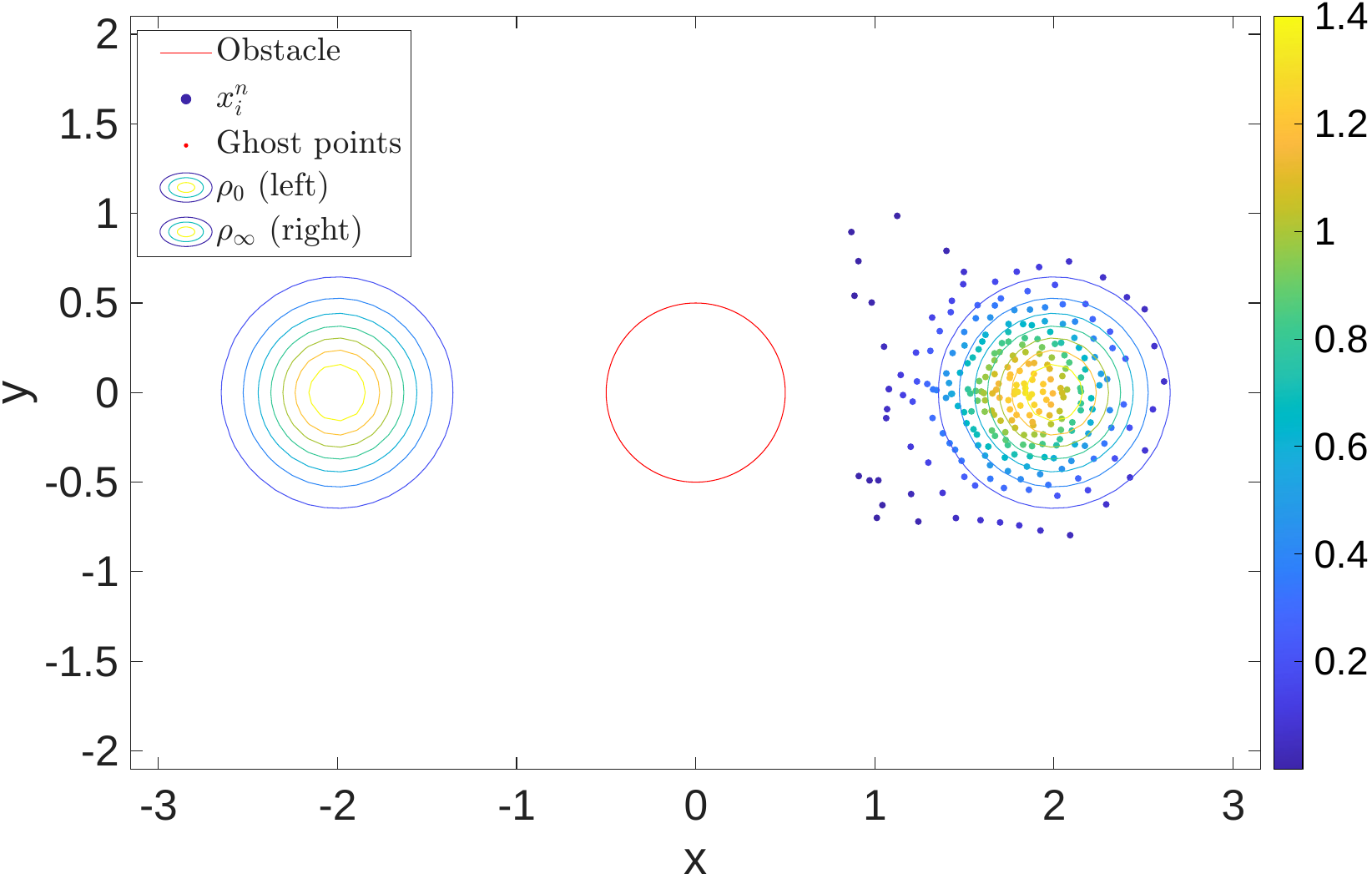}
		\caption{\(t = 3.45 \times 10^{-1}\).}
		\label{fig:example1e}
	\end{subfigure}
	\begin{subfigure}[b]{0.49\textwidth}
		\centering
		\includegraphics[width=\linewidth]{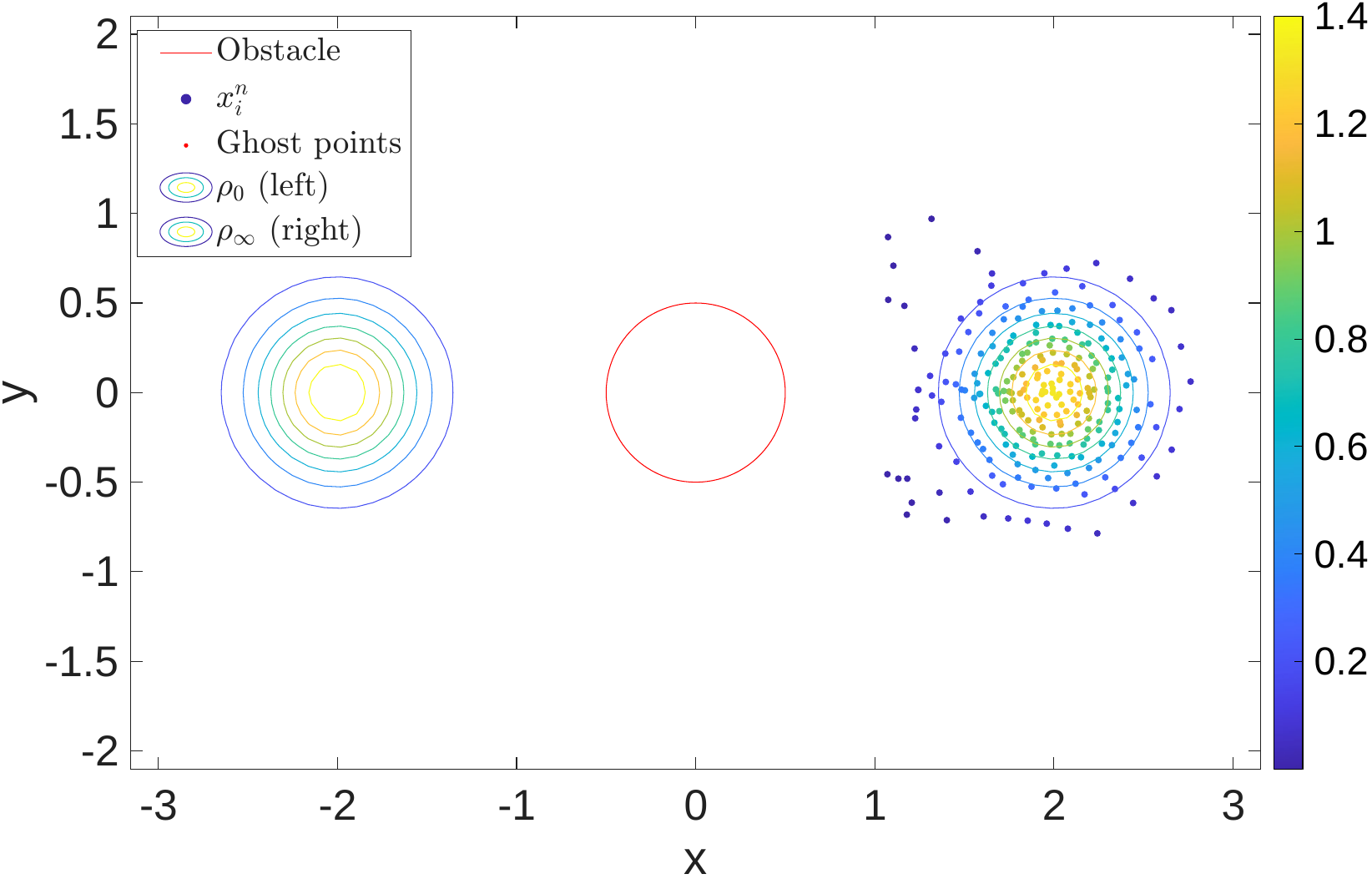}
		\caption{\(t = 7 \times 10^{-1}\).}
		\label{fig:example1f}
	\end{subfigure}
	\caption{Mapping of a Gaussian onto another Gaussian across an obstacle using the MLS method. We plot the initial distribution \(\rho_0\), the target distribution \(\rho_{\infty}\), the cloud of particles, and the ghost points.}
	\label{fig:example1MLS1}
\end{figure}

\subsubsection{Stochastic method}
\label{section:gaussianAcrossCylinderMC}
The simulation uses \(40 \times 10^6\) particles. Because steep particle-density gradients develop near the obstacle, we use a relatively small time step of \(\Delta t = 5 \times 10^{-3}\). The radii of the small regions are initialized as \(r_{\mathrm{init}} = 6 \times 10^{-2}\) and subsequently updated using the adaptive-radius procedure described in Section \ref{section:MCAdaptiveRadius}. Figure \ref{fig:example1MLS1MC} shows the Monte Carlo particle distribution at several times. The particle cloud encounters the obstacle, stretches around it, and ultimately detaches from the boundary before converging to the target distribution. The corresponding osmotic paths are shown in Figure \ref{fig:pathsExample1b}. Near the obstacle, we observe discontinuities in the osmotic paths as the small regions hit the obstacle. This is due to the high density of particles around the obstacle. This very large density gradient, computed through the mean shift vector, briefly overcomes the drift to the target distribution and pushes the small regions back. For some small regions, such as the one that hits the obstacle at a perpendicular angle, this process repeats several times.

\begin{figure}
	\centering 
	\begin{subfigure}[b]{0.49\textwidth}
		\centering
		\includegraphics[width=\linewidth]{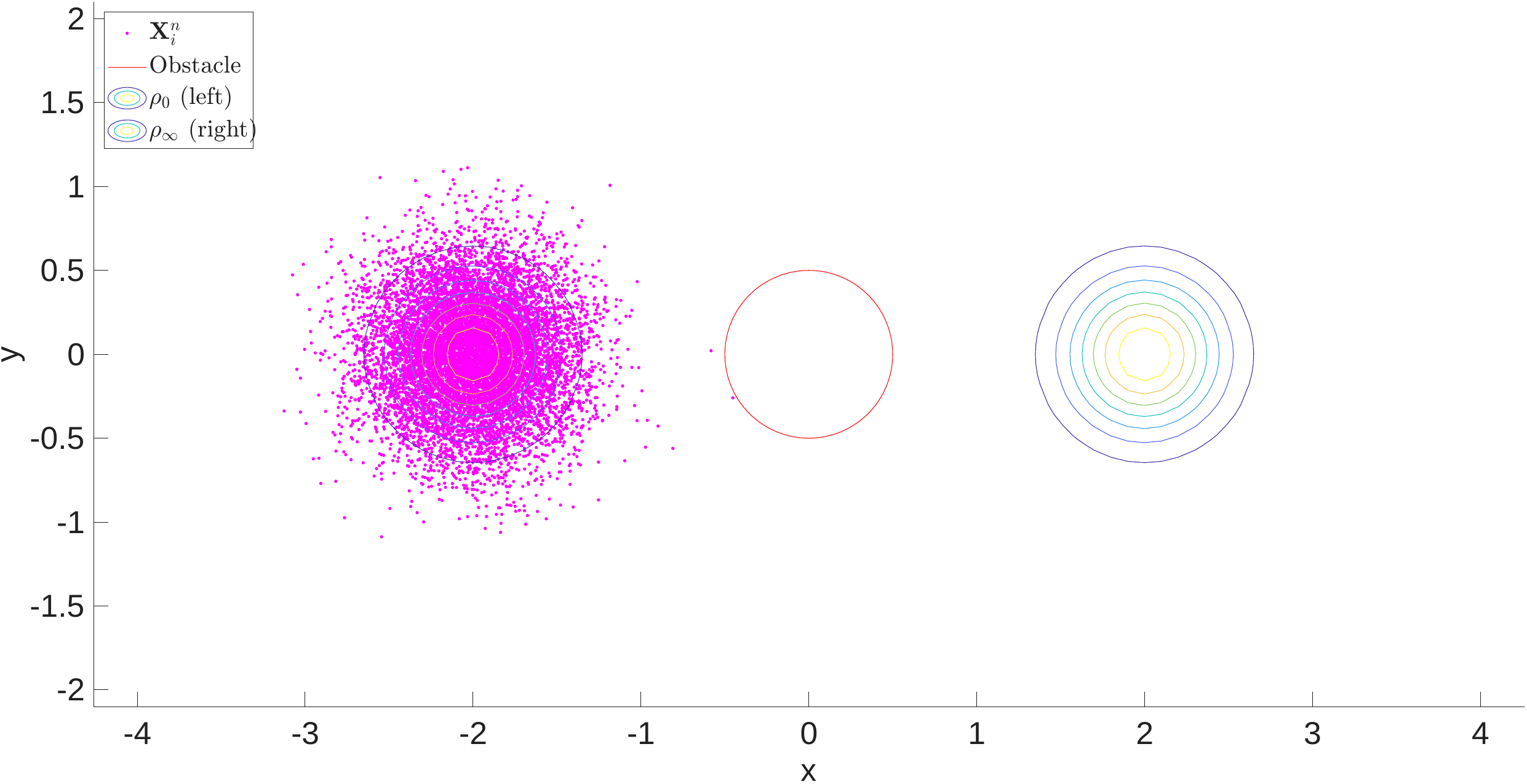}
		\caption{\(t = 0.0\).}
		\label{fig:example1aMC}
	\end{subfigure}
	\begin{subfigure}[b]{0.49\textwidth}
		\centering
		\includegraphics[width=\linewidth]{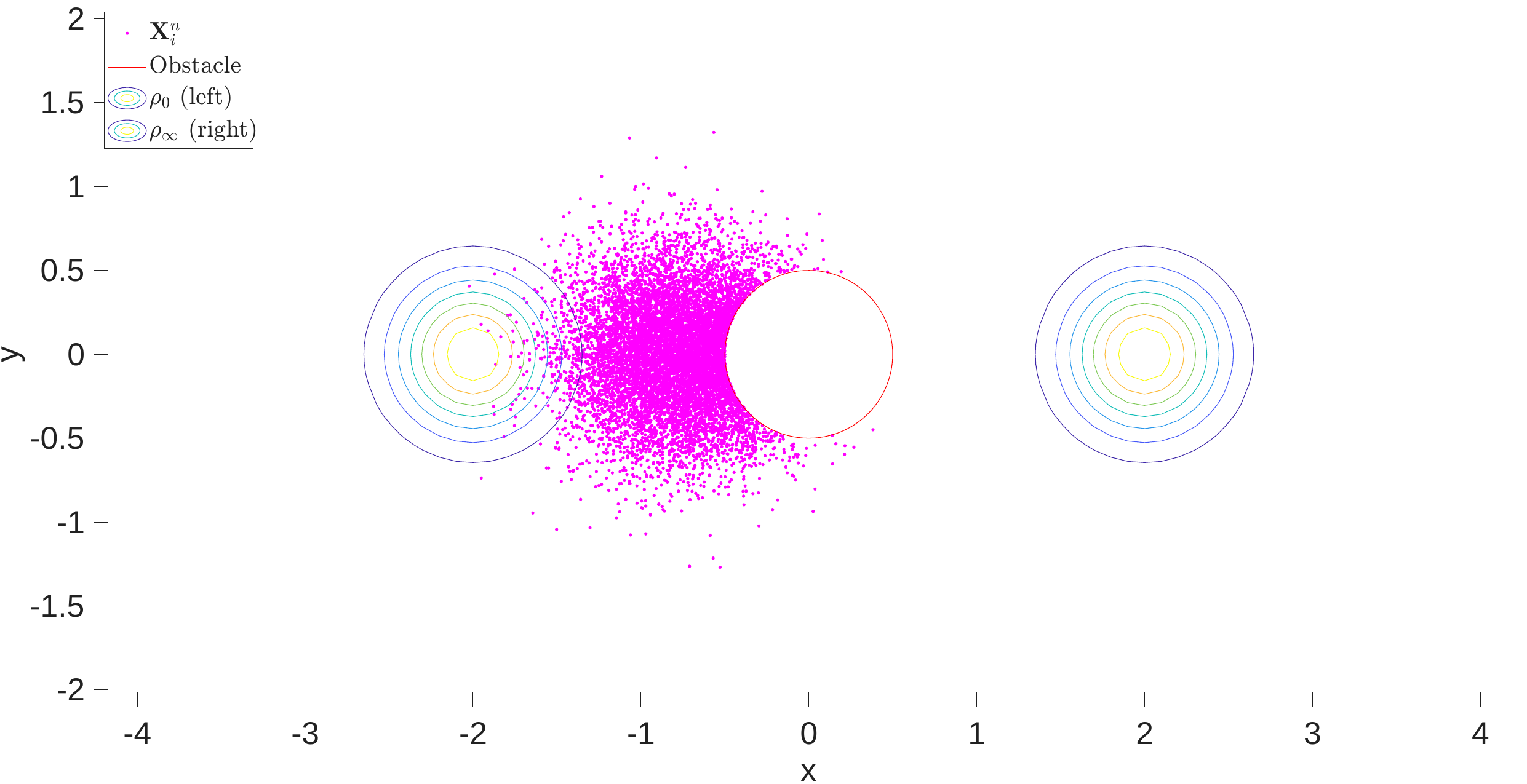}
		\caption{\(t = 3.0 \times 10^{-2}\).}
		\label{fig:example1bMC}
	\end{subfigure}	
	\begin{subfigure}[b]{0.49\textwidth}
		\centering
		\includegraphics[width=\linewidth]{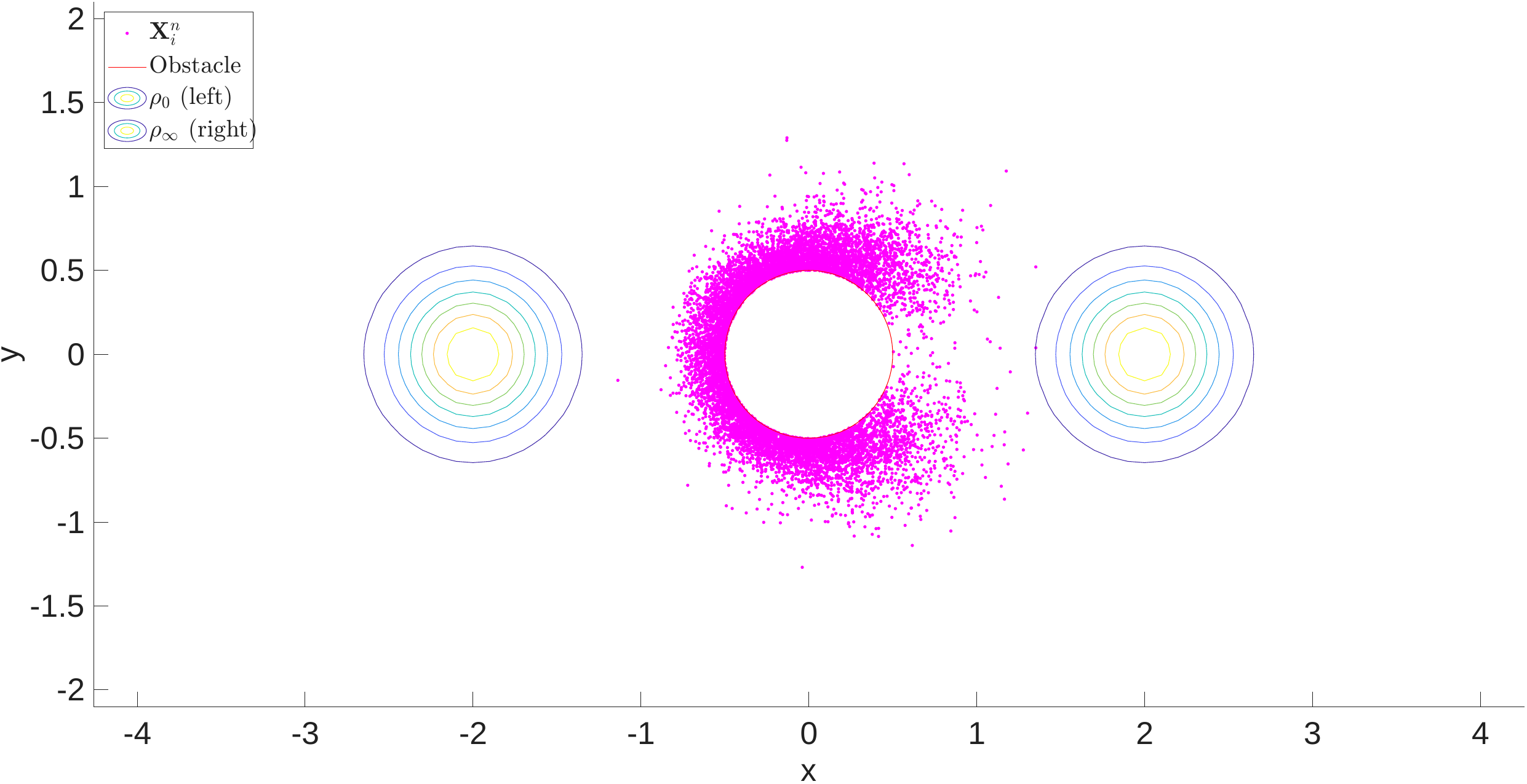}
		\caption{\(t = 7.5 \times 10^{-2}\).}
		\label{fig:example1cMC}
	\end{subfigure}
	\begin{subfigure}[b]{0.49\textwidth}
		\centering
		\includegraphics[width=\linewidth]{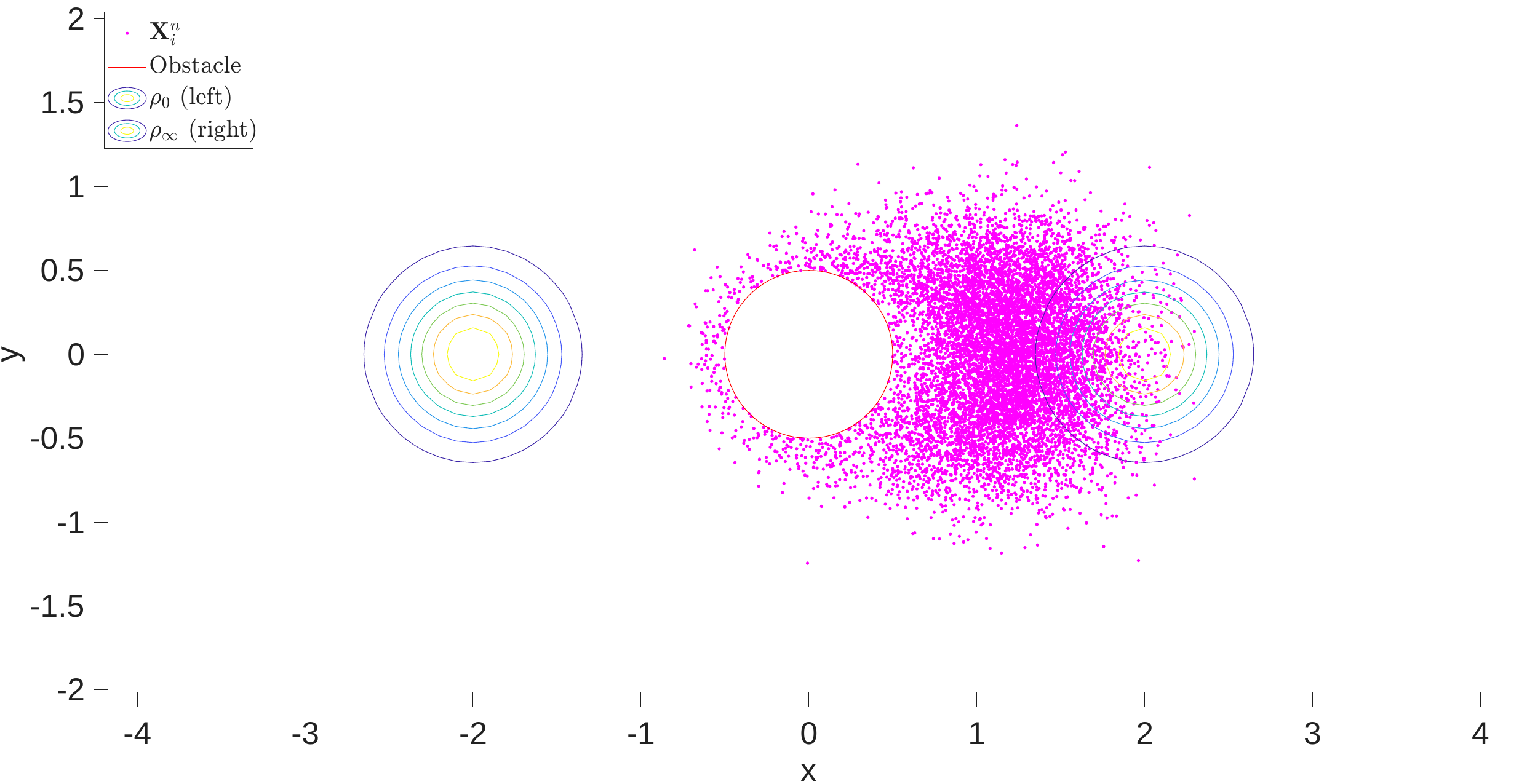}
		\caption{\(t = 1.65 \times 10^{-1}\).}
		\label{fig:example1dMC}
	\end{subfigure}
	\begin{subfigure}[b]{0.49\textwidth}
		\centering
		\includegraphics[width=\linewidth]{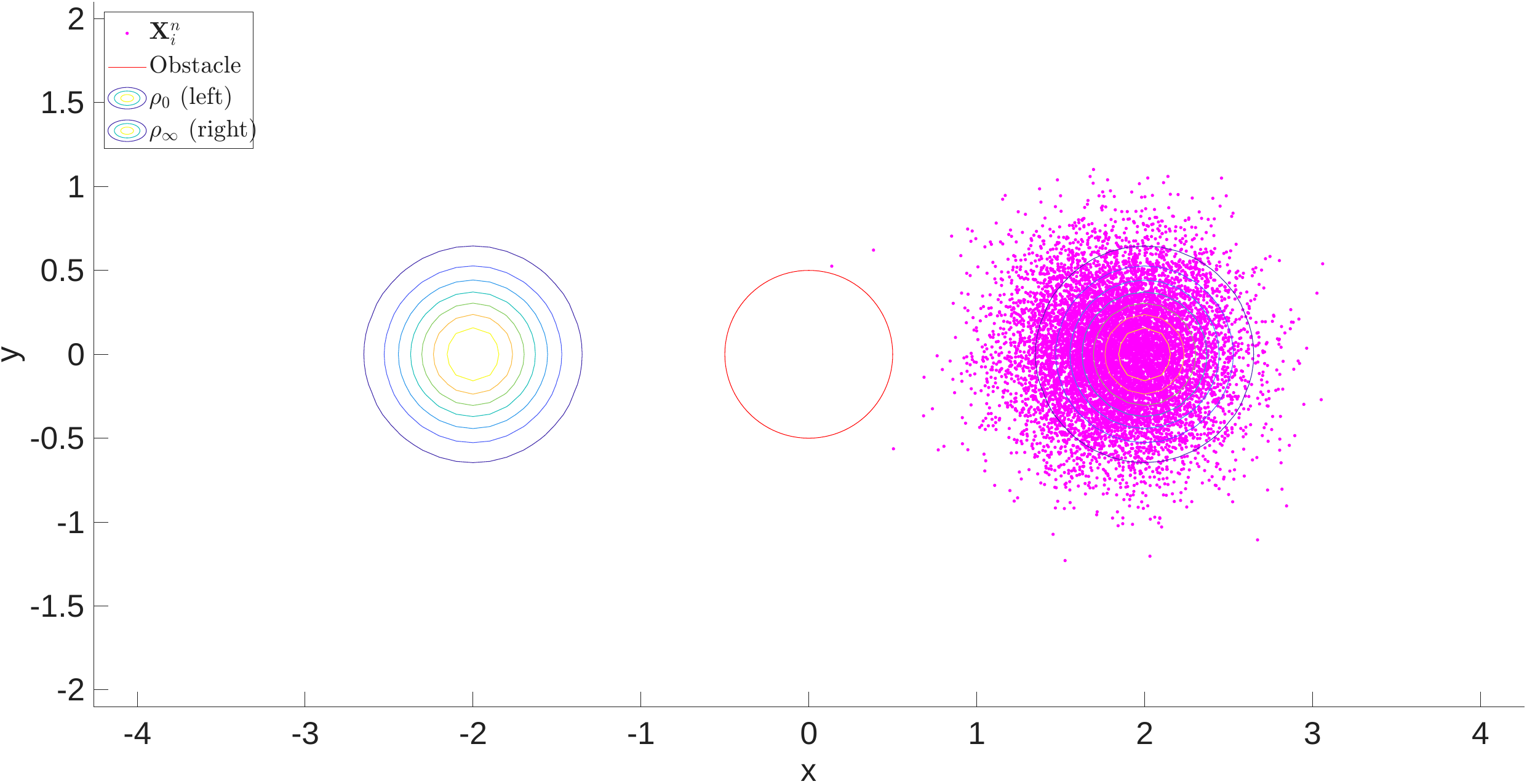}
		\caption{\(t = 3.45 \times 10^{-1}\).}
		\label{fig:example1eMC}
	\end{subfigure}
	\begin{subfigure}[b]{0.49\textwidth}
		\centering
		\includegraphics[width=\linewidth]{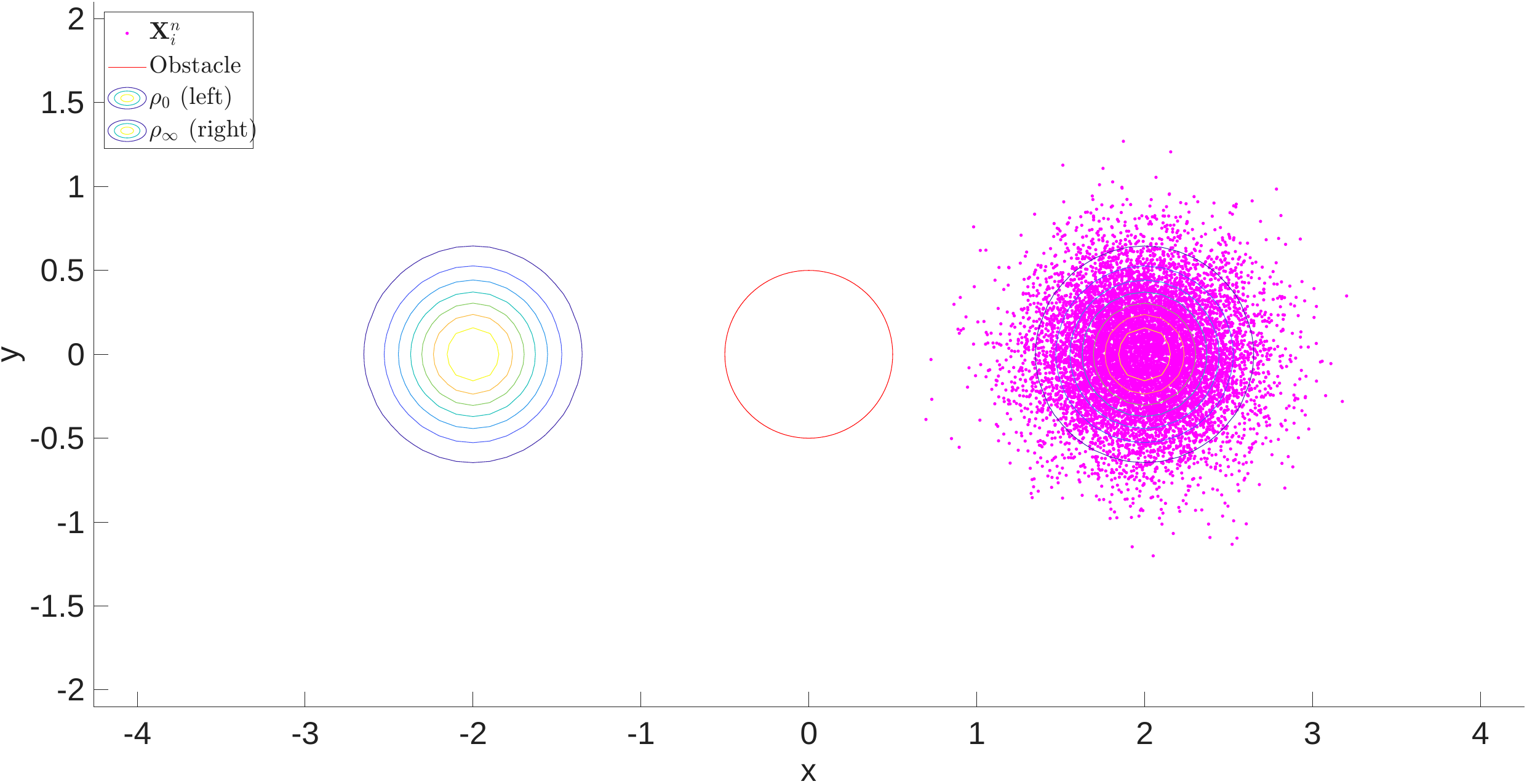}
		\caption{\(t = 7 \times 10^{-1}\).}
		\label{fig:example1fMC}
	\end{subfigure}
	\caption{Mapping of a Gaussian onto another Gaussian across an obstacle using the Monte Carlo method. We plot the initial distribution \(\rho_0\), the target distribution \(\rho_{\infty}\), and the cloud of Monte Carlo particles.}
	\label{fig:example1MLS1MC}
\end{figure}

\begin{figure}
	\centering 
	\begin{subfigure}[b]{0.49\textwidth}
		\centering
		\includegraphics[width=\linewidth]{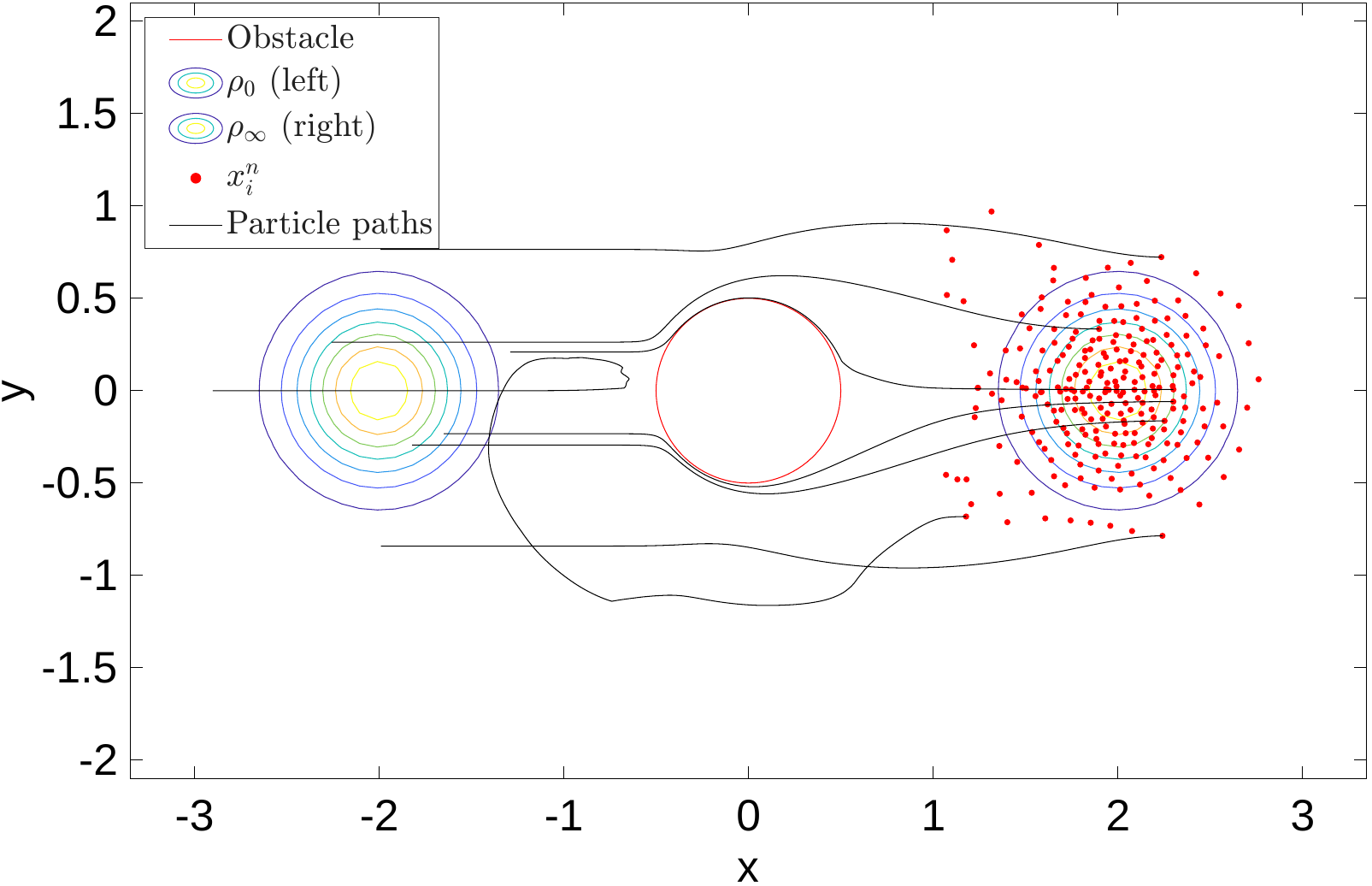}
		\caption{MLS}
		\label{fig:pathsExample1a}
	\end{subfigure}
	\begin{subfigure}[b]{0.49\textwidth}
		\centering
		\includegraphics[width=\linewidth]{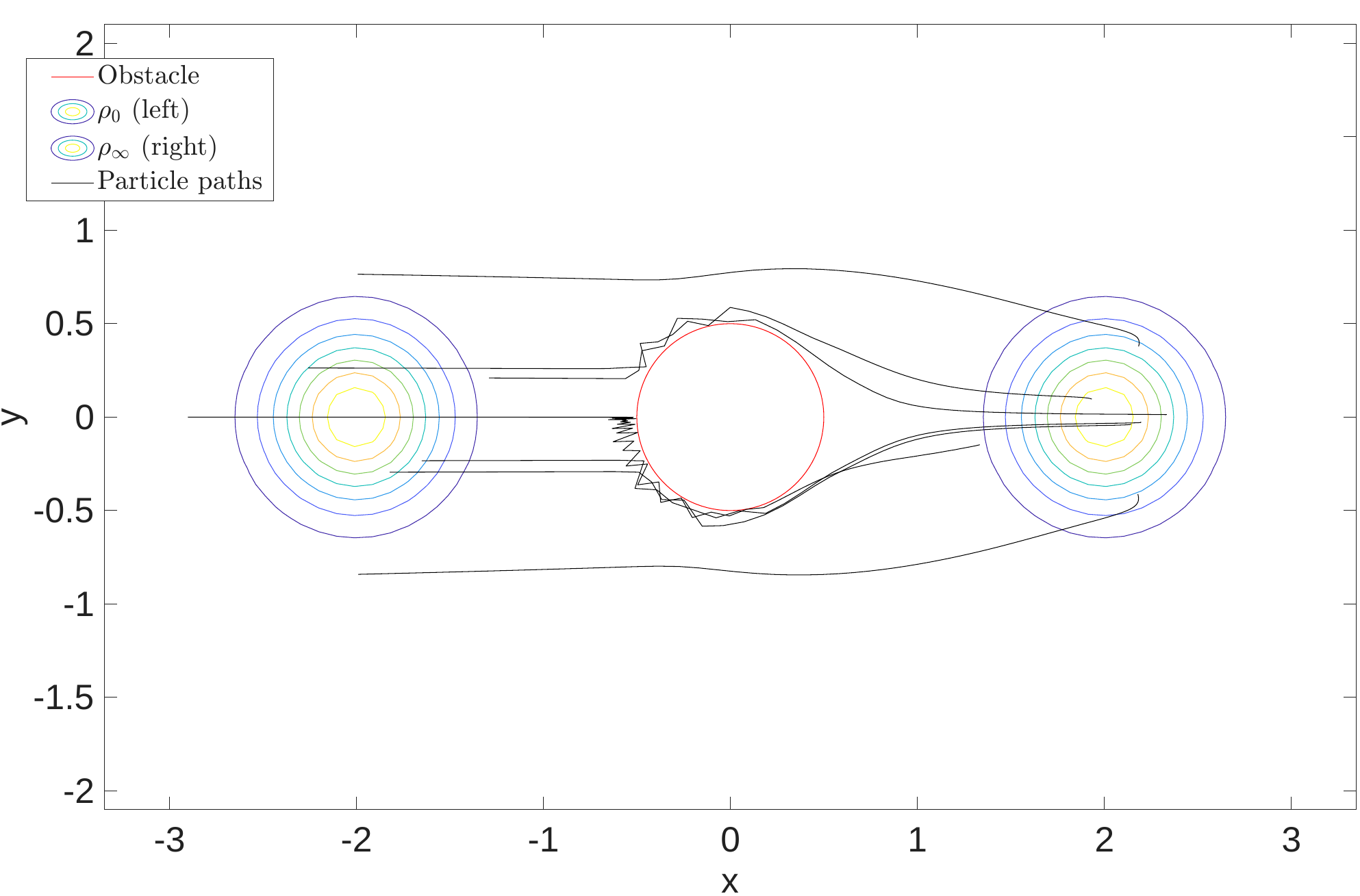}
		\caption{Monte Carlo}
		\label{fig:pathsExample1b}
	\end{subfigure}	
	\caption{Some osmotic paths for the transport of the Gaussian across a cylinder test case generated by the MLS method and the Monte Carlo method.}
	\label{fig:Example1Paths}
\end{figure}

\subsection{Point set registration across a cylinder}
\label{section:PSRAcrossCylinder}
We consider the point set registration problem across a cylinder from \cite{iollo_point-set_2025}. The goal is to map an initial point cloud \(\{X_i^0\}_{i=1}^N\) onto a target point cloud \(\{X_j^\infty\}_{j=1}^M\) and obtain the trajectories of the mapping. We define the point clouds as \begin{align}
	X_i^0 &= \left[\cos \left( \theta_0 + \Delta \theta \frac{i-1}{N-1} \right) - 1, \sin \left( \theta_0 + \Delta \theta \frac{i-1}{N-1} \right)\right]^\top + 0.1 \Tilde{U}_i^0, \quad i = 1 \dots N, \label{eq:PSRInitPos} \\
	X_j^\infty &= \left[\cos \left( \theta_\infty + \Delta \theta \frac{j-1}{M-1} \right) + 1, \sin \left( \theta_\infty + \Delta \theta \frac{j-1}{M-1} \right)\right]^\top + 0.1 \Tilde{U}_j^\infty, \quad j = 1 \dots M,
\end{align} where \(N = M = 400\), \(\theta_0 = \frac{\pi}{2}\), \(\theta_\infty = \frac{3\pi}{2}\), \(\Delta \theta = \pi\) and \(\Tilde{U}_i^0, \Tilde{U}_j^\infty \overset{\mathrm{iid}}{\sim} \text{Uniform}\left((0, 1)^2\right) \). The initial and target point clouds are the left and right halves of a circle separated by a distance of 2. The initial and target point clouds are each approximated using a Gaussian mixture model and kernel density estimation with a Gaussian kernel, see Appendix \ref{appendix:DensityEstimation}. The weights, means and covariance matrices of the GMMs are given in Appendix \ref{appendix:PSRGMM}. The bandwidths of the initial and target distributions obtained using KDE are \(\lambda_{0} = \lambda_{\infty} = 0.3017\). At the origin, we place a cylindrical obstacle with radius \(R = 0.5\). 

The two density-estimation techniques yield distinct approximations of the initial and target densities. We apply the midpoint MLS method to both approximations to assess the influence of the density-estimation technique. For a qualitative comparison between the two particle methods under the same density model, we additionally apply the Monte Carlo method using the GMM approximations. All simulations are performed up to the final time \(t_{\rm f} = 1.2\). 

\subsubsection{MLS}
\label{section:PSRAcrossCylinderMLS}
The initial positions of the MLS particles are set to the point cloud \(\{X_i^0\}_{i=1}^N\) given by Equation \eqref{eq:PSRInitPos}. The MLS particles are thus not `evenly' distributed through the use of the pseudo-random numbers as was the case in the previous simulations. Then, the MLS particles are advanced according to the FPE to the target distribution \(\rho_{\infty}\) that is represented by either a GMM or a KDE. We therefore emphasize that the MLS method does not explicitly map the point cloud \(\{X_i^0\}_{i=1}^N\) point-wise to \(\{X_j^\infty\}_{j=1}^M\). Instead, the MLS method maps the initial point cloud \(\{X_i^0\}_{i=1}^N\) to the target distribution \(\rho_\infty\). It is therefore important for the distribution \(\rho_\infty\) to fit the target point cloud \(\{X_j^\infty\}_{j=1}^M\) in the sense that the target point cloud \(\{X_j^\infty\}_{j=1}^M\) must behave like a sample from \(\rho_\infty\). The simulation is conducted with the MLS midpoint method with \(N_b = 25\) neighbors and a maximum time step of \(\Delta t_{\max} = 5 \times 10^{-4}\). We first perform a simulation using a GMM approximation of the target point cloud. The MLS point distribution as well as the initial and target GMMs are plotted in Figure \ref{fig:example2MLS} at several time steps. In Figure \ref{fig:psrInit}, we find that the initial GMM is visually a good approximation for the initial point cloud. Then, we observe that the initial distribution splits into two clouds that advance around the obstacle. We note that the asymmetry in the top and bottom particle clouds is due to the asymmetry of the initial condition. The two particle clouds merge when reaching the final distribution \(\rho_\infty\). Note that no ghost points are created during the simulation because the cloud does not move close enough to the obstacle. At the final time, we clearly observe the four Gaussians of the final distribution \(\rho_\infty\) in the density plot. Some osmotic paths are plotted in Figure \ref{fig:psrPaths1}. We find that all particles first move outward and then `collapse' onto the final distribution.

The simulation is repeated using KDE approximations of both the initial and target densities. We emphasize that the initial positions of the MLS particles remain as before. The resulting paths are shown in Figure \ref{fig:psrPaths2}. In this case, ghost particles are created and the MLS particles interact with the boundary. We find that the KDE approximation of the initial cloud of points is visibly less accurate than the GMM approximation. In addition, the paths generated from this simulation are qualitatively very different from the GMM case. We conclude that the choice of density estimation procedure has a significant impact on the results if the FPE is used for point set registration. 
\begin{figure}
	\centering 
	\begin{subfigure}[b]{0.49\textwidth}
		\centering
		\includegraphics[width=\linewidth]{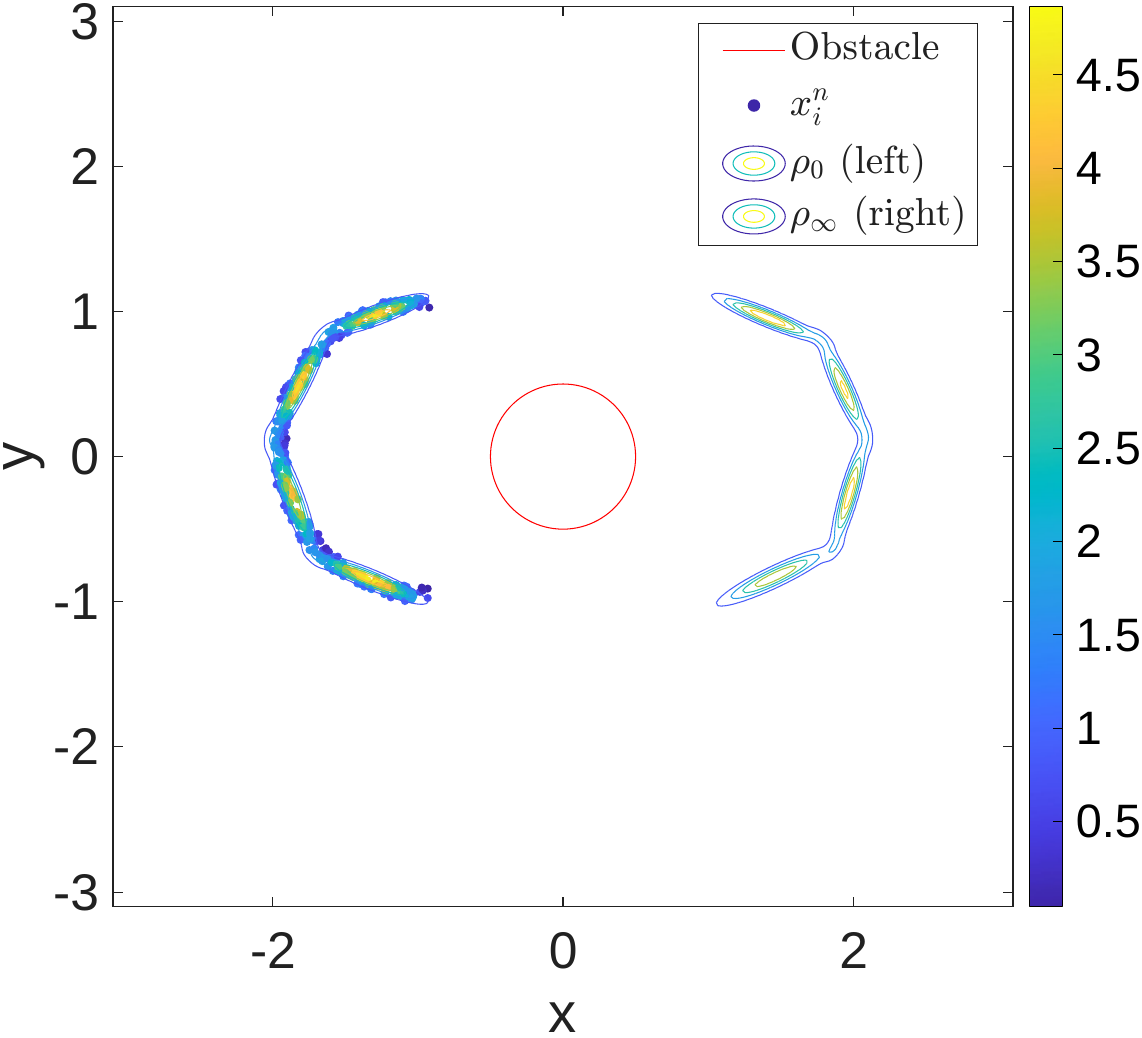}
		\caption{\(t = 0\).}
		\label{fig:psrInit}
	\end{subfigure}
	\begin{subfigure}[b]{0.49\textwidth}
		\centering
		\includegraphics[width=\linewidth]{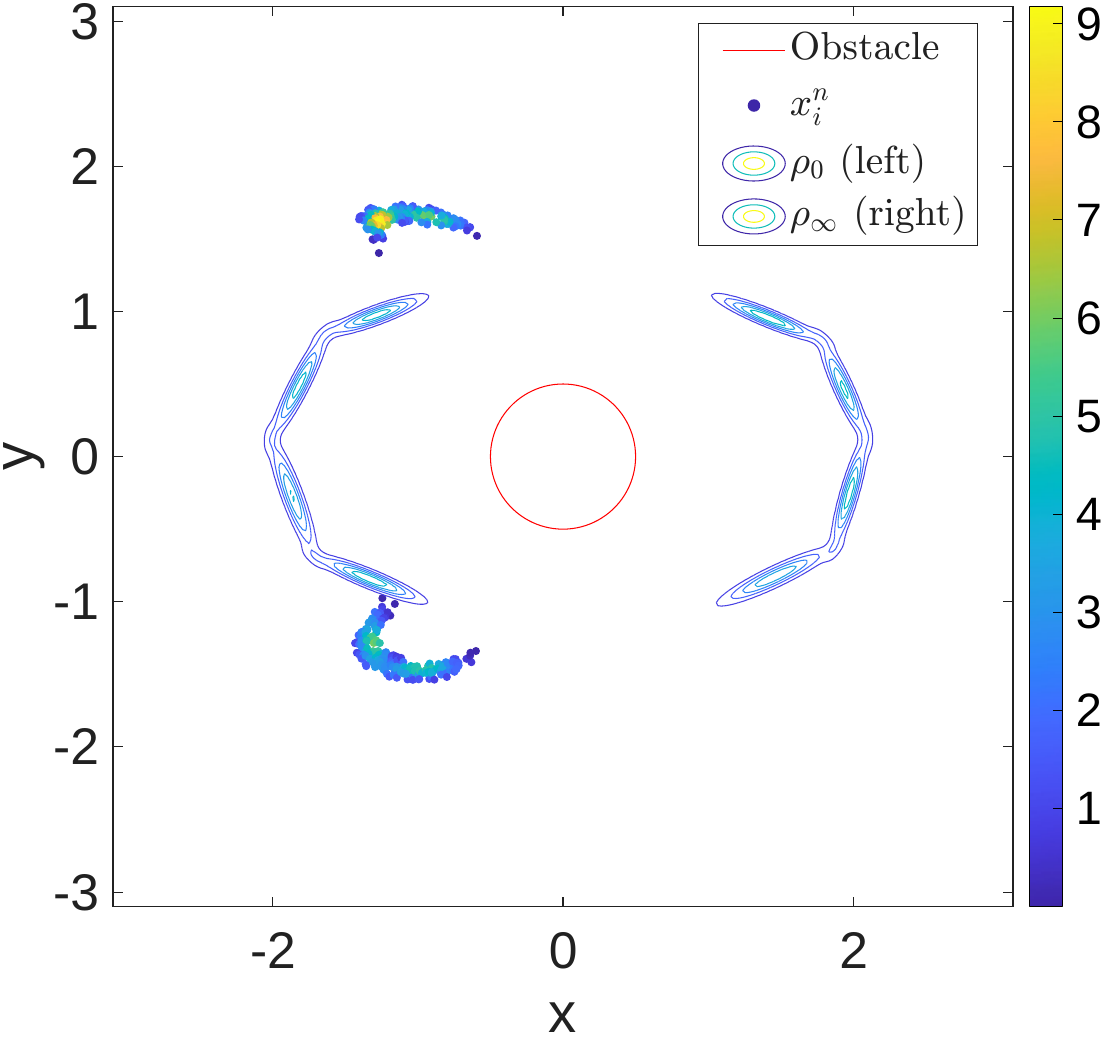}
		\caption{\(t = 1.55 \times 10^{-3}\).}
		\label{fig:psr1}
	\end{subfigure}	
	\begin{subfigure}[b]{0.49\textwidth}
		\centering
		\includegraphics[width=\linewidth]{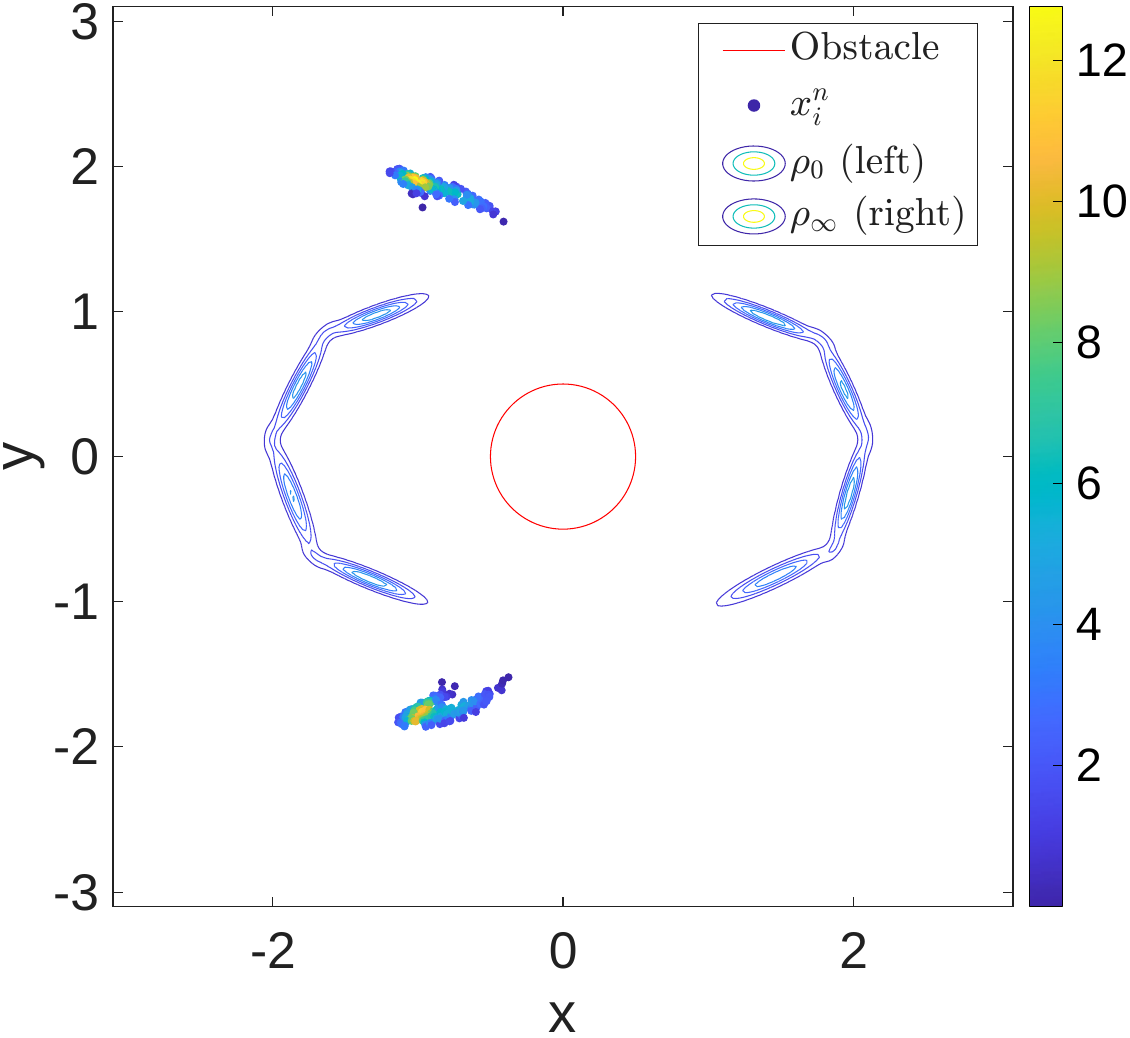}
		\caption{\(t = 3.82 \times 10^{-3}\).}
		\label{fig:psr2}
	\end{subfigure}
	\begin{subfigure}[b]{0.49\textwidth}
		\centering
		\includegraphics[width=\linewidth]{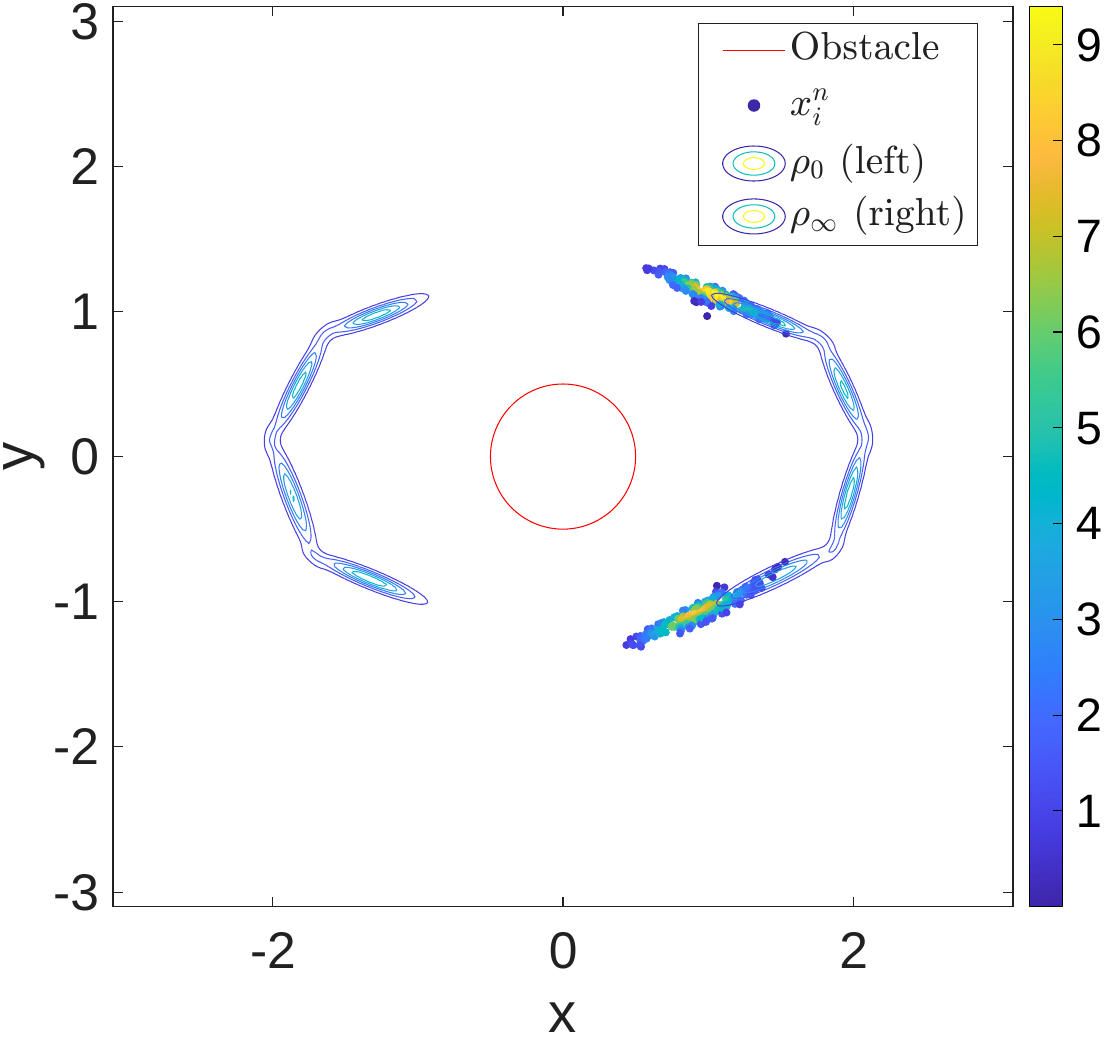}
		\caption{\(t = 9.36 \times 10^{-2}\).}
		\label{fig:psr3}
	\end{subfigure}
	\begin{subfigure}[b]{0.49\textwidth}
		\centering
		\includegraphics[width=\linewidth]{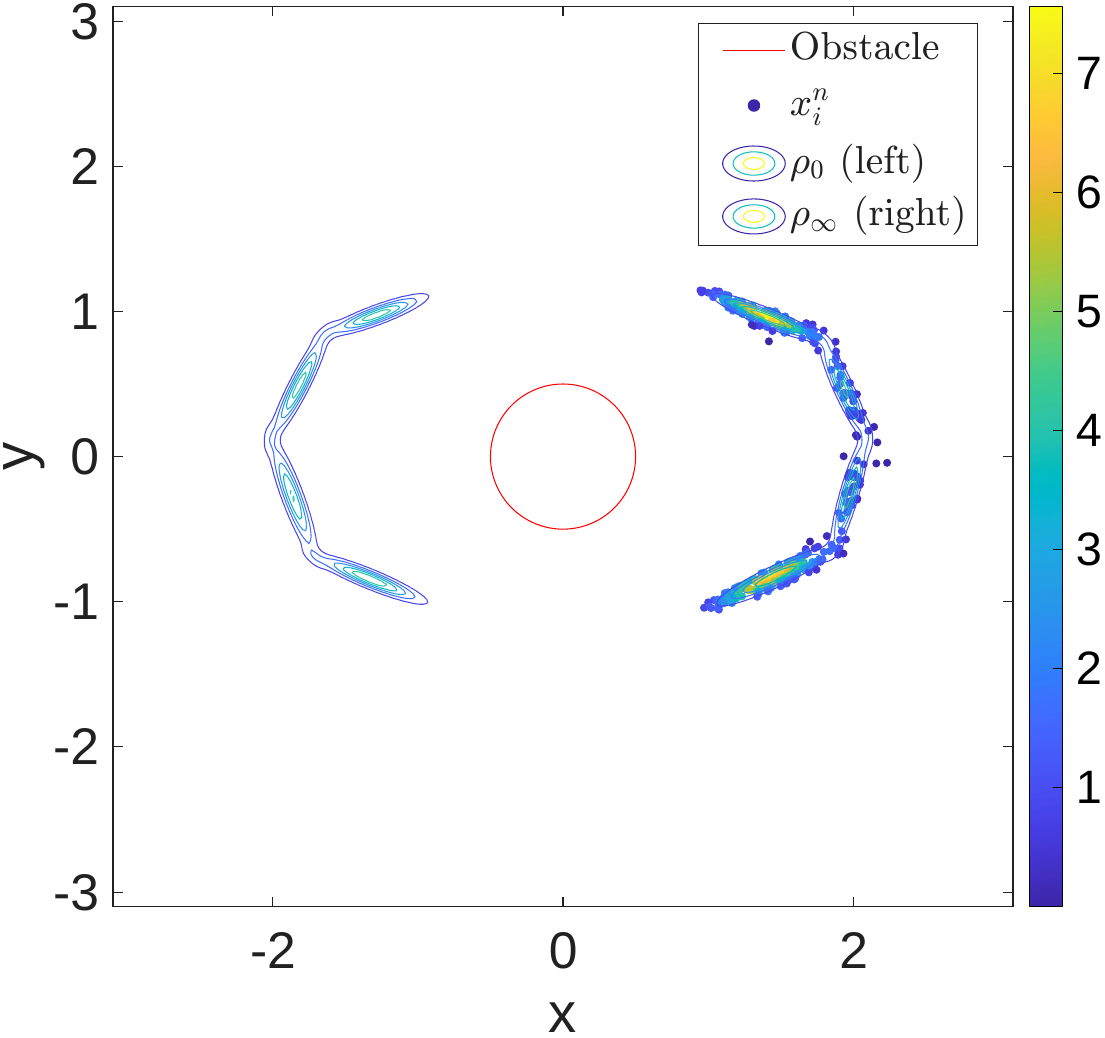}
		\caption{\(t = 3.18 \times 10^{-1}\).}
		\label{fig:psr4}
	\end{subfigure}
	\begin{subfigure}[b]{0.49\textwidth}
		\centering
		\includegraphics[width=\linewidth]{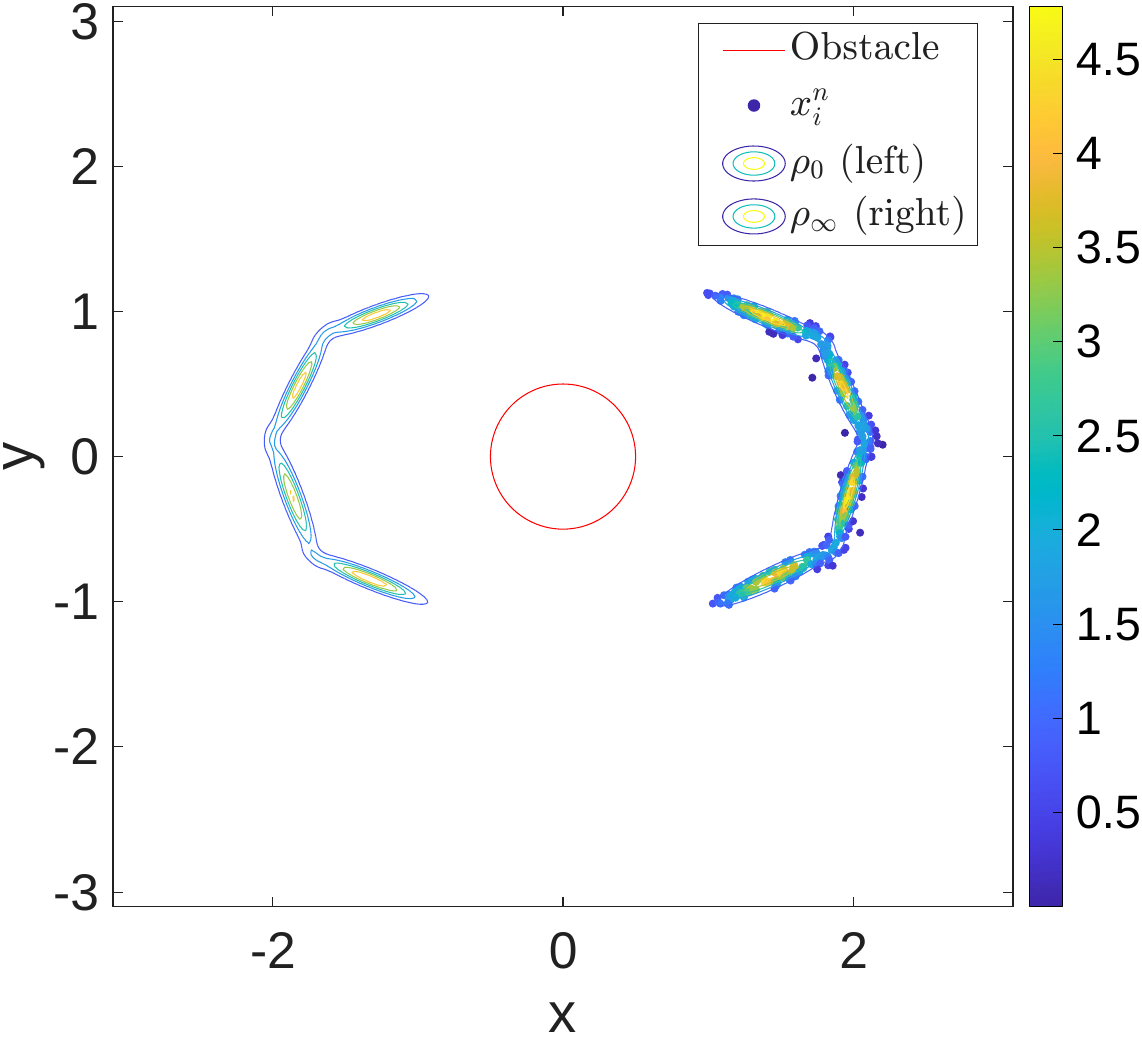}
		\caption{\(t = 1.2\).}
		\label{fig:psr5}
	\end{subfigure}
	\caption{Point-set registration across a cylinder example. We plot the initial distribution \(\rho_0\), the target distribution \(\rho_{\infty}\), and the cloud of particles at several times throughout the simulation.}
	\label{fig:example2MLS}
\end{figure}

\begin{figure}
	\centering 
	\begin{subfigure}[b]{0.49\textwidth}
		\centering
		\includegraphics[width=\linewidth]{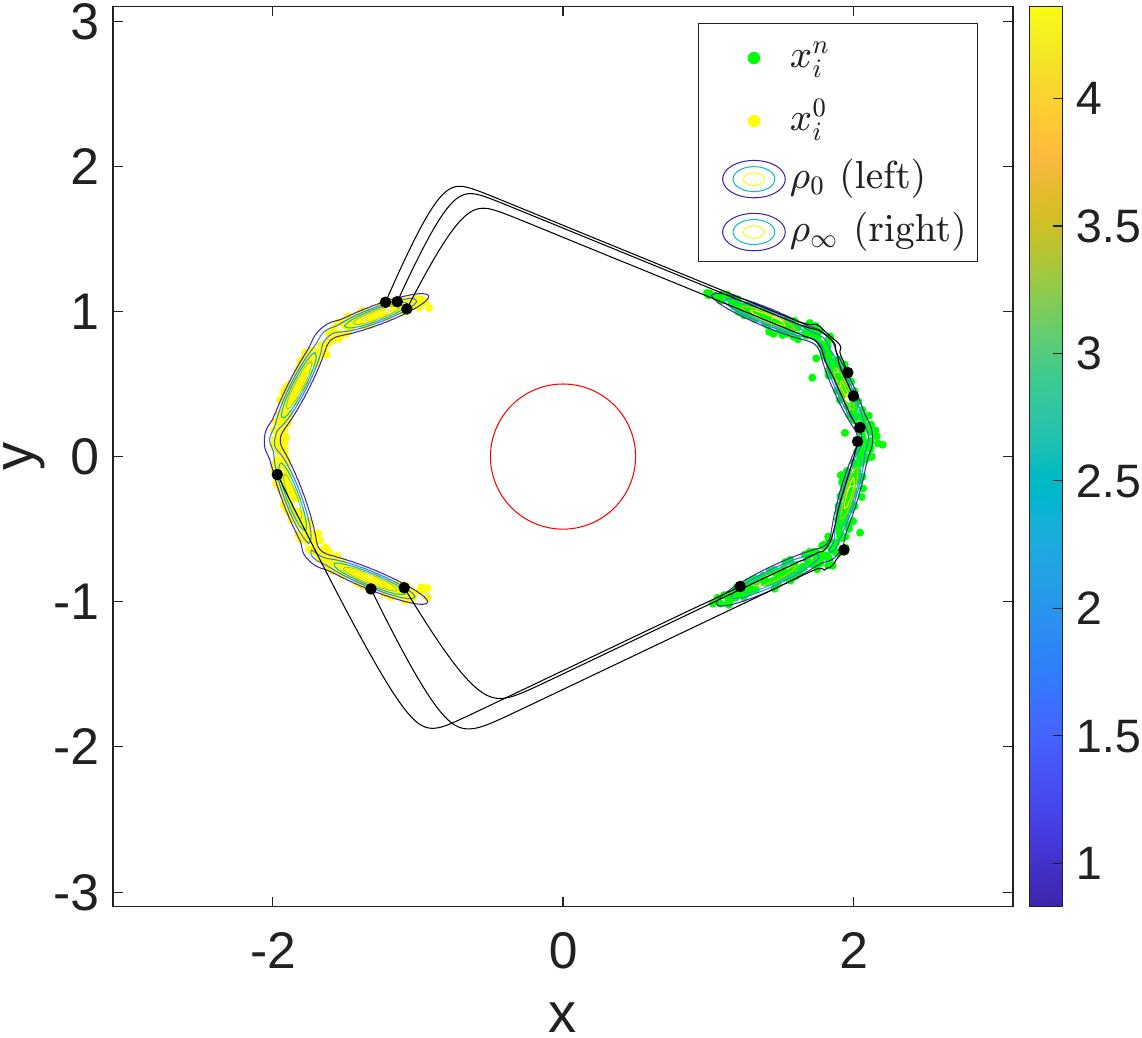}
		\caption{MLS method with \(\rho_{\infty}\) represented by a GMM.}
		\label{fig:psrPaths1}
	\end{subfigure}
	\begin{subfigure}[b]{0.49\textwidth}
		\centering
		\includegraphics[width=\linewidth]{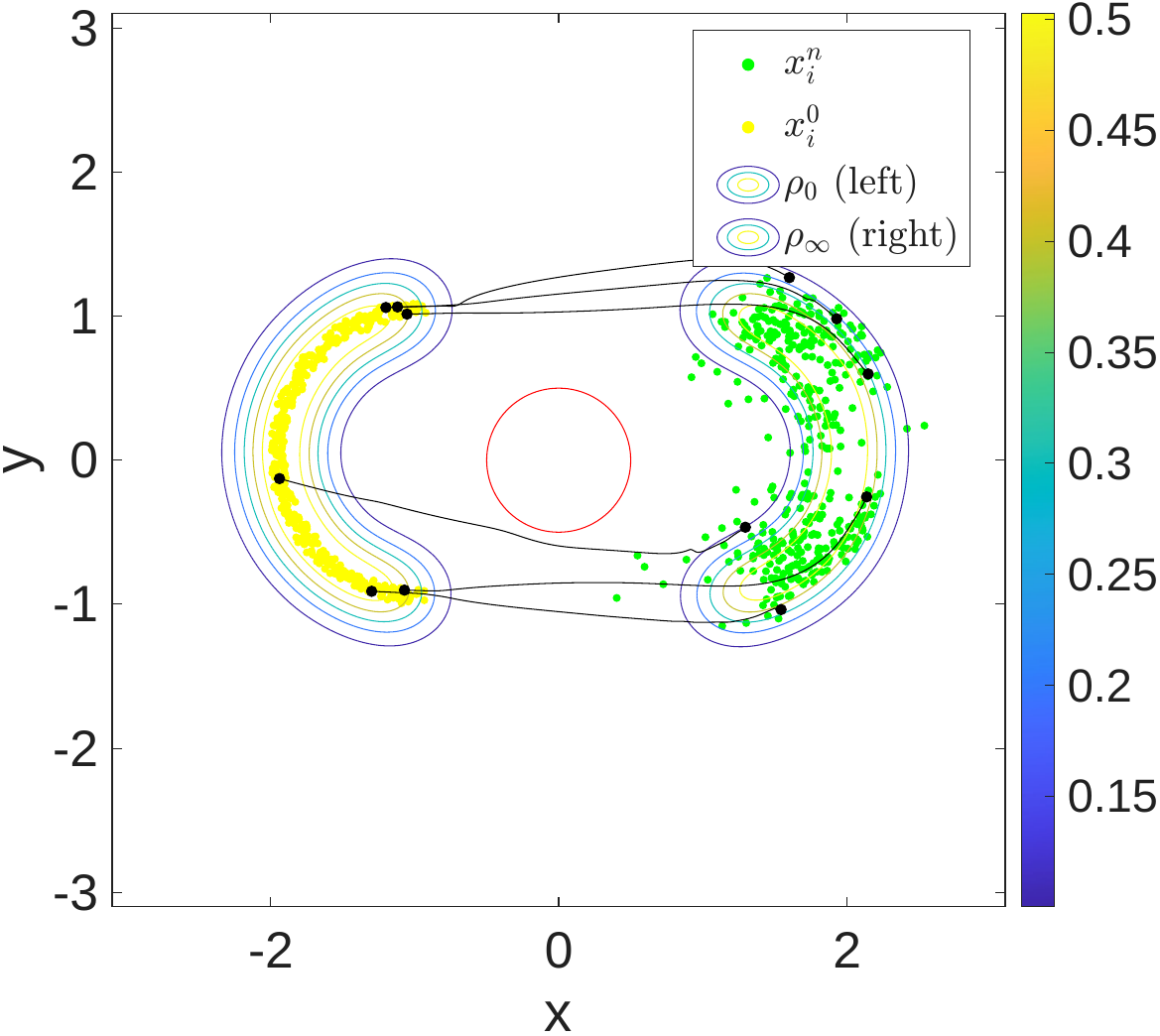}
		\caption{MLS method with \(\rho_{\infty}\) approximated with KDE.}
		\label{fig:psrPaths2}
	\end{subfigure}	
	\begin{subfigure}[b]{0.49\textwidth}
		\centering
		\includegraphics[width=\linewidth]{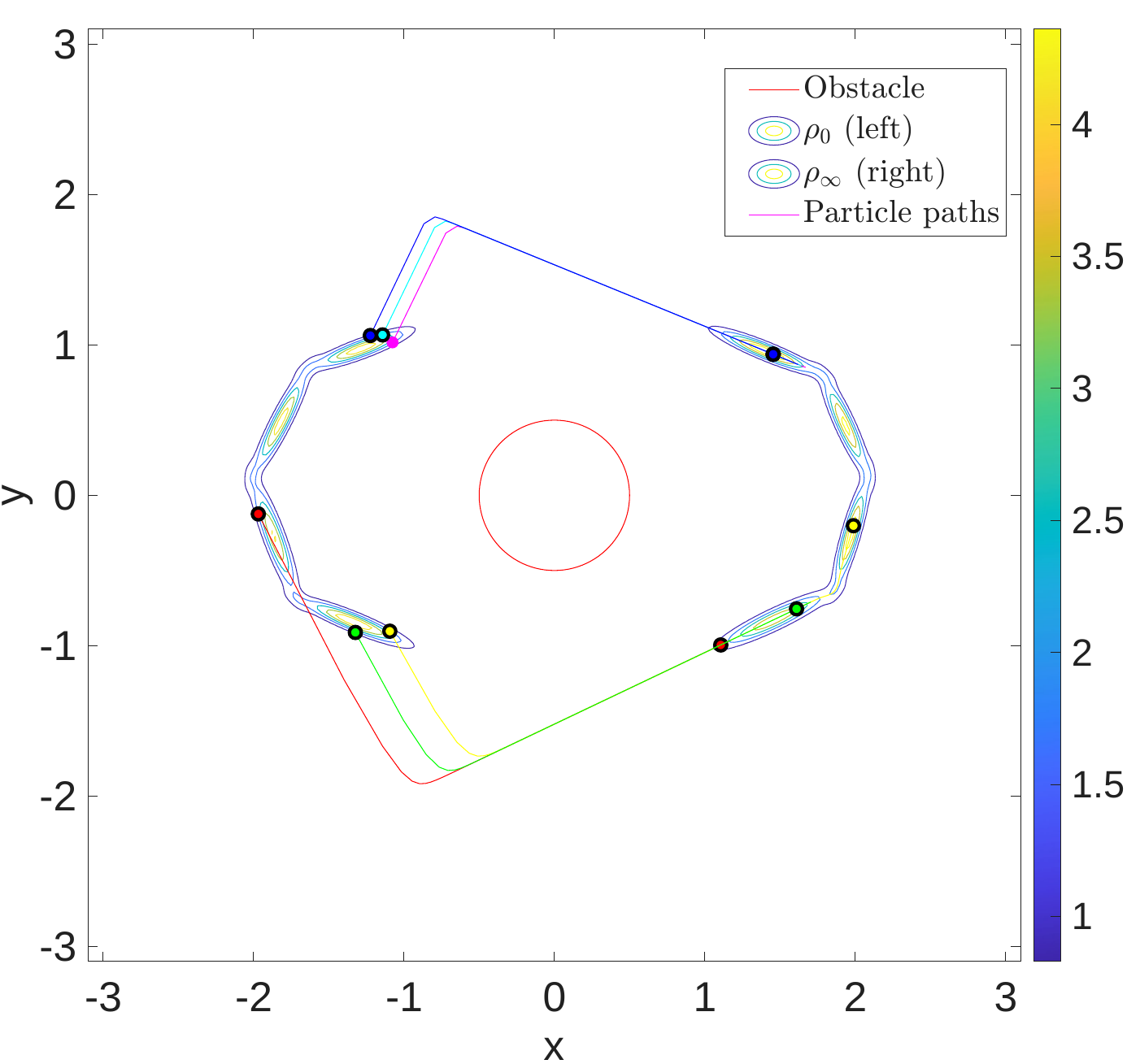}
		\caption{Monte Carlo method with \(\rho_{\infty}\) represented by a GMM.}
		\label{fig:psrPaths3}
	\end{subfigure}
	\caption{Several osmotic paths generated for the MLS method and the Monte Carlo method for different approximations of the initial and target distributions. In the case of the MLS method, the MLS particles are also plotted. In the case of the Monte Carlo method, for visibility, only the paths of the small regions are plotted. }
	\label{fig:psrPaths}
\end{figure}

\subsubsection{Stochastic method}
\label{section:PSRAcrossCylinderMC}
We repeat the simulation in which the point clouds are approximated with a GMM. We first sample \(5 \times 10^5\) particles from the GMM that approximates the initial point cloud. We then place six small regions of radius \(r_{\mathrm{init}} = 1.2 \times 10^{-2}\) at the initial positions of six selected MLS particles from Equation \eqref{eq:PSRInitPos}. The simulation is performed with a time step \(\Delta t = 1 \times 10^{-3}\) and the adaptive-radius method. The osmotic paths are plotted in Figure \ref{fig:psrPaths3}. The paths follow the same qualitative route as those obtained with the MLS method using the GMM, although their final center locations differ appreciably.

\subsection{Discussion}
For the purpose of point set registration, the two particle methods for the FPE presented here are viable options. The MLS method, although the more difficult of the two to implement, yields accurate results. Thanks to the use of an adaptive stencil, the method can handle test cases in which the point cloud undergoes strong expansion (Section \ref{section:ADTestCase}) and deformation (Section \ref{section:gaussianAcrossCylinder}). In addition, the ghost point method correctly enforces reflective boundary conditions, which enables point set registration in bounded domains. We do note that stability of meshfree methods remains an open question. There is no guarantee of the stability of meshfree discretizations for every possible point cloud. Hence, adaptive time stepping criteria such as the one presented in Section \ref{section:FirstOrderScheme} remain necessary. Contrary to meshfree methods, the stability of Monte Carlo methods for SDEs is well-understood. Moreover, the Monte Carlo method is effectively trivial to implement and very easily parallelizable, yielding potentially very efficient algorithms for point set registration. For the purpose of tracing the osmotic paths, the Monte Carlo method inherently introduces noise to the mapping of the point clouds. However, the results in Section \ref{section:DiffusionTestCase} demonstrate that the adaptive radius method substantially reduces this effect. Finally, we note that in regions with large density variations, the current mean-shift approach to estimate the score (Section \ref{section:gaussianAcrossCylinder}) is sensitive to the time step and size of the region. 

\section{Conclusion and Outlook}
\label{section:conclusion}
In this work, we propose two particle-based methods for point set registration in bounded domains, both grounded in the Fokker-Planck formulation of probability transport. By recasting the registration problem as the evolution of a probability density toward a prescribed target distribution, we obtain a continuous flow of mappings that naturally respects mass conservation and boundary constraints while avoiding explicit point-to-point matching.

The first method is a deterministic particle approach based on a moving least squares discretization of the S-formulation of the Fokker-Planck equation. By evolving particles according to the drift induced by the target distribution and resolving diffusion on a dynamically deforming particle cloud, this method directly produces coherent transport paths suitable for registration. A key feature of this approach is the use of a ghost point technique to impose reflective boundary conditions, which enables a consistent and robust treatment of walls and obstacles without the need for mesh generation or domain truncation. The numerical experiments demonstrate that the method can accurately capture strong compression, expansion, and splitting of particle clouds induced by both the drift and the domain geometry.

The second method adopts a stochastic viewpoint and relies on Monte Carlo simulation of the reflected Langevin dynamics associated with the Fokker-Planck equation. Rather than solving the PDE explicitly, this approach evolves large ensembles of particles and reconstructs macroscopic osmotic paths from the stochastic trajectories using a local approximation of the score of the evolving particle distribution. While inherently noisy, this method is straightforward to implement, embarrassingly parallel, and highly scalable, making it particularly attractive for high-dimensional settings where deterministic discretizations become prohibitively expensive.
The numerical results highlight that the two approaches are complementary. The MLS-based method offers higher accuracy and smoother trajectories at a relatively modest particle count, at the cost of additional algorithmic complexity and stability considerations. In contrast, the Monte Carlo method trades deterministic accuracy for simplicity and scalability and provides a viable alternative when large numbers of particles and parallel computing resources are available.

Several directions for future work naturally arise. On the deterministic side, adaptive strategies for stencil selection, time stepping, and particle refinement could further improve robustness in regions of strong deformation or near boundaries. On the stochastic side, an adaptive time-stepping procedure could enhance the quality of the recovered osmotic paths. In view of developing a practically viable point set registration method, a systematic investigation of different density estimation techniques and their influence on transport paths would also be of significant interest. Finally, an interesting direction for future research is the study of the Schrödinger bridge problem in bounded domains \cite{caluya_reflected_2021}.

\textbf{Acknowledgments: } This work was supported by the European Union’s Framework Program for Research and Innovation Horizon Europe under the Marie~Skłodowska-Curie Doctoral Networks action (HORIZON-MSCA-2021-DN-01), Grant Agreement No.~101072546 (DATAHYKING).

\clearpage
\appendix

\section{Density estimation}
\label{appendix:DensityEstimation}
The first step of our method consists of estimating the densities $\rho_0$ and $\rho_\infty$ that enter \eqref{eq:FPequation}.
The broad literature on density estimation ranges from classical statistical methods \cite{hastie2009elements} to recent methods closely linked to generative artificial
intelligence \cite{lai2025principles}. The choice of method critically depends on the number of data points, the dimension $d$, and the task of interest.
In this work, we target two- and three-dimensional domains and expect datasets of moderate size, $N,M \lesssim 10^3$. Therefore, in the numerical experiments, we consider
two methods:
kernel density estimation (KDE) and Gaussian mixture models (GMMs).

\subsection{Kernel density estimation}
We introduce the \emph{kernel} $\phi:\mathbb{R}_+ \to \mathbb{R}_+$ that satisfies
\begin{subequations}
	\label{eq:KDE}
	\begin{equation}
		\label{eq:KDEa}
		\phi(0) = 1, \quad
		c_{\phi,d} = |\mathbb{S}^{d-1}| \int_0^\infty r^{d-1} \phi(r) \, dr <\infty,
		\quad
		\phi'(r)<0 \quad
		\forall \, r>0.
	\end{equation}
		Next, given the $N$ samples $\mathcal{X} = \{x_i\}_{i=1}^N$,
		$x_i\overset{\rm iid}{\sim} \rho^\star$,
		we define the KDE:
	\begin{equation}
		\label{eq:KDEb}
		\widehat{\rho}_{N,\lambda}(x) = \dfrac{C_{N,\lambda}}{N \lambda^d} \sum_{i=1}^N \phi \left(
		\dfrac{\| x - x_i  \|_2}{\lambda}
		\right),
	\end{equation}
	Here, $|\mathbb{S}^{d-1}|$ denotes the surface measure of the unit sphere in $\mathbb{R}^d$, $\lambda>0$ is the kernel bandwidth, and $C_{N,\lambda}$ is chosen so that
	$\int_\Omega \widehat{\rho}_{N,\lambda}(x)\,dx=1$.
	If $\Omega=\mathbb{R}^d$, then $C_{N,\lambda}=1/c_{\phi,d}$. In a bounded domain, the normalization constant generally also depends on the sample locations because
	the kernels are truncated by the boundary.
	
\end{subequations}

For sufficiently smooth densities, the estimator converges in probability to the true density $\rho^\star$ under the standard conditions $\lambda\to0$ and
$N\lambda^d\to\infty$ as $N\to\infty$. The asymptotically optimal bandwidth scales as $\lambda\propto N^{-1/(d+4)}$ \cite{wasserman2006all}.
In practice, we use cross-validation to determine $\lambda$. We consider leave-one-out cross-validation (LOOCV) \cite{stone1974cross} and least-squares cross-validation
(LSCV) \cite[Chapter~6.1]{wasserman2006all}, which lead to similar results in the experiments of Section \ref{section:numerics}.

The choice of the kernel function $\phi$ is important for both approximation and stability. The kernel should be sufficiently regular to ensure that its gradient is Lipschitz continuous and should be strictly positive within its support. Its tail behavior also plays a major role in the numerical approximation of the FPE. For rapidly decaying kernels, if a unique closest sample dominates the sum in the far field, then \eqref{eq:KDEb} can be approximated by
$$
\widehat{\rho}_{N,\lambda}(x) \approx \dfrac{C_{N,\lambda}}{N \lambda^d} \phi \left(
\dfrac{\| x - x_{i^\star}  \|_2}{\lambda}
\right),
\quad
i^\star = {\rm arg}\min_{i=1,\ldots,N} \| x - x_i \|_2
$$
which implies
$$
\nabla \log \widehat{\rho}_{N,\lambda}(x)
\approx
S_\phi \left(
\dfrac{\| x - x_{i^\star}  \|_2}{\lambda}
\right)
\dfrac{  x - x_{i^\star}   }{\lambda \| x - x_{i^\star}  \|_2},
\quad
{\rm where} \;\;
S_\phi(r) = \dfrac{\phi'(r)}{\phi(r)},
$$
and thus
\begin{equation}
	\label{eq:estimate_gradient_farfield}
		\|  \nabla \log \widehat{\rho}_{N,\lambda}(x) \|_2
		\approx
		\dfrac{1}{\lambda}
		\left|S_\phi  \left(
		\dfrac{ {\rm dist} (x, \mathcal{X})  }{\lambda}
		\right)\right|.
\end{equation}
Table \ref{tab:kernels_nightmare} reports the expression of $S_\phi$
for widely used kernels. We observe markedly different behaviors of $S_\phi$ in the far field.

\begin{table}[t]
	\centering
	\begin{tabular}{|c|c|c|}
		\hline
		Kernel & $\phi(r)$ & $S_\phi(r)$  \\[2mm]
		\hline
		Gaussian & $e^{-r^2}$ &   $-2 r$  \\[2mm]
		\hline
		Matérn ($5/2$) & $ (1 + r  + \frac{1}{3} r^2) e^{-r} $ &
		$
		- \frac{r(r + 1)}{r^2 + 3r + 3}$  
		\\[2mm]
		\hline
		inverse multiquadric & $(1+ r^2)^{-\beta}$  &
		$-\frac{2\beta r}{r^2 + 1}$
		\\[2mm]
		\hline
	\end{tabular}
		\caption{Behavior of $S_\phi$ for several popular kernels.}
	\label{tab:kernels_nightmare}
\end{table}

For Gaussian kernels ($\phi(r) = e^{-r^2}$), we find $S_\phi(r) = - 2 r$: this implies that the advection term of the FPE satisfies
\begin{equation}
	\label{eq:tail_estimate_gaussian}
	\|  \nabla \log \widehat{\rho}_{N,\lambda}(x) \|_2
	\approx
		\dfrac{2\,{\rm dist} (x, \mathcal{X})}{\lambda^2}
\end{equation}
in the far field; on the other hand, for inverse multiquadric kernels
($\phi(r) = (1 + r^2)^{-\beta} $ with $\beta>0$), we find $S_\phi(r) = - \dfrac{2\beta r}{1+ r^2}$.
Although the nearest-sample approximation need not hold for this heavy-tailed kernel, a direct expansion of the full KDE shows that its score vanishes in the far field.
We leave a detailed comparison of kernel functions and their impact on point set registration to future work.

\subsection{Gaussian mixture models}

GMMs are widely used for density estimation in low dimensions. They provide a compact representation of a density in terms of $K$ multivariate Gaussian components,
\begin{equation}
	\label{eq:gmm}
	\widehat{\rho}_{\rm gmm}(x)
	=
	\sum_{k=1}^K \;
	\pi_k \mathcal{N}(x; \mu_k, \Sigma_k),
\end{equation}
where the weights $0\leq \pi_1,\ldots,\pi_K \leq 1$ satisfy $\sum_k \pi_k = 1$, the means satisfy
$\{ \mu_k \}_k \subset \mathbb{R}^d$, and the covariance matrices satisfy $\{ \Sigma_k \}_k \subset \mathbb{S}_+^d$, where $\mathbb{S}_+^d$ denotes the set of symmetric positive definite matrices.
As in \cite{iollo2025point}, we use the MATLAB function \texttt{fitgmdist} to fit the GMMs and the Akaike information criterion (AIC)
\cite[Chapter~7]{hastie2009elements} to select $K$. To improve numerical robustness, we set the \texttt{RegularizationValue} option of \texttt{fitgmdist} to \(10^{-4}\).

\begin{remark}
	As in the KDE case, Gaussian mixture variants with flexible far-field tail behavior can be considered; see, for example, \cite{wraith2015location}.
	Here, our primary focus is to simplify the approximation of solutions to \eqref{eq:FPequation}; in this respect, our approach shares the objective of
	\cite{de2025improving}.
\end{remark}

\section{GMM parameters}
\label{appendix:PSRGMM}
In Section \ref{section:PSRAcrossCylinder}, the initial and target distributions are each approximated by a GMM.
The weights, means, and covariance matrices of each component are listed below.
All reported values are rounded to four decimal places; the simulations use the full-precision parameters provided in the accompanying code~\cite{Willems_Moving_Least_Squares}.
Consequently, the displayed mixture weights may not sum exactly to one and should be renormalized if used directly.

\subsection{Initial distribution}

Weights: $\mathbf{w}_0 = [0.2507, 0.2559, 0.2700, 0.2235]^\top$.

Means:
\[
\boldsymbol{\mu}_{0,1} = \begin{pmatrix} -1.8187 \\ 0.4900 \end{pmatrix}, \quad
\boldsymbol{\mu}_{0,2} = \begin{pmatrix} -1.3318 \\ -0.8428 \end{pmatrix}, \quad
\boldsymbol{\mu}_{0,3} = \begin{pmatrix} -1.8621  \\ -0.2770 \end{pmatrix}, \quad
\boldsymbol{\mu}_{0,4} = \begin{pmatrix} -1.2859 \\ 0.9745 \end{pmatrix}.
\]

Covariances:
\[
\boldsymbol{\Sigma}_{0,1} = \begin{pmatrix} 0.0119 & 0.0205 \\ 0.0205 & 0.0415 \end{pmatrix}, \quad
\boldsymbol{\Sigma}_{0,2} = \begin{pmatrix} 0.0450 & -0.0179 \\ -0.0179 & 0.0086 \end{pmatrix},
\]
\[
\boldsymbol{\Sigma}_{0,3} = \begin{pmatrix} 0.0086 & -0.0187 \\ -0.0187 & 0.0523 \end{pmatrix}, \quad
\boldsymbol{\Sigma}_{0,4} = \begin{pmatrix} 0.0376 & 0.0134 \\ 0.0134 & 0.0063 \end{pmatrix}.
\]

\subsection{Target distribution}

Weights: $\mathbf{w}_\infty = [0.2424, 0.2439, 0.2347, 0.2789]^\top$.

Means:
\[
\boldsymbol{\mu}_{\infty,1} = \begin{pmatrix} 1.9739 \\ -0.2523 \end{pmatrix}, \quad
\boldsymbol{\mu}_{\infty,2} = \begin{pmatrix} 1.4127 \\ 0.9556 \end{pmatrix}, \quad
\boldsymbol{\mu}_{\infty,3} = \begin{pmatrix} 1.9349 \\ 0.4663 \end{pmatrix}, \quad
\boldsymbol{\mu}_{\infty,4} = \begin{pmatrix} 1.4609 \\ -0.8263 \end{pmatrix}.
\]

Covariances:
\[
\boldsymbol{\Sigma}_{\infty,1} = \begin{pmatrix} 0.0049 & 0.0130 \\ 0.0130 & 0.0462 \end{pmatrix}, \quad
\boldsymbol{\Sigma}_{\infty,2} = \begin{pmatrix} 0.0425 & -0.0169 \\ -0.0169 & 0.0080 \end{pmatrix},
\]
\[
\boldsymbol{\Sigma}_{\infty,3} = \begin{pmatrix} 0.0090 & -0.0168 \\ -0.0168 & 0.0383 \end{pmatrix}, \quad
\boldsymbol{\Sigma}_{\infty,4} = \begin{pmatrix} 0.0516 & 0.0238 \\ 0.0238 & 0.0130 \end{pmatrix}.
\]

\newpage
\printbibliography

@techreport{kuhnert_grid_2001,
	title = {Grid free method for solving the poisson equation},
	url = {https://publica.fraunhofer.de/handle/publica/291170},
	number = {25},
	urldate = {2024-11-15},
	institution = {Fraunhofer-Gesellschaft},
	author = {Kuhnert, J. and Tiwari, S.},
	year = {2001},
}

@article{ascher_implicit-explicit_1995,
	title = {Implicit-{Explicit} {Methods} for {Time}-{Dependent} {Partial} {Differential} {Equations}},
	volume = {32},
	issn = {0036-1429, 1095-7170},
	url = {http://epubs.siam.org/doi/10.1137/0732037},
	doi = {10.1137/0732037},
	language = {en},
	number = {3},
	urldate = {2023-10-17},
	journal = {SIAM Journal on Numerical Analysis},
	author = {Ascher, Uri M. and Ruuth, Steven J. and Wetton, Brian T. R.},
	month = jun,
	year = {1995},
	pages = {797--823},
}

@book{suchde_conservation_2018,
	address = {Stuttgart},
	title = {Conservation and accuracy in meshfree generalized finite difference methods},
	isbn = {978-3-8396-1325-2},
	language = {eng},
	publisher = {Fraunhofer Verlag},
	author = {Suchde, Pratik},
	collaborator = {{Fraunhofer-Institut für Techno- und Wirtschaftsmathematik}},
	year = {2018},
}

@book{seibold_m-matrices_2006,
	address = {Aachen},
	series = {Berichte aus der {Mathematik}},
	title = {M-matrices in meshless finite difference methods},
	isbn = {978-3-8322-5643-2},
	language = {eng},
	publisher = {Shaker},
	author = {Seibold, Benjamin},
	year = {2006},
}

@article{levin_approximation_1998,
	title = {The approximation power of moving least-squares},
	volume = {67},
	issn = {0025-5718, 1088-6842},
	url = {https://www.ams.org/mcom/1998-67-224/S0025-5718-98-00974-0/},
	doi = {10.1090/S0025-5718-98-00974-0},
	language = {en},
	number = {224},
	urldate = {2024-10-23},
	journal = {Mathematics of Computation},
	author = {Levin, David},
	year = {1998},
	pages = {1517--1531},
}

@article{clain_stencil_2024,
	title = {Stencil and kernel optimisation for mesh-free very high-order generalised finite difference method},
	volume = {218},
	issn = {03784754},
	url = {https://linkinghub.elsevier.com/retrieve/pii/S0378475423004706},
	doi = {10.1016/j.matcom.2023.11.009},
	language = {en},
	urldate = {2024-11-15},
	journal = {Mathematics and Computers in Simulation},
	author = {Clain, S. and Figueiredo, J.},
	month = apr,
	year = {2024},
	pages = {49--78},
}

@article{lancaster_surfaces_1981,
	title = {Surfaces generated by moving least squares methods},
	volume = {37},
	issn = {0025-5718, 1088-6842},
	url = {https://www.ams.org/mcom/1981-37-155/S0025-5718-1981-0616367-1/},
	doi = {10.1090/S0025-5718-1981-0616367-1},
	language = {en},
	number = {155},
	urldate = {2024-11-14},
	journal = {Mathematics of Computation},
	author = {Lancaster, P. and Salkauskas, K.},
	year = {1981},
	pages = {141--158},
}

@article{king_high-order_2022,
	title = {High-order simulations of isothermal flows using the local anisotropic basis function method ({LABFM})},
	volume = {449},
	issn = {00219991},
	url = {https://linkinghub.elsevier.com/retrieve/pii/S0021999121006550},
	doi = {10.1016/j.jcp.2021.110760},
	language = {en},
	urldate = {2025-10-24},
	journal = {Journal of Computational Physics},
	author = {King, J.R.C. and Lind, S.J.},
	month = jan,
	year = {2022},
	pages = {110760},
}

@article{king_high_2020,
	title = {High order difference schemes using the local anisotropic basis function method},
	volume = {415},
	issn = {00219991},
	url = {https://linkinghub.elsevier.com/retrieve/pii/S0021999120303235},
	doi = {10.1016/j.jcp.2020.109549},
	language = {en},
	urldate = {2025-12-03},
	journal = {Journal of Computational Physics},
	author = {King, J.R.C. and Lind, S.J. and Nasar, A.M.A.},
	month = aug,
	year = {2020},
	pages = {109549},
}

@article{wang_compact_2023,
	title = {Compact moving particle semi-implicit method for incompressible free-surface flow},
	volume = {414},
	issn = {00457825},
	url = {https://linkinghub.elsevier.com/retrieve/pii/S004578252300292X},
	doi = {10.1016/j.cma.2023.116168},
	language = {en},
	urldate = {2026-03-17},
	journal = {Computer Methods in Applied Mechanics and Engineering},
	author = {Wang, Zidi and Matsumoto, Toshinori and Duan, Guangtao and Matsunaga, Takuya},
	month = sep,
	year = {2023},
	pages = {116168},
}

@article{tamai_least_2014,
	title = {Least squares moving particle semi-implicit method},
	volume = {1},
	issn = {21964386},
	url = {https://linkinghub.elsevier.com/retrieve/pii/S2196438625008629},
	doi = {10.1007/s40571-014-0027-2},
	language = {en},
	number = {3},
	urldate = {2026-03-17},
	journal = {Computational Particle Mechanics},
	author = {Tamai, Tasuku and Koshizuka, Seiichi},
	month = sep,
	year = {2014},
	pages = {277--305},
}

@article{koshizuka_numerical_1998,
	title = {Numerical analysis of breaking waves using the moving particle semi-implicit method},
	volume = {26},
	copyright = {http://doi.wiley.com/10.1002/tdm\_license\_1.1},
	issn = {0271-2091, 1097-0363},
	url = {https://onlinelibrary.wiley.com/doi/10.1002/(SICI)1097-0363(19980415)26:7<751::AID-FLD671>3.0.CO;2-C},
	doi = {10.1002/(SICI)1097-0363(19980415)26:7<751::AID-FLD671>3.0.CO;2-C},
	language = {en},
	number = {7},
	urldate = {2026-03-17},
	journal = {International Journal for Numerical Methods in Fluids},
	author = {Koshizuka, Seiichi and Nobe, Atsushi and Oka, Yoshiaki},
	month = apr,
	year = {1998},
	pages = {751--769},
}

@article{jandaghian_enhanced_2020,
	title = {An enhanced weakly-compressible {MPS} method for free-surface flows},
	volume = {360},
	issn = {00457825},
	url = {https://linkinghub.elsevier.com/retrieve/pii/S0045782519306632},
	doi = {10.1016/j.cma.2019.112771},
	language = {en},
	urldate = {2026-03-17},
	journal = {Computer Methods in Applied Mechanics and Engineering},
	author = {Jandaghian, M. and Shakibaeinia, A.},
	month = mar,
	year = {2020},
	pages = {112771},
}

@article{BidautVeron2015HJ,
  author  = {Bidaut-V{\'e}ron, Marie-Fran{\c{c}}oise},
  title   = {Local and global estimates of solutions of Hamilton--Jacobi parabolic equation with absorption},
  journal = {Advances in Differential Equations},
  volume  = {20},
  number  = {11--12},
  pages   = {1077--1114},
  year    = {2015}
}

@article{myronenko2010point,
  title={Point set registration: Coherent point drift},
  author={Myronenko, Andriy and Song, Xubo},
  journal={IEEE transactions on pattern analysis and machine intelligence},
  volume={32},
  number={12},
  pages={2262--2275},
  year={2010},
  publisher={IEEE}
}

@article{yizong_cheng_mean_1995,
	title = {Mean shift, mode seeking, and clustering},
	volume = {17},
	copyright = {https://ieeexplore.ieee.org/Xplorehelp/downloads/license-information/IEEE.html},
	issn = {01628828},
	url = {http://ieeexplore.ieee.org/document/400568/},
	doi = {10.1109/34.400568},
	language = {en},
	number = {8},
	urldate = {2026-04-23},
	journal = {IEEE Transactions on Pattern Analysis and Machine Intelligence},
	author = {{Yizong Cheng}},
	month = aug,
	year = {1995},
	pages = {790--799},
}

@article{fukunaga_estimation_1975,
	title = {The estimation of the gradient of a density function, with applications in pattern recognition},
	volume = {21},
	copyright = {https://ieeexplore.ieee.org/Xplorehelp/downloads/license-information/IEEE.html},
	issn = {0018-9448, 1557-9654},
	url = {https://ieeexplore.ieee.org/document/1055330/},
	doi = {10.1109/TIT.1975.1055330},
	number = {1},
	urldate = {2026-04-23},
	journal = {IEEE Transactions on Information Theory},
	author = {Fukunaga, K. and Hostetler, L.},
	month = jan,
	year = {1975},
	pages = {32--40},
}

@article{razavi2025registration,
  title={Registration-based nonlinear model reduction of parametrized aerodynamics problems with applications to transonic {E}uler and {RANS} flows},
  author={Razavi, Alireza H and Yano, Masayuki},
  journal={Journal of Computational Physics},
  volume={521},
  pages={113576},
  year={2025},
  publisher={Elsevier}
}

@article{blanchard_accurately_2021,
	title = {Accurately computing the log-sum-exp and softmax functions},
	volume = {41},
	copyright = {http://creativecommons.org/licenses/by/4.0/},
	issn = {0272-4979, 1464-3642},
	url = {https://academic.oup.com/imajna/article/41/4/2311/5893596},
	doi = {10.1093/imanum/draa038},
	language = {en},
	number = {4},
	urldate = {2026-04-21},
	journal = {IMA Journal of Numerical Analysis},
	author = {Blanchard, Pierre and Higham, Desmond J and Higham, Nicholas J},
	month = oct,
	year = {2021},
	pages = {2311--2330},
}

@book{kloeden_numerical_1992,
	address = {Berlin, Heidelberg},
	title = {Numerical {Solution} of {Stochastic} {Differential} {Equations}},
	copyright = {http://www.springer.com/tdm},
	isbn = {978-3-642-08107-1 978-3-662-12616-5},
	url = {http://link.springer.com/10.1007/978-3-662-12616-5},
	doi = {10.1007/978-3-662-12616-5},
	urldate = {2026-04-23},
	publisher = {Springer Berlin Heidelberg},
	author = {Kloeden, Peter E. and Platen, Eckhard},
	year = {1992},
}

@inproceedings{ambrosio2017lecture,
  title={Lecture notes on the {DiP}erna--{L}ions theory in abstract measure spaces},
  author={Ambrosio, Luigi and Trevisan, Dario},
  booktitle={Annales de la Facult{\'e} des sciences de Toulouse: Math{\'e}matiques},
  volume={26},
  number={4},
  pages={729--766},
  year={2017}
}

@book{oksendal_sde,
  author = {Øksendal, Bernt},
  title = {Stochastic Differential Equations},
  publisher = {Springer},
  edition = {6th},
  year = {2003}
}

@book{pavliotis_stochastic,
  author = {Pavliotis, Grigorios A.},
  title = {Stochastic Processes and Applications},
  publisher = {Springer},
  year = {2014}
}

@article{lions_sznitman_reflected,
  author = {Lions, P.-L. and Sznitman, A.-S.},
  title = {Stochastic differential equations with reflecting boundary conditions},
  journal = {Communications on Pure and Applied Mathematics},
  volume = {37},
  number = {4},
  pages = {511--537},
  year = {1984}
}

@article{stone1974cross,
  title={Cross-validatory choice and assessment of statistical predictions},
  author={Stone, Mervyn},
  journal={Journal of the royal statistical society: Series B (Methodological)},
  volume={36},
  number={2},
  pages={111--133},
  year={1974},
  publisher={Wiley Online Library}
}

@book{wasserman2006all,
  title={All of nonparametric statistics},
  author={Wasserman, Larry},
  year={2006},
  publisher={Springer}
}

@article{balakrishnan2019voxelmorph,
  title={Voxelmorph: a learning framework for deformable medical image registration},
  author={Balakrishnan, Guha and Zhao, Amy and Sabuncu, Mert R and Guttag, John and Dalca, Adrian V},
  journal={IEEE transactions on medical imaging},
  volume={38},
  number={8},
  pages={1788--1800},
  year={2019},
  publisher={IEEE}
}

@article{zhu2019review,
  title={A review of point set registration: From pairwise registration to groupwise registration},
  author={Zhu, Hao and Guo, Bin and Zou, Ke and Li, Yongfu and Yuen, Ka-Veng and Mihaylova, Lyudmila and Leung, Henry},
  journal={Sensors},
  volume={19},
  number={5},
  pages={1191},
  year={2019},
  publisher={MDPI}
}

@incollection{beier2023feature,
  title={Feature-based image metamorphosis},
  author={Beier, Thaddeus and Neely, Shawn},
  booktitle={Seminal Graphics Papers: Pushing the Boundaries, Volume 2},
  pages={529--536},
  year={2023}
}

@article{beg2005computing,
  title={Computing large deformation metric mappings via geodesic flows of diffeomorphisms},
  author={Beg, M Faisal and Miller, Michael I and Trouv{\'e}, Alain and Younes, Laurent},
  journal={International journal of computer vision},
  volume={61},
  pages={139--157},
  year={2005},
  publisher={Springer}
}

@article{iollo2014advection,
  title={Advection modes by optimal mass transfer},
  author={Iollo, Angelo and Lombardi, Damiano},
  journal={Physical Review E},
  volume={89},
  number={2},
  pages={022923},
  year={2014},
  publisher={APS}
}

@article{russo1990particle,
  title={A particle method for collisional kinetic equations. {I. B}asic theory and one-dimensional results},
  author={Russo, G},
  journal={Journal of Computational Physics},
  volume={87},
  number={2},
  pages={270--300},
  year={1990},
  publisher={Elsevier}
}

@article{wraith2015location,
  title={Location and scale mixtures of {G}aussians with flexible tail behaviour: Properties, inference and application to multivariate clustering},
  author={Wraith, Darren and Forbes, Florence},
  journal={Computational Statistics \& Data Analysis},
  volume={90},
  pages={61--73},
  year={2015},
  publisher={Elsevier}
}

@article{de2025improving,
  title={Improving the robustness of neural {ODE}s with minimal weight perturbation},
  author={De Marinis, A and Guglielmi, N and Sicilia, S and Tudisco, F},
  journal={arXiv preprint arXiv:2501.10740},
  year={2025}
}

@article{labatut2026parametric,
  title={Parametric vector flows for registration fields in bounded domains with applications to nonlinear interpolation of shock-dominated flows},
  author={Labatut, Jon and Chapelier, Jean-Baptiste and Iollo, Angelo and Taddei, Tommaso},
  journal={arXiv preprint arXiv:2601.22712},
  year={2026}
}

@article{iollo2025point,
  title={Point-set registration in bounded domains via the {F}okker--{P}lanck equation},
  author={Iollo, Angelo and Taddei, Tommaso},
  journal={Comptes Rendus. Math{\'e}matique},
  volume={363},
  number={G8},
  pages={809--824},
  year={2025}
}

@article{iollo2022mapping,
  title={Mapping of coherent structures in parameterized flows by learning optimal transportation with {G}aussian models},
  author={Iollo, Angelo and Taddei, Tommaso},
  journal={Journal of Computational Physics},
  volume={471},
  pages={111671},
  year={2022},
  publisher={Elsevier}
}

@article{taddei2020registration,
  title={A registration method for model order reduction: data compression and geometry reduction},
  author={Taddei, Tommaso},
  journal={SIAM Journal on Scientific Computing},
  volume={42},
  number={2},
  pages={A997--A1027},
  year={2020},
  publisher={SIAM}
}

@inproceedings{caluya_reflected_2021,
  author={Caluya, Kenneth F. and Halder, Abhishek},
  title={Reflected {S}chr{\"o}dinger Bridge: Density Control with Path Constraints},
  booktitle={2021 American Control Conference (ACC)},
  pages={1137--1142},
  year={2021},
  publisher={IEEE},
  doi={10.23919/ACC50511.2021.9482813}
}

@book{hastie2009elements,
  address = {New York, NY, USA},
  author = {Hastie, Trevor and Tibshirani, Robert and Friedman, Jerome},
  publisher = {Springer New York Inc.},
  series = {Springer Series in Statistics},
  title = {The Elements of Statistical Learning},
  year = 2009
}

@article{russo1990deterministic,
  title={Deterministic diffusion of particles},
  author={Russo, Giovanni},
  journal={Communications on Pure and Applied Mathematics},
  volume={43},
  number={6},
  pages={697--733},
  year={1990},
  publisher={Wiley Online Library}
}

@article{iollo_point-set_2025,
	title = {Point-set registration in bounded domains via the {Fokker}–{Planck} equation},
	volume = {363},
	copyright = {https://creativecommons.org/licenses/by/4.0/},
	issn = {1631-073X, 1778-3569},
	url = {https://comptes-rendus.academie-sciences.fr/mathematique/articles/10.5802/crmath.753/},
	doi = {10.5802/crmath.753},
	language = {en},
	number = {G8},
	urldate = {2025-11-10},
	journal = {Comptes Rendus. Mathématique},
	author = {Iollo, Angelo and Taddei, Tommaso},
	month = jul,
	year = {2025},
	pages = {809--824},
}

@article{lai2025principles,
  title={The principles of diffusion models},
  author={Lai, Chieh-Hsin and Song, Yang and Kim, Dongjun and Mitsufuji, Yuki and Ermon, Stefano},
  journal={arXiv preprint arXiv:2510.21890},
  year={2025}
}

@software{Willems_Moving_Least_Squares,
author = {Willems, Klaas and Taddei, Tommaso  and Iollo, Angelo and Russo, Giovanni },
license = {MIT},
title = {{Moving Least Squares Method for Point Set Registration}}
}

@book{karatzas_brownian_1998,
	address = {New York, NY},
	series = {Graduate {Texts} in {Mathematics}},
	title = {Brownian {Motion} and {Stochastic} {Calculus}},
	volume = {113},
	copyright = {http://www.springer.com/tdm},
	isbn = {978-0-387-97655-6 978-1-4612-0949-2},
	url = {http://link.springer.com/10.1007/978-1-4612-0949-2},
	doi = {10.1007/978-1-4612-0949-2},
	language = {en},
	urldate = {2026-04-15},
	publisher = {Springer New York},
	author = {Karatzas, Ioannis and Shreve, Steven E.},
	year = {1998},
}

@article{caflisch_monte_1998,
	title = {Monte {Carlo} and quasi-{Monte} {Carlo} methods},
	volume = {7},
	copyright = {https://www.cambridge.org/core/terms},
	issn = {0962-4929, 1474-0508},
	url = {https://www.cambridge.org/core/product/identifier/S0962492900002804/type/journal_article},
	doi = {10.1017/S0962492900002804},
	language = {en},
	urldate = {2026-02-10},
	journal = {Acta Numerica},
	author = {Caflisch, Russel E.},
	month = jan,
	year = {1998},
	pages = {1--49},
}

@book{niederreiter_random_1992,
	address = {Philadelphia, Pa},
	series = {Regional conference series in applied mathematics},
	title = {Random number generation and quasi- {Monte} {Carlo} methods},
	isbn = {978-0-89871-295-7},
	language = {eng},
	number = {63},
	publisher = {Society for Industrial and Applied Mathematics},
	author = {Niederreiter, Harald},
	year = {1992},
}

@book{kampen_stochastic_2007,
	address = {Amsterdam, The Netherlands Oxford, UK},
	edition = {3rd ed},
	series = {North-{Holland} personal library},
	title = {Stochastic processes in physics and chemistry},
	isbn = {978-0-444-52965-7},
	language = {eng},
	publisher = {Elsevier},
	author = {Kampen, N. G. van},
	year = {2007},
}

@article{cucchiara_model_2024,
	title = {Model order reduction by convex displacement interpolation},
	volume = {514},
	issn = {00219991},
	url = {https://linkinghub.elsevier.com/retrieve/pii/S0021999124004790},
	doi = {10.1016/j.jcp.2024.113230},
	language = {en},
	urldate = {2026-05-06},
	journal = {Journal of Computational Physics},
	author = {Cucchiara, Simona and Iollo, Angelo and Taddei, Tommaso and Telib, Haysam},
	month = oct,
	year = {2024},
	pages = {113230},
}

@misc{song_score-based_2020,
	title = {Score-{Based} {Generative} {Modeling} through {Stochastic} {Differential} {Equations}},
	copyright = {arXiv.org perpetual, non-exclusive license},
	url = {https://arxiv.org/abs/2011.13456},
	doi = {10.48550/ARXIV.2011.13456},
	urldate = {2026-05-06},
	publisher = {arXiv},
	author = {Song, Yang and Sohl-Dickstein, Jascha and Kingma, Diederik P. and Kumar, Abhishek and Ermon, Stefano and Poole, Ben},
	year = {2020},
	note = {Version Number: 2},
}

\end{document}